\documentclass{amsart}
\usepackage{amssymb}
\usepackage{colortbl}
\usepackage{hyperref}
\usepackage{thmtools,thm-restate}

\hypersetup{ colorlinks = true, urlcolor = blue, linkcolor = blue, citecolor = red }
\newtheorem{thm}{Theorem}[section]
\newtheorem{lem}[thm]{Lemma}
\newtheorem{cor}[thm]{Corollary}
\newtheorem{rem}[thm]{Remark}

\def\XXint#1#2#3{{\setbox0=\hbox{$#1{#2#3}{\int}$}
	\vcenter{\hbox{$#2#3$}}\kern-.5\wd0}}

\newcommand{ \mr }{ \mathbb{R} }

\newtheorem{defn}[thm]{Definition}

\numberwithin{equation}{section}
\makeatletter
\@namedef{subjclassname@2020}{\textup{2020} Mathematics Subject Classification}
\makeatother

\begin{document}
\title{A geometric result for composite materials with $C^{1,\gamma}$-boundaries}


\date{\today}


\author{Youchan Kim}
\email[Youchan Kim]{youchankim@uos.ac.kr}
\address[Youchan Kim]{Department of Mathematics, University of Seoul, Seoul 02504, Republic of Korea}
\author{Pilsoo Shin}
\email[Pilsoo Shin]{shinpilsoo.math@kgu.ac.kr}
\address[Pilsoo Shin]{Department of Mathematics, Kyonggi University, Suwon 16227, Republic of Korea}

\begin{abstract}
In this paper, we obtain a geometric result for composite materials related to elliptic and parabolic partial differential equations. In the classical papers \cite{LN_CPAM1,LV_ARMA1}, they assumed that for any scale and for any point there exists a coordinate system such that the boundaries of the individual components of a composite material locally become $C^{1,\gamma}$-graphs. We prove that if the individual components of a composite material are composed of $C^{1,\gamma}$-boundaries then such a coordinate system in \cite{LN_CPAM1,LV_ARMA1} exists, and therefore obtaining the gradient boundedness and the piecewise gradient H\"{o}lder continuity results for linear elliptic systems related to composite materials.
\end{abstract}

\maketitle

\section{Introduction}

In this paper, we study a geometric property of composite materials related to elliptic and parabolic partial differential equations. The classical results \cite{LN_CPAM1,LV_ARMA1} obtained gradient boundedness and gradient H\"{o}lder continuity of the weak solutions to linear elliptic equations and systems related to composite materials under the assumption that for any scale and for any point  there exists a coordinate system such that the boundaries of the individual components become $C^{1,\gamma}$-graphs. A natural question then arises, ``for a composite material, if the individual components are composed of $C^{1,\gamma}$-boundaries then this geometry satisfies the assumption in \cite{LN_CPAM1,LV_ARMA1}?". For this geometry, we prove that for any scale and for any point there exists a coordinate system such that the boundaries of the individual component become $C^{1,\gamma}$-graphs. 

A typical composite material $U \subset \mathbb{R}^{n}$ ($n \geq 2$) composed of $C^{1,\gamma}$-boundaries can be described as the following. Let $D_{m} \subset U$ be a connected component in $U$ and $U_{2}, \cdots, U_{m} \subset \mathbb{R}^{n}$ be the connected components contained in $D_{1}$. Without loss of generality, we may assume that $D_{1} \cup (U_{2} \cup \cdots \cup U_{m})$, $U_{2}$, $\cdots$, $U_{m}$ are open.  Then one can find an open set  $U_{1} \subset U$ satisfying $D_{1} = U_{1} \setminus (U_{2} \cup \cdots \cup U_{m})$. If $U_{1},U_{2},\cdots,U_{l}$ are $C^{1,\gamma}$-domains then we may say that $D_{1}$ is composed of $C^{1,\gamma}$-boundaries. Moreover, for any $U_{i}, U_{j} \in \{ U_{1}, \cdots, U_{m} \}$ one of the followings holds: (1) $U_{i} \subset U_{j}$ (2) $U_{j} \subset U_{i}$ (3) $U_{i} \cap U_{j} = \emptyset$. With the observation in this paragraph, we define $C^{1,\gamma}$-class composite domains in Definition \ref{intro_composite domain}. Before stating Definition \ref{intro_composite domain}, we introduce the following notations.

\begin{enumerate}
\item $x'=(x^{2},\cdots,x^{n}) \in \mr^{n-1}$ and $x =(x^{1},x') \in \mr \times \mr^{n-1}$. 
\item $B_{r}(y)=\{x \in\mathbb{R}^{n}: |x-y|< r\}$, $B'_{r}(y')=\{x' \in\mathbb{R}^{n-1}: |x'-y'|< r\}$ and 
      $Q_{r}(y)= (y^{1}-r, y^{1}+r) \times B_{r}'(y') $.
\item $B_{r} = B_{r}(\mathbf{0})$, $B_{r}'= B_{r}'(\mathbf{0}')$ and $Q_{r} = Q_{r}(\mathbf{0})$
\item $A = A_{1} \sqcup \cdots \sqcup A_{i}$ means that $A = A_{1} \cup \cdots \cup A_{i}$ and $A_{1}$, $\cdots$, $A_{i}$ are mutually disjoint.
\item $e_{i}$ denotes the vector with one $1$ in the $i$-th coordinate and $0$'s elsewhere, say $e_{1}=(1,0,\cdots,0)$, $\cdots$, $e_{n}=(0,\cdots,0,1)$.
\item For the matrix $M \in \mathbb{R}^{m \times n}$, let $M^{T}$ be the transpose of $M$. Also we denote $M =\left(\begin{array}{c} M_{1} \\ \vdots \\ M_{m} \end{array} \right) = \left(\begin{array}{ccc} M_{11} & \cdots & M_{1n} \\ \vdots & \ddots & \vdots \\ M_{m1} & \cdots & M_{mn} \end{array} \right) $ for $M_{i} \in \mathbb{R}^{1 \times n}$.
\end{enumerate}

\begin{restatable}{defn}{introdomain}
\label{intro_domain}
For $n \geq 2$, $U \subset \mathbb{R}^{n}$  is a $(C^{1,\gamma},R,\theta)$-domain, if for any $B_{R}(y)$ with $B_{R}(y) \cap \partial U \not = \emptyset$ there exists $x$-coordinate system and $C^{1,\gamma}$-function $\psi : B_{R}' \to \mathbb{R}$  such that
\begin{equation*}
U \cap B_{R}(y) = \{ x \in B_{R}: x^{1} > \psi(x') \}
\qquad \text{and} \qquad
\| \psi \|_{C^{1,\gamma}(B_{R}')} \leq \theta
\end{equation*}
where $y$ is the origin in the new $x$-coordinate system.
\end{restatable}

\begin{restatable}{defn}{introcompositedomain} \label{intro_composite domain}
For $n \geq 2$, $U \subset \mathbb{R}^{n}$ is a composite $(C^{1,\gamma},R,\theta)$-domain with subdomains $\{ U_{1}, \cdots, U_{K} \}$ if 

(a) $U_{0} := U$ and $\{ U_{1}, \cdots, U_{K} \}$ are $(C^{1,\gamma},R,\theta)$-domains

(b) one of the following holds for any $U_{i}$, $U_{j}$ $(i,j \in \{ 0,,\cdots, K \}, ~ i \not = j)$:
\begin{equation}\label{HHH6500}
(1) ~ U_{i} \subsetneq U_{j}
\qquad
(2) ~ U_{j} \subsetneq U_{i}
\qquad
(3) ~ U_{i} \cap U_{j} = \emptyset.
\end{equation}
\end{restatable}

\medskip

Then for the family of set $S = \{ U_{0} : =U, U_{1}, \cdots, U_{K} \}$,
\begin{equation}\label{intro_W_{j}}
W_{j} = U_{j} \setminus \left( \bigcup_{ U_{i} \in S, ~ U_{i} \subsetneq U_{j} } U_{i} \right)
\end{equation}
represents an individual component in $U_{j}$. Also $\bigcup\limits_{U_{i} \in S, U_{i} \subsetneq U_{j}} U_{i}$ represents the union of the components contained in $U_{j}$ except $W_{j}$. So the coefficient of an elliptic equation on composite material $U$ can be written by
\begin{equation*}
a_{ij} = \sum_{0 \leq k \leq K} a_{ij}^{k} \chi_{W_{k}}
\end{equation*}
where $a_{ij}^{k}$ represent the physical property of the material in component $W_{k}$. In this paper we prove that for any $B_{R}$, there exists a coordinate system such that $W_{j} \cap B_{R}$ ($j \in \{ 0,\cdots, K \} $) can be described by $C^{1,\gamma}$-functions as in the following theorems. We first state the interior result. 

\begin{restatable}{thm}{maintheoreminterior}\label{main theorem1}
For any $\tau \in (0,1]$, one can find  $R_{0} = R_{0}(n,\gamma,\theta,\tau) \in (0,1]$ so that the following holds for any $R \in (0,R_{0}]$. Suppose that $U \supset B_{R}$ is a composite $(C^{1,\gamma},80R,\theta)$-domain with subdomains $\{ U_{1}, \cdots, U_{K} \}$. Let
\begin{equation*}
S = \{ U_{0} : =U, U_{1},\cdots, U_{K},U_{K+1} : =\emptyset \}
\end{equation*}
and
\begin{equation*}
W_{j} = U_{j} \setminus \left( \bigcup_{U_{i} \in S, ~ U_{i} \subsetneq U_{j} } U_{i} \right)
\qquad (j \in \{ 0, \cdots, K \} ).
\end{equation*}
Then there exist $y$-coordinate system and $C^{1,\gamma}$-functions  $\varphi_{-m} , \cdots, \varphi_{l+1} : B_{R}' \to \mathbb{R}$ $(l,m \geq 0)$ satisfying
\begin{equation}\label{MMM5300}
U \cap B_{R} =  ( W_{i_{-m}} \sqcup \cdots \sqcup W_{i_{l}} ) \cap B_{R}
\end{equation}
and 
\begin{equation}\begin{aligned}\label{MMM5400}
& \{ (y^{1},y') \in B_{R} : \varphi_{d}(y') < y^{1} < \varphi_{d+1}(y')  \} \\
& \qquad \subset W_{i_{d}} \cap B_{R} \\
& \qquad \subset \{ (y^{1},y') \in B_{R} : \varphi_{d}(y') \leq y^{1} \leq \varphi_{d+1}(y')  \}
\end{aligned}\end{equation}
with the estimates
\begin{equation}\label{MMM5500}
\|D_{y'} \varphi_{d} \|_{L^{\infty}(B_{R}')} \leq  \tau
\quad \text{and} \quad
[D_{y'} \varphi_{d}]_{C^{\gamma}(B_{R}')} \leq 288 n \theta
\end{equation}
for any $d \in \{ -m, \cdots, l \}$ where $\varphi_{l+1} \equiv R$ and $\varphi_{-m} \equiv -R$. Moreover, if $B_{R} \not \subset W_{0}$ then 
\begin{equation}\label{MMM5600}
\mathbf{0} \in W_{k},
\qquad
(\varphi_{k}(0'),0') \in B_{R}
\qquad \text{and} \qquad D_{y'}\varphi_{k}(0')=0'
\end{equation}
for some $k \in \{ -m,\cdots, l \}$.
\end{restatable}

We next state the boundary result. We remark that the general case that $\mathbf{0} \not \in \partial U$ can be handled by the case that $\mathbf{0} \in \partial U$ by taking a sufficiently larger ball.

\begin{restatable}{thm}{maintheoremboundary}\label{main theorem2}
For any $\tau \in (0,1]$, one can find  $R_{0} = R_{0}(n,\gamma,\theta,\tau) \in (0,1]$ so that the following holds for any $R \in (0,R_{0}]$. Suppose that $U$ is a composite $(C^{1,\gamma},80R,\theta)$-domain with subdomains $\{ U_{1}, \cdots, U_{K} \}$ and $\mathbf{0} \in \partial U$. Let
\begin{equation*}
S = \{ U_{0} : =U, U_{1},\cdots, U_{K},U_{K+1} : =\emptyset \}
\end{equation*}
and
\begin{equation*}
W_{j} = U_{j} \setminus \left( \bigcup_{U_{i} \in S, ~ U_{i} \subsetneq U_{j} } U_{i} \right)
\qquad (j \in \{ 0, \cdots, K\}).
\end{equation*}
Then there exist $y$-coordinate system and $C^{1,\gamma}$-functions  $\varphi_{0} , \cdots, \varphi_{l+1} : B_{R}' \to \mathbb{R}$ $(l \geq 0)$ satisfying
\begin{equation}\label{NNN5300}
U \cap B_{R} =  ( W_{i_{0}} \sqcup \cdots \sqcup W_{i_{l}} ) \cap B_{R}
\end{equation}
and 
\begin{equation}\begin{aligned}\label{NNN5400}
& \{ (y^{1},y') \in B_{R} : \varphi_{d}(y') < y^{1} < \varphi_{d+1}(y')  \} \\
& \qquad \subset W_{i_{d}} \cap B_{R} \\
& \qquad \subset \{ (y^{1},y') \in B_{R} : \varphi_{d}(y') \leq y^{1} \leq \varphi_{d+1}(y')  \}
\end{aligned}\end{equation}
with the estimates
\begin{equation}\label{NNN5500}
\|D_{y'} \varphi_{d} \|_{L^{\infty}(B_{R}')} \leq  \tau
\quad \text{and} \quad
[D_{y'} \varphi_{d}]_{C^{\gamma}(B_{R}')} \leq 288 n \theta
\end{equation}
for any $d \in \{ 0, \cdots, l \}$ where $\varphi_{l+1} \equiv R$. Moreover, we have that
\begin{equation}\label{NNN5600}
\varphi_{0}(0')= 0
\qquad \text{and} \qquad D_{y'}\varphi_{0}(0')=0'.
\end{equation}
\end{restatable}

\bigskip

By Theorem \ref{main theorem1} and Theorem \ref{main theorem2}, for the coefficients  $\displaystyle a_{ij} = \sum_{ 0 \leq k \leq K } a_{ij}^{k} \chi_{W_{k}}$, we have that
\begin{equation*}
a_{ij} = \sum_{ -m-k \leq d \leq l-k } a_{ij}^{i_{d+k}} \chi_{ \varphi_{d}(y') < y^{1} \leq \varphi_{d+1}(y') } 
\text{ a.e. in } U \cap B_{R}.
\end{equation*}
So one can use Theorem \ref{main theorem1} and \cite[Theorem 1.1]{LN_CPAM1} to obtain the following lemma.

\begin{cor}
For any $R \in (0,R_{0}]$ with  $R_{0} = R_{0}(n,\gamma,\theta,1) \in (0,1]$ in Theorem \ref{main theorem1}, let $U$ be composite $(C^{1,\gamma},80R,\theta)$-domain with subdomains $\{ U_{1}, \cdots, U_{K} \}$. Assume that $A_{ij}^{\alpha \beta} : \mathbb{R}^{n} \to \mathbb{R}$ 
$(1 \leq i,j \leq N$, $1 \leq \alpha, \beta \leq n)$ satisfy
\begin{equation*}
A_{ij}^{\alpha \beta}(x)
\xi_{\alpha} \xi_{\beta}
\eta^{i} \eta^{j}
\geq  \lambda |\xi|^{2} |\eta|^{2}
\quad (x \in \mathbb{R}^{n},\xi \in \mathbb{R}^{n},
\eta \in \mathbb{R}^{N})
\quad \text{and} \quad
\| A_{ij}^{\alpha \beta} \|_{L^{\infty}(\mathbb{R}^{n})} \leq \Lambda.
\end{equation*}
Then for any $\epsilon >0$ and a weak solution $u$ of
\begin{equation*}
\partial_{\alpha} (A_{ij}^{\alpha \beta} (x) \partial_{\beta} u^{j}) = h_{i} + \partial_{\beta} g_{i}^{\beta}
\end{equation*}
with $A_{ij}^{\alpha \beta} \in C^{\mu}(\overline{W_{k}})$, $h_{i} \in L^{\infty}(U)$, $g_{i}^{\beta} \in C^{\mu}(\overline{W_{k})}$ 
$(i = 1,\cdots, N, ~ k = 0,\cdots,K)$ and 
\begin{equation*}
\gamma' \in \left( 0, \min \left\{ \mu, \frac{\gamma}{2(\gamma+1)} \right\} \right],
\end{equation*}
there exists $c = c(n, N , K, \lambda, \Lambda, \mu, \gamma, \epsilon, \| A \|_{C^{\gamma'}(\overline{W_{i}})}, \theta) $ such that
\begin{equation*}
\sum_{0 \leq k \leq K}
\| u \|_{C^{1,\gamma'}(\overline{W_{k}} \cap U_{\epsilon})} 
\leq c \left( 
\| u \|_{L^{2}(U)}
+ \| h \|_{L^{\infty}(U)}
+ \sum_{0 \leq k \leq K}
\| g \|_{C^{\gamma'}(\overline{W_{k}})} \right),
\end{equation*}
and
\begin{equation*}
\| Du \|_{L^{\infty}(U_{\epsilon})} 
\leq c \left( 
\| u \|_{L^{2}(U)}
+ \| h \|_{L^{\infty}(U)}
+ \sum_{0 \leq k \leq K}
\| g \|_{C^{\gamma'}(\overline{W_{k}})} \right),
\end{equation*}
where 
\begin{equation*}
U_{\epsilon}
= \{ x \in U : \mathrm{dist ~} (x,\partial U) > \epsilon \}.
\end{equation*}
\end{cor}

\medskip

We explain the main idea of this paper. For $C^{1,\gamma}$-domain $U_{1}$, there exists $y$-coordinate system such that $U_{1} \cap B_{R} = \{ y \in B_{R} : y^{1} > \psi(y') \}$ for some $C^{1,\gamma}$-function $\psi: B_{R}' \to \mathbb{R}$. Then for $C^{1,\gamma}$-domain $U_{2}$ with $U_{1} \cap U_{2} = \emptyset$, by \cite[Lemma 2.4]{JK_CV1}, the normals on $\partial U_{1}$ and $\partial U_{2}$ are almost opposite in $B_{R}$ if $R>0$ is sufficiently small. So by using the implicit function theorem, one can find $C^{1,\gamma}$-function $\varphi: B_{R}' \to \mathbb{R}$ such that $U_{2} \cap B_{R} = \{ y \in B_{R} : y^{1} < \varphi(y') \}$. By Definition \ref{intro_composite domain}, one of the followings holds for composite $(C^{1,\gamma},R,\theta)$-domains : (1) $U_{i} \subsetneq U_{j}$, (2) $U_{j} \subsetneq U_{i}$, (3) $U_{i} \cap U_{j} = \emptyset$. So we will prove Theorem \ref{main theorem1} and Theorem \ref{main theorem2} by applying the argument in this paragraph to $\mathbb{R}^{n} \setminus \overline{U_{j}}$ and $U_{i}$ when (1) holds, to $\mathbb{R}^{n} \setminus \overline{U_{i}}$ and $U_{j}$ when (2) holds, and to $U_{i}$ and $U_{j}$ when (3) holds instead of $U_{1}$ and $U_{2}$. 

Elliptic and parabolic equations in composite materials have been studied by many researchers, see for instance \cite{BASLD1,BV_SIAM1}. For piecewise gradient boundedness and gradient H\"{o}lder continuity for linear elliptic equation and systems related to composite $C^{1,\gamma}$-domain, we refer to \cite{LN_CPAM1,LV_ARMA1}. For Lipschitz regularity results for linear laminates, we refer to \cite{BK_AM1,BK_IMRN1,CKV1,D_ARMA1,DK_SIAM1}. Partial regularity result for monotone systems had been obtained in \cite{NS_arxiv1}. Gradient $L^{p}$-estimates for composite materials had been considered in \cite{JK_CV1,JK_MMAS,HLW,U_JD1,Z_NA1}. Also we refer to \cite{KLY_MA1,KL_MA1} for the blow up phenomena when the coefficients of an elliptic equation have critical values.

\section{Coordinate system for graph functions} \label{Coordinate system for graph functions}

The implicit function theorem give the existence of the graph function in a local neighborhood, but does not give an information about the precise size of that neighborhood. To control the size of that neighborhood, we derive Lemma \ref{AAA1000} and Lemma \ref{BBB6000}.

Let $S = \{ (\psi(x'),x') : x' \in B_{8R}' \}$ be a graph in $x$-coordinate system. Let $y$-coordinate system has the orthonormal basis $\{ W_{1} , \cdots, W_{n} \}$ and $O = ( W_{1}^{T}  \cdots W_{n}^{T} )$. We prove in Lemma \ref{AAA1000} that if the implicit function theorem can be applied to $zO$ for any $z \in S$ with respect to $y^{1}$-variable which gives a local existence of a graph function with respect to $y^{1}$-variable on $(zO)'$, then one can find a graph function defined in $B_{R}'$ with respect to $y^{1}$-variable. In Lemma \ref{AAA1000}, a point $z \in S$ in $x$-coordinate system is represented by $zO$ with respect to $y$-coordinate system.

\begin{lem}\label{AAA1000}
Let $\tau \in (0,1]$, $O \in \mathbb{R}^{n \times n}$ be an orthonormal matrix with $\det O >0$ and $C^{1}$-function $\psi : B_{8R}' \to \mathbb{R}$ be a given function. Then for the graph
\begin{equation}\label{AAA1100}
S = \{ (\psi(x'),x')  : x' \in B_{8R}' \} \subset \mathbb{R}^{n}
\end{equation}
assume that 
\begin{equation}\label{AAA1220}
S \cap B_{R} \not = \emptyset
\qquad \text{ and } \qquad 
O_{11} - \sum_{2 \leq j \leq n} O_{j1} D_{x^{j}} \psi \not = 0
\text{ in } B_{8 R}'.
\end{equation}
Also further assume that for any $z \in S$, there exist ball $U_{z} \subset  \mathbb{R}^{n-1}$ and $C^{1}$-function $\varphi_{z} : U_{z} \to \mathbb{R}$ satisfying 
\begin{equation}\label{AAA1240}
(zO)' \in U_{z},
\qquad 
(zO)^{1} = \varphi_{z} \big( (zO)' \big),
\qquad 
\|D_{y'} \varphi_{z}\|_{L^{\infty}(U_{z})} \leq \tau,
\end{equation}
\begin{equation}\label{AAA1260}
\big( \varphi_{z}(y'),y' \big) \cdot O_{1} - \psi \big( \big( \varphi_{z}(y'),y' \big) \cdot O_{2}, \cdots , \big( \varphi_{z}(y'),y' \big) \cdot O_{n} \big) =0
\qquad (y' \in U_{z}),
\end{equation}
and
\begin{equation}\label{AAA1250}
\big\{ \big( \varphi_{z}(y'),y' \big) : y' \in U_{z} \big\} \subset B_{8R}.
\end{equation} 
Then there exists $C^{1}$-function $\varphi : B_{R}' \to \mathbb{R}$ such that
\begin{equation}\label{AAA1270}
\big( \varphi(y'),y' \big) \cdot O_{1} - \psi \big( \big( \varphi(y'),y' \big) \cdot O_{2}, \cdots , \big( \varphi(y'),y' \big) \cdot O_{n} \big) =0
\qquad (y' \in B_{R}'),
\end{equation}
\begin{equation}\label{AAA1275}
(zO)^{1} = \varphi \big( (zO)' \big)
\qquad (z \in S \cap B_{R}),
\end{equation}
\begin{equation}\label{AAA1280}
\| D_{y'}\varphi \|_{L^{\infty}(B_{R}')} 
\leq \tau,
\end{equation}
and
\begin{equation}\label{AAA1290}
\big\{ \big( \varphi(y'),y' \big) : y' \in B_{R}' \big\} \subset B_{8R}.
\end{equation}
\end{lem}

\begin{proof}
Let $U = \cup_{z \in S} U_{z}$. For any $z' \in U$, we define $\phi_{z'} : B'_{\rho_{z'}}(z') \to \mathbb{R}$ in the following way. Since $U = \cup_{z \in S} U_{z}$, there exists $\tilde{z} \in S$ and $B'_{\rho_{z'}}(z')$ such that $ B'_{\rho_{z'}}(z') \subset U_{\tilde{z}}$. Then we set
\begin{equation}\label{AAA1320}
\phi_{z'} = \varphi_{\tilde{z}}
\quad \text{in} \quad B'_{\rho_{z'}}(z').
\end{equation}
So by \eqref{AAA1240}, \eqref{AAA1260} and \eqref{AAA1250}, we have that
\begin{equation}\label{AAA1340}
\|D_{y'} \phi_{z'}\|_{L^{\infty}(B'_{\rho_{z'}}(z'))} 
\leq \| D_{y'} \varphi_{\tilde{z}} \|_{L^{\infty}(U_{\tilde{z}})} 
\leq \tau,
\end{equation}
\begin{equation}\label{AAA1360}
\big( \phi_{z'}(y'),y' \big) \cdot O_{1} - \psi \big( \big( \phi_{z'}(y'),y' \big) \cdot O_{2}, \cdots , \big( \phi_{z'}(y'),y' \big) \cdot O_{n} \big) =0
~ \text{ in } ~
B'_{\rho_{z'}}(z'),
\end{equation}
and
\begin{equation}\label{AAA1350}
\big\{ \big( \phi_{z'}(y'),y' \big) : y' \in B'_{\rho_{z'}}(z') \big\} \subset B_{8R}.
\end{equation} 
We claim that for $z', \tilde{z}'\in U$, if $\phi_{z'} : B'_{\rho_{z'}}(z') \to \mathbb{R}$ and $\phi_{\tilde{z}'} : B'_{\rho_{\tilde{z}'}}(\tilde{z}') \to \mathbb{R}$ satisfy \eqref{AAA1360} for $z'$ and $\tilde{z}'$ respectively then
\begin{equation}\label{AAA3000}
\phi_{z'}(y')  = \phi_{\tilde{z}'}(y') \text{ for any } y' \in B'_{\rho_{z'}}(z') \cap B'_{\rho_{\tilde{z}'}}(\tilde{z}').
\end{equation}
Suppose not. Then by \eqref{AAA1360}, there exists $z', \tilde{z}' \in l_{m}$ such that
\begin{equation}\label{AAA3100}
\phi_{z'}(y') \not = \phi_{\tilde{z}'}(y') 
\text{ for some } y' \in B'_{\rho_{z'}}(z') \cap B'_{\rho_{\tilde{z}'}}(\tilde{z}')
\end{equation}
and
\begin{equation}\begin{aligned}\label{AAA3200}
& \big( \phi_{\tilde{z}'}(y'),y' \big) \cdot O_{1} 
- \big( \phi_{z'}(y'),y' \big) \cdot O_{1} \\
& \quad = \psi \big( \big( \phi_{\tilde{z}'}(y'),y' \big) \cdot O_{2}, \cdots , \big( \phi_{\tilde{z}'}(y'),y' \big) \cdot O_{n} \big) \\
& \qquad - \psi \big( \big( \phi_{z'}(y'),y' \big) \cdot O_{2}, \cdots , \big( \phi_{z'}(y'),y' \big) \cdot O_{n} \big).
\end{aligned}\end{equation}
By a direct calculation,
\begin{equation*}\begin{aligned}
\Big( \big( \phi_{z'}(y'),y' \big) \cdot O_{2},
\cdots,
\big( \phi_{z'}(y'),y' \big)  \cdot O_{n} \Big) 
= \Big( \big( \phi_{z'}(y'),y' \big) O^{T} \Big)',
\end{aligned}\end{equation*}
and
\begin{equation*}\begin{aligned}
\Big( \big( \phi_{\tilde{z}'}(y'),y' \big) \cdot O_{2},
\cdots,
\big( \phi_{\tilde{z}'}(y'),y' \big)  \cdot O_{n} \Big) 
= \Big( \big( \phi_{\tilde{z}'}(y'),y' \big) O^{T} \Big)'.
\end{aligned}\end{equation*}
So we find from \eqref{AAA3200} that
\begin{equation*}\begin{aligned}
& \big( \phi_{\tilde{z}'}(y') - \phi_{z'}(y') \big) \cdot O_{11} \\
& \quad = \int_{0}^{1} \sum_{2 \leq k \leq n} D_{x^{k}} \psi \Big( t \big( \big( \phi_{\tilde{z}'}(y'),y' \big) O^{T} \big)' + (1-t) \big( \big( \phi_{z'}(y'),y' \big) O^{T} \big)' \Big) \, dt
\cdot \big( \phi_{\tilde{z}'}(y') - \phi_{z'}(y') \big)
O_{k1},
\end{aligned}\end{equation*}
which implies
\begin{equation*}\begin{aligned}
& \big( \phi_{\tilde{z}'}(y') - \phi_{z'}(y') \big) \cdot \left( O_{11} - \sum_{2 \leq j \leq n} \int_{0}^{1} O_{j1} D_{x^{j}} \psi \Big( t  \big( \big( \phi_{\tilde{z}'}(y'),y' \big) O^{T} \big)' + (1-t) \big( \big( \phi_{z'}(y'),y' \big) O^{T} \big)' \Big) \, dt \right) =0.
\end{aligned}\end{equation*}
In view of \eqref{AAA1350}, the fact that $z',\tilde{z}' \in B_{R}'$ gives that for any $ t \in [0,1]$, we have that $(1-t) \big( \big( \phi_{\tilde{z}'}(y'),y' \big) O^{T} \big)' + t \big( \big( \phi_{z'}(y'),y' \big) O^{T} \big)' \in B_{8R}'$. So it follows from \eqref{AAA1220} that
\begin{equation*}\begin{aligned}
\int_{0}^{1}  \left( O_{11} - \sum_{2 \leq j \leq n} O_{j1} D_{x^{j}} \psi \Big( t  \big( \big( \phi_{\tilde{z}'}(y'),y' \big) O^{T} \big)' + (1-t) \big( \big( \phi_{z'}(y'),y' \big) O^{T} \big)' \Big) \right) \, dt \not = 0,
\end{aligned}\end{equation*}
and we obtain $\phi_{\tilde{z}'}(y') = \phi_{z'}(y')$, 
which contradicts \eqref{AAA3100}. So \eqref{AAA3000} follows. 

Next, we claim that 
\begin{equation}\label{AAA1500}
\overline{B_{R}'} \subset U.
\end{equation}
Suppose not. Then $\overline{B_{R}'} \not \subset U$ and there exists a point $p'$ such that
\begin{equation}\label{AAA1600}
p' \in  U^{c} \cap \overline{B_{R}'}.
\end{equation}
Next, we will choose a point $q'$ in $U \cap \overline{B_{R}'}$ and a point $w'$ on $\partial U \cap \overline{B_{R}'}$. By \eqref{AAA1220}, there exists $\bar{z} \in S \cap B_{R}$. From the fact that $\bar{z} \in B_{R}$,
\begin{equation*}
(\bar{z} O)' \in B_{R}'.
\end{equation*}
Also from the fact that $\bar{z} \in S$,  
\begin{equation*}
(\bar{z} O)' \in U_{ \bar{z} O} \subset U.
\end{equation*}
So by combining the above two inclusions,
\begin{equation}\label{AAA1700}
q' : = (\bar{z}  O)' \in U \cap B_{R}'.
\end{equation}
Let $\mathit{l} : [0,1] \to \overline{B_{R}'}$ be the line connecting from $q'$ to $p'$. Then from \eqref{AAA1600} and \eqref{AAA1700}, we find that $\mathit{l} \cap \partial U \not = \emptyset$. Let
\begin{equation}\label{AAA1800}
w' \in \overline{B_{R}'} \cap \partial U
\end{equation}
be the point which the line $\mathit{l}$ from $q'$ to $p'$ first meets $\partial U$. 

\medskip

To prove \eqref{AAA1500}, we will choose $w^{1} \in \mathbb{R}$ so that
\begin{equation}\label{AAA2000}
(w^{1},w') O^{T}  \in S.
\end{equation} 
If \eqref{AAA2000} holds, then by the fact that $(zO)' \in U_{z} \subset U$ for any $z \in S$, we have that 
\begin{equation*}
w' = \left( (w^{1},w') O^{T} O \right)' \in U_{(w^{1},w') O^{T}} \subset U,
\end{equation*}
and a contradiction occurs from \eqref{AAA1800}. So the claim \eqref{AAA1500} follows. To prove \eqref{AAA1500}, we will show \eqref{AAA2000}.

\medskip

Since $w' \in \overline{U} \cap \overline{B_{R}'} = \overline{ U \cap B_{R}'}$, one can choose a sequence $\{ w_{m}' \}_{m=1}^{\infty}$ so that 
\begin{equation}\label{AAA1900}
w_{m}' \in U \cap B_{R}'
\qquad \text{and} \qquad w_{m}' \to w'.
\end{equation}
For a fixed $w_{m}' \in U \cap B_{R}'$, let $\mathit{l}_{m} \subset U \cap B_{R}'$ be the line connecting from $q'$ to $w_{m}'$. Then for any point $z' \in \mathit{l}_{m} \subset U$, by \eqref{AAA1340} and \eqref{AAA1360}, there exist  $B'_{\rho_{z'}}(z') \subset U$ and $C^{1}$-function $\phi_{z'} : B'_{\rho_{z'}}(z') \to \mathbb{R}$ such that 
\begin{equation}\label{AAA2300}
\big( \phi_{z'}(y'),y' \big) \cdot O_{1} - \psi \big( \big( \phi_{z'}(y'),y' \big) \cdot O_{2}, \cdots , \big( \phi_{z'}(y'),y' \big) \cdot O_{n} \big) =0 
\text{ in } B'_{\rho_{z'}}(z')
\end{equation}
and
\begin{equation}\label{AAA2400}
\| D_{y'} \phi_{z'} \|_{L^{\infty}(B'_{\rho_{z'}}(z'))} \leq \tau.
\end{equation}
Since $z' \in l_{m}$ was arbitrary chosen in \eqref{AAA2300} and $l_{m}$ is compact, by using partition of unity, one can choose functions $\eta_{i} : B'_{\rho_{z_{i}'}}(z_{i}') \to \mathbb{R}$ for $i \in [1,I]$ so that 
\begin{equation*}
\eta_{i} \in C_{c}^{\infty}(B'_{\rho_{z_{i}'}}(z_{i}')),
\quad
0 \leq \eta_{i} \leq 1
\quad \text{and} \quad 
\sum_{1 \leq i \leq I} \eta_{i} = 1
\text{ on } l_{m},
\end{equation*}
where $\phi_{z_{i}'}$ and $B'_{\rho_{z_{i}'}}(z_{i}')$  in \eqref{AAA2300} chosen for $z_{i}'$ instead of $z'$. So we take 
\begin{equation}\label{AAA2500}
\phi_{m} =  \sum_{1 \leq i \leq I} \phi_{z_{i}'} \eta_{i} 
\quad \text{on} \quad l_{m}.
\end{equation}

By \eqref{AAA2300}  and \eqref{AAA3000},
\begin{equation}\label{AAA2650}
\phi_{m} = \phi_{z_{i}'}
\quad \text{in} \quad B'_{\rho_{z_{i}'}}(z_{i}').
\end{equation}
Since $w_{m}' \in l_{m}$, we have  that
\begin{equation*}
\big( \phi_{w_{m}'}(w_{m}'),w_{m}' \big) \cdot O_{1} - \psi \big( \big( \phi_{w_{m}'}(w_{m}'),w_{m}' \big) \cdot O_{2}, \cdots , \big( \phi_{w_{m}'}(w_{m}'),w_{m}' \big) \cdot O_{n} \big) =0.
\end{equation*}
So it follows from \eqref{AAA2650}  that
\begin{equation}\label{AAA2700}
(\phi_{m} (w_{m}'),w_{m}') \cdot O_{1}  = \psi \big( (\phi_{m} (w_{m}'),w_{m}') \cdot O_{2}, \cdots ,  (\phi_{m} (w_{m}'),w_{m}') \cdot O_{n} \big).
\end{equation}
By \eqref{AAA1700}, we have that $\bar{z} \in S$ and $(\bar{z}O)' = q' \in l_{m}$. So by \eqref{AAA1320}, \eqref{AAA2500} and  \eqref{AAA3000},
\begin{equation}\label{AAA2600}
(\bar{z} O)^{1} 
= \varphi_{\bar{z}} \left( (\bar{z}O)' \right)
= \phi_{q'} \left( (\bar{z}O)' \right)
= \phi_{m}\big( (\bar{z} O)' \big) ,
\end{equation}
From \eqref{AAA2400}, \eqref{AAA3000} and the fact that $\sum_{1 \leq i \leq I} \eta_{i} = 1$, we find that
\begin{equation*}
\| D_{y'}\phi_{m} \|_{L^{\infty}(l_{m})} 
= \left\| \sum_{1 \leq i \leq I} (D_{y'}\phi_{z_{i}'}) \eta_{i} \right\|_{L^{\infty}(l_{m})}
\leq \tau 
\leq 1.
\end{equation*}
So from \eqref{AAA2600} and the fact that $|\bar{z}|, |q'|, |w_{m}'| < R$, we have
\begin{equation*}\begin{aligned}
|\phi_{m}(w_{m}')| 
& \leq |(\bar{z} O)^{1}| + |(\bar{z} O)^{1} - \phi_{m}(w_{m}')| \\
& = |(\bar{z} O)^{1}| + |\phi_{m} (q') -  \phi_{m}(w_{m}')| \leq R + 4R 
\leq 5R.
\end{aligned}\end{equation*}
Thus by Bolzano-Weierstrass theorem, there exists a subsequence of $\{ w_{m}' \}_{m=1}^{\infty}$ which still denoted by $\{ w_{m}' \}_{m=1}^{\infty}$ such that
\begin{equation}\label{AAA4000}
\phi_{m}(w_{m}') \xrightarrow{m \to \infty} w^{1}
\qquad \text{and} \qquad |w^{1}| \leq 5R.
\end{equation}
It follows from \eqref{AAA1900} and \eqref{AAA4000} that
\begin{equation}\label{AAA4100}
\big( \phi_{m}(w_{m}'), w_{m}' \big) \xrightarrow{m \to \infty} (w^{1},w') = w \in B_{8R}.
\end{equation}
From \eqref{AAA2700} and \eqref{AAA4100}, the continuity of $\psi$ gives that
\begin{equation*}
0 = w \cdot O_{1} - \psi \big( w \cdot O_{2}, \cdots, w \cdot O_{n} \big).
\end{equation*}
Since $( w \cdot O_{1}, \cdots, w \cdot O_{n}) = wO^{T} $, it follows that $0 = (wO^{T})^{1} - \psi \big( (wO^{T})' \big)$. So by the definition of $S$ and the fact that $(wO^{T})' \in B_{8R}'$,
\begin{equation*}
wO^{T}  = \left( (wO^{T})^{1}, (wO^{T})' \right)
= \left( \psi \big( (wO^{T})' \big) , (wO^{T})' \right) \in S,
\end{equation*}
and by the assumption of the lemma, we have
\begin{equation*}
w' = \big( (wO^{T}) O \big)' \in U_{wO^{T}} \subset U,
\end{equation*}
and we have a contradiction from \eqref{AAA1800}.  So the claim \eqref{AAA1500} holds.

We next show that \eqref{AAA1270}, \eqref{AAA1275}, \eqref{AAA1280} and \eqref{AAA1290}.  Since $\overline{B_{R}'}$ is compact, by using partition of unity and \eqref{AAA1500}, one can choose finite number of functions $\eta_{i} \in C_{c}^{\infty}(B'_{\rho_{z_{i}'}}(z_{i}'))$ for $1 \leq i \leq I$ such that $0 \leq \eta_{i} \leq 1$ and $\sum_{1 \leq i \leq I} \eta_{i} = 1$ in $\overline{B_{R}'}$. So we take 
\begin{equation}\label{AAA5000}
\varphi =  \sum_{1 \leq i \leq I} \phi_{z_{i}'} \eta_{i} 
\quad \text{ in } \quad \overline{B_{R}'}.
\end{equation}
It follows from \eqref{AAA3000} and \eqref{AAA5000} that
\begin{equation}\label{AAA5100}
\varphi = \phi_{z_{i}'} 
\quad \text{ in } \quad
B'_{\rho_{z_{i}'}}(z_{i}').
\end{equation}
So by \eqref{AAA1340}, \eqref{AAA1360} and the fact that $\{ B'_{\rho_{z_{i}'}}(z_{i}') \}_{1 \leq i \leq I}$ is a covering of $\overline{B_{R}'}$,
\begin{equation*}
\big( \varphi(y'),y' \big) \cdot O_{1} - \psi \big( ( \varphi(y'),y') \cdot O_{2}, \cdots , ( \varphi(y'),y' ) \cdot O_{n} \big) =0
\qquad (y' \in \overline{B_{R}'}),
\end{equation*}
and
\begin{equation*}
\| D_{y'}\varphi \|_{L^{\infty}(\overline{B_{R}'})} 
= \left\| \sum_{1 \leq i \leq I} (D\phi_{z_{i}'}) \eta_{i} \right\|_{L^{\infty}(\overline{B_{R}'})}
\leq \tau.
\end{equation*}
Thus \eqref{AAA1270} and \eqref{AAA1280} follows. So it only remains to prove \eqref{AAA1275} and \eqref{AAA1290}. Since $\sum\limits_{1 \leq i \leq I} \eta_{i} = 1$ in $\overline{B_{R}'}$, for any $\tilde{z} \in S \cap B_{R}$, there exists $z_{i}'$ with $(\tilde{z}O)' \in  B'_{\rho_{z_{i}'}}(z_{i}')$. So \eqref{AAA1320}, \eqref{AAA3000} and \eqref{AAA5100} imply that 
\begin{equation*}
(\tilde{z}O)^{1} = \varphi_{\tilde{z}} \big( (\tilde{z}O)' \big) = \phi_{(\tilde{z}O)'} \big( (\tilde{z}O)' \big)
= \phi_{z_{i}'} \big( (\tilde{z}O)' \big)
= \varphi \big( (\tilde{z}O)' \big),
\end{equation*}
and we obtain \eqref{AAA1275}. Since $\sum\limits_{1 \leq i \leq I} \eta_{i} = 1$ in $\overline{B_{R}'}$, $0 \leq \eta_{i} \leq 1$, and $\eta_{i} \in C_{c}^{\infty}(B'_{\rho_{z_{i}'}}(z_{i}'))$, we have from \eqref{AAA1350} and \eqref{AAA5000} that 
\begin{equation*}\begin{aligned}
\big\{ \big( \varphi(y'),y' \big) : y' \in B_{R}' \big\}
& = \left\{ \sum_{1 \leq i \leq I} \left( \phi_{z_{i}'}(y'),y' \right) \eta_{i}(y') : y' \in B_{R}' \right\} \subset B'_{8R}.
\end{aligned}\end{equation*} 
and \eqref{AAA1290} follows. 
\end{proof}

In Lemma \ref{BBB6000}, we prove that for an individual component in a neighborhood, there exists a coordinate system such that the boundary becomes almost flat graph.

\begin{lem}\label{BBB6000}
Let $\psi : B_{8R}' \to \mathbb{R}$ be $C^{1,\gamma}$-function with 
\begin{equation}\label{BBB6100}
S = \{ ( \psi(x'),x') \in B_{8R} : x \in B_{8R}' \}.
\end{equation}
Assume that
\begin{equation}\label{BBB6200}
S \cap B_{R} \not = \emptyset
\quad \text{and} \quad
n [ D_{x'}\psi]_{C^{\gamma}(B_{8R}')} (16R)^{\gamma} \leq 1/4.
\end{equation}
Then there exist an orthonormal matrix $O \in \mathbb{R}^{n \times n}$ with $\det O >0$ and $C^{1,\gamma}$-function $\varphi : B_{R}' \to \mathbb{R}$ such that 
\begin{equation}\label{BBB6300}
\big( \varphi(y'),y' \big) \cdot O_{1} - \psi \big( \big( \varphi(y'),y' \big) \cdot O_{2}, \cdots , \big( \varphi(y'),y' \big) \cdot O_{n} \big) =0
\qquad (y' \in B_{R}'),
\end{equation}
\begin{equation}\label{BBB6380}
(zO)^{1} = \varphi \big( (zO)' \big) 
\qquad (z \in S \cap B_{R}),
\end{equation}
\begin{equation}\label{BBB6350}
O_{1} \cdot e_{1} >0,
\end{equation}
and
\begin{equation}\label{BBB6360}
\big\{ \big( \varphi(y'),y' \big) : y' \in B_{R}' \big\} \subset B_{8R},
\end{equation} 
with the estimate
\begin{equation}\label{BBB6370}
|\varphi(0')| < R,
\qquad 
D_{y'}\varphi(0')=0',
\qquad \| D_{y'}\varphi \|_{L^{\infty}(B_{R}')} 
\leq 2\sqrt{n}
\end{equation}
and
\begin{equation}\label{BBB6400}
[D_{y'}\varphi]_{C^{\gamma}(B_{R}')} 
\leq \frac{ 18n [D_{y'}\psi]_{C^{\gamma}(B_{8R}')} }{ \| (-1,D_{y'} \psi ) \|_{L^{\infty}(B_{8R}')} }.
\end{equation}
\end{lem}

\begin{proof}
By the assumption in \eqref{BBB6200}, there exists 
\begin{equation}\label{BBB6500}
\text{a point } \bar{x} = (\psi(\bar{x}'),x') \in B_{R}
\text{ with }
|\bar{x}| = \min_{x' \in B_{R}'} |(\psi(x'), x')|. 
\end{equation}
To prove the lemma, we use the implicit function theorem.

\medskip

\noindent [Step 1: Choosing the new $y$-coordinate system] We define a vector 
\begin{equation}\label{BBB6530}\begin{aligned}
V_{1} = (V_{11}, \cdots, V_{1n}) 
& = - \frac{ \big( -1,  D_{x'} \psi(\bar{x}') \big) }{\big| \big( -1,D_{x'}\psi(\bar{x}') \big)\big|} \\ 
& = - \left( - \frac{1}{ \sqrt{ 1+ |D_{x'}\psi(\bar{x}')|^{2}}}, \frac{ D_{x'} \psi(\bar{x}') }{ \sqrt{ 1+ |D_{x'}\psi(\bar{x}')|^{2}}}\right) .
\end{aligned}\end{equation}
Also let $H$ be the hyperplane orthogonal to $V_{1}$ at the origin, and let $\{ V_{2} , \cdots, V_{n} \}$ be an orthonormal basis of the hyperplane $H$ satisfying $\det (V_{1}^{T}, \cdots, V_{n}^{T}) >0$. Then let $y$-coordinate system be the coordinate system with the orthonormal basis $\{ V_{1}, \cdots, V_{n} \}$.

\medskip

\noindent [Step 2: Transformation matrix between $y$-coordinate system and $x$-coordinate system]
Set the orthonormal matrix $O$ as
\begin{equation}\label{BBB6600}
O = \left( \begin{array}{ccc} 
V_{11} & \cdots & V_{n1} \\
\vdots & \vdots & \vdots \\
V_{1n} & \cdots & V_{nn}
\end{array}\right)
= \left( \begin{array}{c} V_{1} \\ \vdots \\ V_{n} \end{array} \right)^{T}
\qquad \Longrightarrow \qquad
\det O = \det V >0.
\end{equation}
Since $y$-coordinate system has the basis $\{ V_{1}, \cdots, V_{n} \}$, we have that
\begin{equation*}\label{}
x^{k} = \left \langle \sum_{j=1}^{n} y^{j} V_{j} , e_{k} \right \rangle 
= \sum_{j=1}^{n} y^{j} V_{jk} 
= \sum_{j=1}^{n} y^{j} O_{kj} 
= y \cdot O_{k}
\qquad (k \in [1,n]),
\end{equation*}
and so
\begin{equation}\label{BBB6800}
x = (y \cdot O_{1}, \cdots, y \cdot O_{n}) = yO^{T} 
\qquad \text{and} \qquad
y = (x \cdot V_{1}, \cdots, x \cdot V_{n}) = x O .
\end{equation}
Moreover, by \eqref{BBB6530},
\begin{equation*}
O_{1} \cdot e_{1} = V_{11} =  \frac{1}{ \sqrt{ 1+ |D_{x'}\psi(\bar{x}')|^{2}}} > 0,
\end{equation*}
and \eqref{BBB6350} follows.

\medskip

\noindent [Step 3: Local existence of the graph function] 
To apply Lemma \ref{AAA1000} with respect to $y$-coordinate system, we show that local existence of the graph function.

In view of \eqref{BBB6800}, the point $z$ in $x$-coordinate system is represented by the point $zO$ with respect to $y$-coordinate system. So to use the implicit function theorem, we compute the following. For a fixed $z \in S$,
\begin{equation}\begin{aligned}\label{BBB7200}
& D_{y^{1}} \left[ y\cdot   O_{1} - \psi(y \cdot O_{2}, \cdots, y \cdot  O_{n}) \right] \big|_{y=zO}\\
& \quad = O_{11} - \sum_{2 \leq k \leq n} O_{k1} D_{x^{k}}\psi(zO \cdot O_{2}, \cdots, zO \cdot O_{n}) \\
& \quad = O_{11} - \sum_{2 \leq k \leq n} O_{k1} D_{x^{k}}\psi(z').
\end{aligned}\end{equation}
We estimate the last term of \eqref{BBB7200}. In view of \eqref{BBB6530} and \eqref{BBB6600}, we find that
\begin{equation*}
(O_{11}, \cdots, O_{n1}) = V_{1} = - \frac{ \big( -1,  D_{x'} \psi(\bar{x}') \big) }{\big| \big( -1,D_{x'}\psi(\bar{x}') \big)\big|},
\end{equation*}
which implies that
\begin{equation*}
O_{11} - \sum_{2 \leq k \leq n} O_{k1} D_{x^{k}}\psi(z') 
 = \frac{ \big( -1,  D_{x'} \psi(\bar{x}') \big) \cdot \big( -1,  D_{x'} \psi(z') \big) }{\big| \big( -1,D_{x'}\psi(\bar{x}') \big)\big|}
\qquad (z' \in B_{8R}').
\end{equation*}
So we find that for any $z' \in B_{8R}'$,
\begin{equation*}\begin{aligned}
&O_{11} - \sum_{2 \leq k \leq n} O_{k1} D_{x^{k}}\psi(z') \\
& \quad =   \frac{ \big( -1,  D_{x'} \psi(\bar{x}') \big) \cdot \big[ \big( -1,  D_{x'} \psi(\bar{x}') \big) +   \big( -1,  D_{x'} \psi(z') \big) - \big( -1,  D_{x'} \psi(\bar{x}') \big)  \big]}{\big| \big( -1,D_{x'}\psi(\bar{x}') \big)\big|}.
\end{aligned}\end{equation*}
From the assumption that $n [D_{x'}\psi ]_{C^{\gamma}(B_{8R}')} (16R)^{\gamma} \leq 1/4$, for any $z' \in B_{8R}'$,
\begin{equation*}
\frac{ \big|  \big( -1,  D_{x'} \psi(\bar{x}') \big) \cdot \big[ \big( -1,  D_{x'} \psi(z') \big) - \big( -1,  D_{x'} \psi(\bar{x}') \big) \big] \big|}{\big| \big( -1,D_{x'}\psi(\bar{x}') \big)\big|} \leq |D_{x'}\psi (z') - D_{x'}\psi(\bar{x}')| 
 \leq \frac{1}{4}.
\end{equation*}
By combining the above two estimates, we obtain that
\begin{equation}\label{BBB7300}\begin{aligned}
O_{11} - \sum_{2 \leq k \leq n} O_{k1} D_{x^{k}}\psi(z') 
\geq \frac{ 3 \big| \big( -1,D_{x'}\psi(\bar{x}') \big)\big| }{4}
\qquad (z' \in B_{8R}').
\end{aligned}\end{equation}
It follows from \eqref{BBB7200} and \eqref{BBB7300} that
\begin{equation}\label{BBB7400}
D_{y^{1}} \left[ y\cdot   O_{1} - \psi(y \cdot O_{2}, \cdots, y \cdot  O_{n}) \right] \big|_{y=zO}
\geq \frac{1}{2}
\qquad (z \in S).
\end{equation}
Fix an arbitrary point $z \in S$. By the implicit function theorem with \eqref{BBB7400}, there exist ball $U_{z} \subset \mathbb{R}^{n-1}$ and function $\varphi_{z} : U_{z} \to \mathbb{R}$ such that
\begin{equation}\label{BBB7500}
(zO)' \in U_{z},
\qquad (zO)^{1} = \varphi_{z}\big( (zO)' \big),
\end{equation}
and
\begin{equation}\label{BBB7600}
(\varphi_{z}(y'),y') \cdot O_{1} - \psi \big(  (\varphi_{z}(y'),y') \cdot O_{2} , \cdots ,  (\varphi_{z}(y'),y') \cdot O_{n} \big) = 0
\qquad
(y' \in U_{z}).
\end{equation}
Also \eqref{BBB7500} yields that
$ \big( \varphi_{z}\big( (zO)' \big), (zO)' \big) = \big( (zO)^{1}, (zO)' \big) = zO \in B_{8R}$. So  with that $(zO)' \in U_{z}$, one can choose ball $U_{z}$ satisfying \eqref{BBB7500} and \eqref{BBB7600} so small that
\begin{equation}\label{BBB7800}
\big\{ \big( \varphi_{z}(y'),y' \big) : y' \in U_{z} \big\} \subset B_{8R}.
\end{equation} 
To apply Lemma \ref{AAA1000}, we need to check the assumptions \eqref{AAA1220}, \eqref{AAA1240}, \eqref{AAA1260} and \eqref{AAA1250}. By comparing with \eqref{BBB6200}, \eqref{BBB7300}, \eqref{BBB7500}, \eqref{BBB7600} and \eqref{BBB7800}, we only need to estimate $ \| D_{y'} \varphi_{z} \|_{L^{\infty}(U_{z})}$. 

From \eqref{BBB7600}, we have
\begin{equation*}\begin{aligned}
D_{y^{k}} \varphi_{z} (w')
& = - \frac{ D_{y^{k}}  \left[ y \cdot   O_{1} - \psi \big( y \cdot O_{2}, \cdots , y \cdot O_{n} \big) \right] \big|_{y = (\varphi_{z}(w'),w')} }
{ D_{y^{1}} \left[ y \cdot   O_{1} - \psi \big( y \cdot O_{2}, \cdots , y \cdot O_{n} \big) \right]  \big|_{y = (\varphi_{z}(w'),w')} } \\
& = - \frac{ O_{1k}  - \sum\limits_{2 \leq i \leq n} O_{ik}  D_{x^{i}}\psi \big( y \cdot O_{2}, \cdots , y \cdot O_{n} \big) \big|_{y = (\varphi_{z}(w'),w')} }{ O_{11} - \sum\limits_{2 \leq i \leq n} O_{i1} D_{x^{i}}\psi \big( y \cdot O_{2} , \cdots , y \cdot O_{n} \big) \big|_{y = (\varphi_{z}(w'),w')} } 
\end{aligned}\end{equation*}
for any $w' \in U_{z}$. Then from the fact that $(y \cdot O_{2}, \cdots, y \cdot O_{n}) = (y O^{T})'$, we have
\begin{equation}\label{BBB7900}
D_{y^{k}} \varphi_{z} (w') 
= - \frac{ O_{1k}  - \sum\limits_{2 \leq i \leq n} O_{ik} D_{x^{i}}\psi \Big( \big[ (\varphi_{z}(w'),w')O^{T} \big]' \Big) }{ O_{11} - \sum\limits_{2 \leq i \leq n} O_{i1} D_{x^{i}}\psi \Big( \big[ (\varphi_{z}(w'),w')O^{T} \big]' \Big) }
\qquad (w' \in U_{z}).
\end{equation}
To estimate \eqref{BBB7900}, we claim that 
\begin{equation}\begin{aligned}\label{BBB8300}
O_{11} - \sum_{2 \leq k \leq n} O_{i1}D_{x^{i}}\psi (x') 
\geq \frac{ \big\| \big( -1,D_{x'}\psi  \big)\big\|_{L^{\infty}(B_{8R}')}  }{2}
\qquad (x' \in  B_{8R}').
\end{aligned}\end{equation}
By the triangle inequality, 
\begin{equation*}\begin{aligned}
2|D_{x'}\psi(\bar{x}')|^{2} + 2| D_{x^{i}}\psi ( x')  - D_{x'}\psi(\bar{x}')|^{2} 
 \geq \left| D_{x^{i}}\psi (x') \right|^{2}
\qquad (x' \in B_{8R}').
\end{aligned}\end{equation*}
Then from the assumption that $n [ D_{x'}\psi]_{C^{\gamma}(B_{8R}')} (16R)^{\gamma} \leq 1/4$, we have that
\begin{equation*}
2 |D_{x'}\psi(\bar{x}')|^{2} 
\geq \left| D_{x'}\psi (x') \right|^{2}- 1
\qquad (x' \in B_{8R}'),
\end{equation*}
which implies that
\begin{equation*}
1+|D_{x'}\psi(\bar{x}')|^{2} 
\geq \frac{1 + \left| D_{x'}\psi (x') \right|^{2}}{2}
\qquad (x' \in B_{8R}'),
\end{equation*}
and so 
\begin{equation*}
\frac{ 3 \big| \big( -1,D_{x'}\psi(\bar{x}') \big)\big| }{4} 
\geq \frac{ 3 \big\| \big( -1,D_{x'}\psi  \big) \big\|_{L^{\infty}(B_{8R}')}  }{4 \sqrt{2} }
\geq \frac{ \big\| \big( -1,D_{x'}\psi  \big) \big\|_{L^{\infty}(B_{8R}')}  }{2}. 
\end{equation*}
Thus the claim \eqref{BBB8300} follows from \eqref{BBB7300}. On the other-hand, for any $k \in [2,n]$,
\begin{equation}\begin{aligned}\label{BBB8350}
& O_{1k}  - \sum\limits_{2 \leq i \leq n} O_{ik} D_{x^{i}}\psi \Big( \big[ (\varphi_{z}(w'),w')O^{T} \big]' \Big) \\
& \quad = - \big( O_{1k}, \cdots, O_{nk} \big) \cdot
\Big( -1, D_{x'}\psi\Big( \big[ (\varphi_{z}(w'),w')O^{T} \big]' \Big) \Big)
\qquad (w' \in U_{z}).
\end{aligned}\end{equation} 
In view of \eqref{BBB7800}, we have that $ \big[ (\varphi_{z}(w'),w')O^{T} \big]' \in B_{8R}'$ for any $w' \in U_{z}$. So by \eqref{BBB7900}, \eqref{BBB8300} and \eqref{BBB8350}, we have that
\begin{equation}\label{BBB8400}
\| D_{y^{k}}\varphi_{z} \|_{L^{\infty}(U_{z})}
\leq \frac{2 \Big| \Big( -1, D_{x'}\psi \big( \big[ (\varphi_{z}(w'),w')O^{T} \big]' \big) \Big) \Big| }{ \big\| \big( -1,D_{x'}\psi  \big) \big\|_{L^{\infty}(B_{8R}')} } 
\leq 2
\quad  (k \in [2,n]).
\end{equation}

\medskip

\noindent [Step 4: Existence of the graph function $\varphi$ in $B_{R}'$]
We apply Lemma \ref{AAA1000} by comparing \eqref{BBB6500} and \eqref{BBB7300} with \eqref{AAA1220}, \eqref{BBB7500} and \eqref{BBB8400} with \eqref{AAA1240}, \eqref{BBB7600} with \eqref{AAA1260} and \eqref{BBB7800} with \eqref{AAA1250}. Then there exists $C^{1}$-function $\varphi : B_{R}' \to \mathbb{R}$ such that
\begin{equation}\label{BBB8500}
\big( \varphi(y'),y' \big) \cdot O_{1} - \psi \big( \big( \varphi(y'),y' \big) \cdot O_{2}, \cdots , \big( \varphi(y'),y' \big) \cdot O_{n} \big) =0
\quad (y' \in B_{R}'),
\end{equation}
\begin{equation}\label{BBB8600}
(zO)^{1} = \varphi \big( (zO)' \big) 
\qquad (z \in S \cap B_{R}),
\end{equation}
\begin{equation}\label{BBB8700}
D_{y^{k}} \varphi(w')
= - \frac{ O_{1k}  - \sum\limits_{2 \leq i \leq n} O_{ik} D_{x^{i}}\psi \Big( \big[ (\varphi(w'),w')O^{T} \big]' \Big) }{ O_{11} - \sum\limits_{2 \leq i \leq n} O_{i1} D_{x^{i}}\psi \Big( \big[ (\varphi(w'),w')O^{T} \big]' \Big) } 
\quad (w' \in B_{R}'),
\end{equation}
\begin{equation}\label{BBB8750}
\| D_{y'}\varphi \|_{L^{\infty}(B_{R}')} 
\leq 2\sqrt{n},
\end{equation}
and
\begin{equation}\label{BBB8800}
\big\{ \big( \varphi(y'),y' \big) : y' \in B_{R}' \big\} \subset B_{8R}.
\end{equation} 
We find that \eqref{BBB6300}, \eqref{BBB6380} and \eqref{BBB6360} hold from \eqref{BBB8500}, \eqref{BBB8600} and \eqref{BBB8800} respectively. Moreover, \eqref{BBB8300} and \eqref{BBB8800} yields that
\begin{equation}\label{BBB8850}
O_{11} - \sum_{2 \leq k \leq n} O_{i1}D_{x^{i}}\psi \Big( \big[ (\varphi(w'),w')O^{T} \big]' \Big)
\geq \frac{ \big\| \big( -1,D_{x'}\psi  \big) \big\|_{L^{\infty}(B_{8R}')} }{2}
\end{equation}
for any $w' \in B_{R}'$. We will prove \eqref{BBB6370} and \eqref{BBB6400} in Step 5.

\medskip

\noindent [Step 5: Estimate of $|\varphi(0')|$, $D_{y'}\varphi(0')$ and $[D_{y'}\varphi]_{C^{\gamma}(B_{R}')}$]
From the minimality of $\bar{x}$ in \eqref{BBB6500}, we use Lagrange multiplier method  to find that $D[x^{1} - \psi(x')] = \big( 1,-D_{x'}\psi(\bar{x}') \big)$ and $D \big( |x|^{2} \big) = 2x$ are parallel at $\bar{x} = (\psi(\bar{x}'),x')$. Thus
\begin{equation}\label{BBB9000}
\bar{x}
= C_{*} |\bar{x}| \cdot \frac{ \big( 1, - D_{x'}\psi(\bar{x}') \big)}{\sqrt{ 1+ |D_{x'}\psi(\bar{x}')|^{2}}}
= C_{*} |\bar{x}| V_{1}
~\text{ for }~ C_{*} =1 \text{ or } C_{*}=-1.
\end{equation}
It follows from \eqref{BBB6600}, \eqref{BBB8600} and \eqref{BBB9000} that
\begin{equation}\label{BBB9100}
(\bar{x}O)' = (\bar{x} \cdot V_{2}, \cdots, \bar{x} \cdot V_{n}) = 0' 
\qquad \text{and} \qquad
(\bar{x}O)^{1} = \varphi \big( (\bar{x}O)' \big) = \varphi(0').
\end{equation}
So we find that
\begin{equation}\label{BBB9200}
|\varphi(0')| = |(\bar{x}O)^{1}| <  R.
\end{equation}
Also if $\psi(0')=0$ then by letting $\bar{x} = (0,0')$, we have that $\varphi(0')=0$. We find from \eqref{BBB9100} that
\begin{equation*}\begin{aligned}
D_{x'}\psi \Big( \big[ (\varphi(0'),0')O^{T} \big]' \Big)
= D_{x'} \psi \Big( [\bar{x} O O^{T}]' \Big)
= D_{x'} \psi (\bar{x}'),
\end{aligned}\end{equation*}
and \eqref{BBB6600}  implies that for any $k \in [2,n]$,
\begin{equation*}\begin{aligned}
O_{1k}  - \sum\limits_{2 \leq i \leq n} O_{ik} D_{x^{i}}\psi \Big( \big[ (\varphi(0'),0')O^{T} \big]' \Big) 
& = - V_{k}\cdot (-1 ,D_{x'} \psi (\bar{x}')) \\
& = (V_{k} \cdot V_{1}) |(-1 ,D_{x'} \psi (\bar{x}'))| = 0.
\end{aligned}\end{equation*}
So we have from  \eqref{BBB8700} that
\begin{equation}\label{BBB9400}
D_{y^{k}} \varphi(0')
= - \frac{ O_{1k}  - \sum\limits_{2 \leq i \leq n} O_{ik} D_{x^{i}}\psi \Big( \big[ (\varphi(0'),0')O^{T} \big]' \Big) }{ O_{11} - \sum\limits_{2 \leq i \leq n} O_{i1} D_{x^{i}}\psi \Big( \big[ (\varphi(0'),0')O^{T} \big]' \Big) } = 0
\quad (k \in [2,n]),
\end{equation}
and \eqref{BBB6370} follows from \eqref{BBB8750}, \eqref{BBB9200} and \eqref{BBB9400}.

Next, we estimate $[D_{y'}\varphi]_{C^{\gamma}(B_{R}')}$. 
For any $w',z' \in B_{R}'$, we have from \eqref{BBB8750} that
\begin{equation}\label{BBB9500}
|(\varphi(w'),w') - (\varphi(z'),z')|
\leq |(2\sqrt{n}|w'-z'|,w'-z')|
\leq 3\sqrt{n} |w'-z'|
\end{equation}
for any $w', z' \in B_{R}'$. Recall from  \eqref{BBB8700} that
\begin{equation*}
D_{y^{k}} \varphi (w')
= - \frac{  O_{1k}  - \sum\limits_{2 \leq i \leq n}  O_{ik}  D_{x^{i}}\psi \Big( \big[ (\varphi(w'),w')O^{T} \big]' \Big)  }{ O_{11} - \sum\limits_{2 \leq i \leq n} O_{i1} D_{x^{i}}\psi \Big( \big[ (\varphi(w'),w')O^{T} \big]' \Big)  } 
\qquad (w' \in B_{R}').
\end{equation*}
Then for any $w',z' \in B_{R}'$, we have
\begin{equation}\begin{aligned}\label{BBB9550}
|D_{y^{k}}\varphi(w') - D_{y^{k}}\varphi(z')| 
& \leq \bigg| \frac{I}{II}  - \frac{III}{IV} \bigg| \\
& \leq \bigg| \frac{I(IV-II) + II(I-III)}{II \cdot IV} \bigg| \\
& \leq \bigg| \frac{I}{II} \bigg| \bigg| \frac{IV-II}{IV} \bigg| + \bigg| \frac{I-III}{IV} \bigg|,
\end{aligned}\end{equation}
where
\begin{equation*}
I = O_{1k}  - \sum_{ 2 \leq i \leq n }  O_{ik}  D_{x^{i}}\psi \Big( \big[ (\varphi(w'),w')O^{T} \big]' \Big) ,
\end{equation*}
\begin{equation*}
II = O_{11} - \sum_{ 2 \leq i \leq n } O_{i1}   D_{x^{i}}\psi \Big( \big[ (\varphi(w'),w')O^{T} \big]' \Big) ,
\end{equation*}
\begin{equation*}
III = O_{1k}  - \sum_{ 2 \leq i \leq n } O_{ik}  D_{x^{i}}\psi \Big( \big[ (\varphi(z'),z')O^{T} \big]' \Big),
\end{equation*}
and
\begin{equation*}
IV = O_{11}  - \sum_{ 2 \leq i \leq n }  O_{i1}   D_{x^{i}}\psi \Big( \big[ (\varphi(z'),z')O^{T} \big]' \Big).
\end{equation*}
So one can check from \eqref{BBB8800}, \eqref{BBB8850} and  \eqref{BBB9500} that
\begin{equation*}\begin{aligned}
\bigg| \frac{IV-II}{IV} \bigg|
& \leq \frac{ 2 |(O_{2k},\cdots,O_{nk})| \big| D_{x'}\psi \big( \big[ (\varphi(w'),w')O^{T} \big]' \big) - D_{x'}\psi \big( \big[ (\varphi(z'),z')O^{T} \big]' \big) \big|}{ \| (-1,D_{x'} \psi ) \|_{L^{\infty}(B_{8R}')} } \\
& \leq \frac{ 2 [D_{x'}\psi]_{C^{\gamma}(B_{8R}')} \left| \big[ (\varphi(w'),w')O^{T} \big]' - (\varphi(z'),z')O^{T} \big]' \right|^{\gamma}}{ \| (-1,D_{x'} \psi ) \|_{L^{\infty}(B_{8R}')} } \\
& \leq \frac{6\sqrt{n} [D_{x'}\psi]_{C^{\gamma}(B_{8R}')} |w'-z'|^{\gamma} }{ \| (-1,D_{x'} \psi ) \|_{L^{\infty}(B_{8R}')} },
\end{aligned}\end{equation*}
and one can easily check from \eqref{BBB8850} that
\begin{equation*}\begin{aligned}
\bigg| \frac{I}{II} \bigg|
& \leq \frac{ 2 |(O_{1k}, \cdots, O_{nk})| \big \| (-1, D_{x'}\psi) \big \|_{L^{\infty}(B_{8R}')} }{ \big\| (-1, D_{x'}\psi) \big\|_{L^{\infty}(B_{8R}')} }
\leq 2.
\end{aligned}\end{equation*}
By combining the above two estimates,
\begin{equation}\begin{aligned}\label{BBB9600}
\bigg| \frac{I}{II} \bigg| \bigg| \frac{IV-II}{IV} \bigg|
\leq \frac{12\sqrt{n} [D_{x'}\psi]_{C^{\gamma}(B_{8R}')} |w'-z'|^{\gamma} }{ \big\| (-1, D_{x'}\psi) \big\|_{L^{\infty}(B_{8R}')} }.
\end{aligned}\end{equation}
Similarly,
\begin{equation}\begin{aligned}\label{BBB9700}
\bigg| \frac{I-III}{IV} \bigg|
\leq \frac{6\sqrt{n} [D_{x'}\psi]_{C^{\gamma}(B_{8R}')} |w'-z'|^{\gamma} }{ \| (-1,D_{x'} \psi ) \|_{L^{\infty}(B_{8R}')} }.
\end{aligned}\end{equation}
By combining  \eqref{BBB9550}, \eqref{BBB9600} and \eqref{BBB9700}, we have that
\begin{equation}\label{BBB9800}
|D_{y^{k}}\varphi(w') - D_{y^{k}}\varphi(z')| 
\leq \frac{ 18\sqrt{n} [D_{x'}\psi]_{C^{\gamma}(B_{8R}')} |w'-z'|^{\gamma} }{ \| (-1,D_{x'} \psi ) \|_{L^{\infty}(B_{8R}')} }
\quad (w',z' \in B_{R}').
\end{equation}
So we discover that the estimate  \eqref{BBB6400} holds from  \eqref{BBB9800}.
\end{proof}

To apply Lemma \ref{BBB6000} to a domain, we will use Lemma \ref{CCC1000} which can be obtained by Lemma \ref{PPP1000}.

\begin{lem}\label{PPP1000}
For $U \subset \mathbb{R}^{n}$, orthonormal matrices $V \in \mathbb{R}^{ n \times n}$ and $W \in \mathbb{R}^{n \times n}$ with $\det V>0$ and $\det W>0$, assume that $C^{1}$-functions $\psi : B_{8R}' \to \mathbb{R}$ and $\varphi : B_{R}' \to \mathbb{R}$ satisfy
\begin{equation}\label{PPP1300}
U \cap B_{R} = \left \{ \sum\limits_{1 \leq k \leq n} x^{k} V_{k} \in B_{R} : x^{1} > \psi(x') \right \},
\end{equation}
\begin{equation}\label{PPP1200}
\big( \varphi(y'),y' \big) \cdot W_{1} - \psi \big( \big( \varphi(y'),y' \big) \cdot W_{2}, \cdots , \big( \varphi(y'),y' \big) \cdot W_{n} \big) =0
\qquad (y' \in B_{R}'),
\end{equation}
and
\begin{equation*}
\big\{ \big( \varphi(y'),y' \big) : y' \in B_{R}' \big\} \subset B_{8R}.
\end{equation*}
For the orthonormal matrix
\begin{equation*}
O = W^{T} V 
\quad \text{with} \quad
\det O > 0,
\end{equation*}
the corresponding sets
\begin{equation*}
U^{+} = \left \{ \sum\limits_{1 \leq k \leq n} x^{k} V_{k} \in B_{R} : x^{1} > \psi(x') \right \},
~ 
V^{+} = \left \{ \sum\limits_{1 \leq k \leq n} y^{k} O_{k} \in B_{R} : y^{1} > \varphi(y') \right \},
\end{equation*}
\begin{equation*}
U^{-} = \left \{ \sum\limits_{1 \leq k \leq n} x^{k} V_{k} \in B_{R} : x^{1} < \psi(x') \right \},
~
V^{-} = \left \{ \sum\limits_{1 \leq k \leq n} y^{k} O_{k} \in B_{R} : y^{1} < \varphi(y') \right \},
\end{equation*}
and
\begin{equation*}
U^{0} = \left \{ \sum\limits_{1 \leq k \leq n} x^{k} V_{k} \in B_{R} : x^{1} = \psi(x') \right \},
~
V^{0} = \left \{ \sum\limits_{1 \leq k \leq n} y^{k} O_{k} \in B_{R} : y^{1} = \varphi(y') \right \},
\end{equation*}
further assume that
\begin{equation}\label{PPP3000}
U^{+} \cap V^{+} \not = \emptyset
\qquad \text{and} \qquad
U^{-} \cap V^{-} \not = \emptyset.
\end{equation}
Then $U^{+} = V^{+}$, $U^{0}=V^{0}$, $U^{-}=V^{-}$ and
\begin{equation}\label{PPP1800}
U \cap B_{R} = \left \{ \sum\limits_{1 \leq k \leq n} y^{k} O_{k} \in B_{R} : y^{1} > \varphi(y') \right \}.
\end{equation}
\end{lem}

\begin{proof}
\eqref{PPP1800} follows from $U^{+}=V^{+}$. So we will prove that $U^{+} = V^{+}$, $U^{0}=V^{0}$ and $U^{-}=V^{-}$. One can easily check that
\begin{equation}\label{PPP2900}
U^{+} \sqcup U^{0} \sqcup U^{-}
= B_{R}
= V^{+} \sqcup V^{0} \sqcup V^{-}.
\end{equation}

\medskip

\noindent [Step 1 : $U^{+} \subset V^{+}$ and $U^{-} \subset V^{-}$] We claim that 
\begin{equation}\label{PPP4100}
U^{+} \subset V^{+}
\qquad \text{and} \qquad
U^{-} \subset V^{-}.
\end{equation}
To prove it, suppose that
\begin{equation}\label{PPP4200}
U^{+} \cap (V^{-} \cup V^{0}) \not = \emptyset
\qquad \text{and} \qquad
U^{+} \cap (V^{0} \cup V^{+}) \not = \emptyset.
\end{equation}
Then there exist $z_{1} = \sum\limits_{1 \leq k \leq n} (z_{1} \cdot O_{k}) O_{k} \in U^{+}$ and $z_{2} = \sum\limits_{1 \leq k \leq n} (z_{2} \cdot O_{k}) O_{k} \in U^{+}_{} $ satisfying
\begin{equation*}
z_{1} \in U^{+} 
~ \text{ with } ~ 
z_{1} \cdot O_{1}
\leq \varphi \left( z_{1} \cdot O_{2}, \cdots, z_{2} \cdot O_{n}  \right),
\end{equation*}
and
\begin{equation*}
z_{2} \in U^{+} 
~ \text{ with } ~ 
z_{2} \cdot O_{1}
\geq \varphi \left( z_{2} \cdot O_{2}, \cdots, z_{2} \cdot O_{n}  \right).
\end{equation*}
Since $U^{+}$ is connected, one can choose a path $\mathit{l} \subset U^{+}$ connecting $z_{1}$ and $z_{2}$. So there exists $z \in \mathit{l} \subset U^{+}$ such that
\begin{equation*}
z \in l \subset U^{+} 
\qquad \text{and} \qquad
z \cdot O_{1}
= \varphi \left( z \cdot O_{2}, \cdots, z \cdot O_{n}  \right),
\end{equation*}
and by letting $\bar{z} = (\bar{z}^{1},\bar{z}') : = zO^{T} = (z \cdot O_{1}, \cdots, z \cdot O_{n})$, 
\begin{equation}\label{PPP4300}
z \in U^{+},
\qquad
(zO^{T})^{1} = \varphi \left( (zO^{T})' \right)
\qquad \text{and} \qquad
\bar{z}^{1} = \varphi(\bar{z}').
\end{equation}
By the fact that $\bar{z} = zO^{T} \in B_{R}$, we have that $\bar{z}' \in B_{R}'$. So we find from \eqref{PPP1200} that
\begin{equation*}
\left( \varphi \left( \bar{z}'  \right) , \bar{z}' \right) \cdot W_{1} 
= \psi \big( \left( \varphi \left( \bar{z}' \right) ,\bar{z}' \right) \cdot W_{2}, \cdots, \left( \varphi \left( \bar{z}'  \right)  , \bar{z}' \right) \cdot W_{n} \big),
\end{equation*}
and \eqref{PPP4300} implies that
\begin{equation}\label{PPP4350}
\bar{z} \cdot W_{1} = \psi ( \bar{z} \cdot W_{2}, \cdots, \bar{z} \cdot W_{n}).
\end{equation}
By a direct calculation,
\begin{equation*}\begin{aligned}
\bar{z} \cdot W_{l} 
= (z O^{T}) \cdot W_{l}
= \left( \sum_{1 \leq i \leq n} z^{i} O_{1i}, \cdots,  \sum_{1 \leq i \leq n}  z^{i} O_{ni} \right) \cdot W_{l}
= \sum_{1 \leq i,j \leq n} z^{i} O_{ji} W_{lj}.
\end{aligned}\end{equation*}
Since $O = W^{T}V $, we have that
$O_{ji} = \sum
\limits_{1 \leq k \leq n} W_{kj} V_{ki}$. Thus
\begin{equation*}\begin{aligned}
\bar{z} \cdot W_{l} 
= \sum_{1 \leq i,j \leq n} z^{i} O_{ji} W_{lj} 
= \sum_{1 \leq i,j,k \leq n} z^{i} W_{kj} V_{ki} W_{lj}
= \sum_{1 \leq i \leq n} z^{i} V_{li}= z \cdot V_{l}.
\end{aligned}\end{equation*}
So it follows from \eqref{PPP4350} that
\begin{equation*}
z \cdot V_{1} = \psi ( z \cdot V_{2}, \cdots, z \cdot V_{n}),
\end{equation*}
and by the fact that $z = \sum\limits_{1 \leq k \leq n} (z \cdot V_{k}) V_{k}$, $z \in U^{0}$ holds. Since $U^{+}$, $U^{-}$ and $U^{0}$ are mutually disjoint, this contradicts $z \in U^{+}$ in \eqref{PPP4300}. So \eqref{PPP4200} can not occur, and
\begin{equation}\label{PPP4400}
\text{one of} \quad
U^{+} \cap (V^{-} \cup V^{0}) = \emptyset
\quad \text{or} \quad
U^{+} \cap (V^{0} \cup V^{+}) = \emptyset
\quad \text{holds.}
\end{equation}
If $U^{+} \cap (V^{0} \cup V^{+}) = \emptyset$ then  \eqref{PPP2900} implies that $U^{+} \subset V^{-}$. So \eqref{PPP3000} yields that
\begin{equation*}
\emptyset \not = U^{+} \cap V^{+} \subset V^{-} \cap V^{+} = \emptyset,
\end{equation*}
which gives a contradiction. Thus by \eqref{PPP4400}, we obtain that
\begin{equation*}
U^{+} \cap (V^{-} \cup V^{0}) = \emptyset.
\end{equation*}
and it follows from \eqref{PPP2900} that
\begin{equation}\label{PPP4500}
U^{+} \subset V^{+}.
\end{equation}
By repeating the same argument for proving \eqref{PPP4500}, one can also prove that
\begin{equation}\label{PPP4600}
U^{-} \subset V^{-},
\end{equation}
and \eqref{PPP4100} follows from \eqref{PPP4500} and \eqref{PPP4600}.

\medskip

\noindent [Step 3 : $U^{0} \subset V^{0}$] We prove that
\begin{equation}\label{PPP5000}
U^{0} \subset V^{0}.
\end{equation}
To prove it, we claim that
\begin{equation}\label{PPP5100}
U^{0} \cap (V^{+} \cup V^{-}) = \emptyset.
\end{equation}
Suppose not. Then one of $U^{0} \cap V^{+} \not = \emptyset$ or $U^{0} \cap V^{-} \not = \emptyset$ holds. If $z \in U^{0} \cap V^{+}$ then by the fact that $V^{+}$ is open, there exists $\rho >0$ such that 
\begin{equation*}
B_{\rho}(z) \subset V^{+}.
\end{equation*}
By the definition of $U^{0}$, $B_{\rho}(z) \cap U^{+} \not = \emptyset$ and $B_{\rho}(z) \cap U^{-} \not = \emptyset$, which implies that
\begin{equation*}
V^{+} \cap U^{-} \supset B_{\rho}(z) \cap U^{-} \not = \emptyset
\qquad \text{and} \qquad
V^{+} \cap U^{+} \supset B_{\rho}(z) \cap U^{+} \not = \emptyset.
\end{equation*}
So by \eqref{PPP2900} and \eqref{PPP4100},
\begin{equation*}
\emptyset \not = V^{+} \cap U^{-} \subset V^{+} \cap V^{-} = \emptyset,
\end{equation*}
and a contradiction occurs. Similarly, also for the case that $z \in U^{0} \cap V^{-}$, one can obtain a contradiction. Thus we find that the claim \eqref{PPP5100} holds. By \eqref{PPP2900} and \eqref{PPP5100}, \eqref{PPP5000} follows.

\medskip

\noindent [Step 4 : $U^{+} = V^{+}$, $U^{0} = V^{0}$ and 
$U^{-} = V^{-}$] Recall from \eqref{PPP2900} that  $U^{+} \sqcup U^{0} \sqcup U^{-} = B_{R} = V^{+} \sqcup V^{0} \sqcup V^{-}$. So it follows from \eqref{PPP4100} and \eqref{PPP5000} that
\begin{equation*}
U^{+} = V^{+},
\qquad
U^{0} = V^{0}
\qquad \text{and} \qquad
U^{-} = V^{-},
\end{equation*}
which proves the lemma.
\end{proof}

 The proof of Lemma \ref{PPP1000} and Lemma \ref{PPP7000} are similar and we only state the lemma.

\begin{lem}\label{PPP7000}
For $U \subset \mathbb{R}^{n}$, orthonormal matrices $V \in \mathbb{R}^{ n \times n}$ and $W \in \mathbb{R}^{n \times n}$ with $\det V>0$ and $\det W>0$, assume that $C^{1}$-functions $\psi : B_{8R}' \to \mathbb{R}$ and $\varphi : B_{R}' \to \mathbb{R}$ satisfy
\begin{equation}\label{PPP7300}
U \cap B_{R} = \left \{ \sum\limits_{1 \leq k \leq n} y^{k} V_{k} \in B_{R} : x^{1} > \psi(x') \right \},
\end{equation}
\begin{equation}\label{PPP7200}
\big( \varphi(y'),y' \big) \cdot W_{1} - \psi \big( \big( \varphi(y'),y' \big) \cdot W_{2}, \cdots , \big( \varphi(y'),y' \big) \cdot W_{n} \big) =0
\qquad (y' \in B_{R}'),
\end{equation}
and
\begin{equation*}
\big\{ \big( \varphi(y'),y' \big) : y' \in B_{R}' \big\} \subset B_{8R}.
\end{equation*}
For the orthonormal matrix
\begin{equation*}
O = W^{T} V
\quad \text{with} \quad
\det O > 0,
\end{equation*}
the corresponding sets
\begin{equation*}
U^{+} = \left \{ \sum\limits_{1 \leq k \leq n} y^{k} V_{k} \in B_{R} : y^{1} > \psi(y') \right \},
~ 
V^{+} = \left \{ \sum\limits_{1 \leq k \leq n} x^{k} O_{k} \in B_{R} : x^{1} > \varphi(x') \right \},
\end{equation*}
\begin{equation*}
U^{-} = \left \{ \sum\limits_{1 \leq k \leq n} y^{k} V_{k} \in B_{R} : y^{1} < \psi(y') \right \},
~
V^{-} = \left \{ \sum\limits_{1 \leq k \leq n} x^{k} O_{k} \in B_{R} : x^{1} < \varphi(x') \right \},
\end{equation*}
and
\begin{equation*}
U^{0} = \left \{ \sum\limits_{1 \leq k \leq n} y^{k} V_{k} \in B_{R} : y^{1} = \psi(y') \right \},
~
V^{0} = \left \{ \sum\limits_{1 \leq k \leq n} x^{k} O_{k} \in B_{R} : x^{1} = \varphi(x') \right \},
\end{equation*}
further assume that
\begin{equation}\label{PPP7700}
U^{+} \cap V^{-} \not = \emptyset
\qquad \text{and} \qquad
U^{-} \cap V^{+} \not = \emptyset.
\end{equation}
Then $U^{+} = V^{-}$, $U^{0}=V^{0}$, $U^{-}=V^{+}$ and
\begin{equation}\label{PPP7800}
U \cap B_{R} = \left \{ \sum\limits_{1 \leq k \leq n} x^{k} O_{k} \in B_{R} : x^{1} < \varphi(x') \right \}.
\end{equation}
\end{lem}

\medskip

In view of Lemma \ref{BBB6000} and Lemma \ref{PPP1000}, we obtain the following lemma.

\begin{lem}\label{CCC1000}
For $U \subset \mathbb{R}^{n}$ and $C^{1,\gamma}$-function $\psi : B_{8R}' \to \mathbb{R}$ with
\begin{equation}\label{CCC1100}
S = \{ ( \psi(x'),x') \in B_{8R} : x \in B_{8R}' \}.
\end{equation}
Assume that
\begin{equation}\label{CCC1200}
S \cap B_{R} \not = \emptyset,
\qquad
n [ D_{x'}\psi]_{C^{\gamma}(B_{8R}')} (16R)^{\gamma} \leq 1/4,
\end{equation}
and
\begin{equation}\label{CCC1300}
U \cap B_{R} = \{ (x^{1},x') \in B_{R} : x^{1} > \psi(x') \}.
\end{equation}
Then there exist an orthonormal matrix $ V \in \mathbb{R}^{n \times n}$ with $\det V >0$ and $C^{1,\gamma}$-function $\varphi : B_{R}' \to \mathbb{R}$ such that
\begin{equation}\label{CCC1500}
U \cap B_{R} = \left \{ \sum\limits_{1 \leq k \leq n} y^{k} V_{k} \in B_{R} : y^{1} > \varphi(y') \right\},
\end{equation}
and
\begin{equation}\label{CCC1550}
z \cdot V_{1} = \varphi \big( z \cdot V_{2} , \cdots, z \cdot V_{n} \big) 
\qquad (z \in S \cap B_{R}),
\end{equation}
with the estimate
\begin{equation}\label{CCC1600}
|\varphi(0')| < R,
\qquad
D_{y'}\varphi(0')=0',
\qquad
\| D_{y'}\varphi \|_{L^{\infty}(B_{R}')} 
\leq 2\sqrt{n},
\end{equation}
and
\begin{equation}\label{CCC1700}
[D_{y'}\varphi]_{C^{\gamma}(B_{R}')} 
\leq \frac{ 18n  [D_{y'}\psi]_{C^{\gamma}(B_{8R}')} }{ \| (-1,D_{y'} \psi ) \|_{L^{\infty}(B_{8R}')} }.
\end{equation}
Moreover, if $\psi(0')=0$ then $\varphi(0')=0$. 
\end{lem}

\begin{rem}\label{CCC2000}
In Lemma \ref{CCC1000}, $U \cap B_{R} = \{ (x^{1},x') \in B_{R} : x^{1} > \psi(x') \}$ in $x$-coordinate system is represented by the set $\{ (y^{1},y') \in B_{R} : y^{1} > \varphi(y')\}$ with respect to $y$-coordinate system which has the basis $\{ V_{1}, \cdots, V_{n} \}$.
\end{rem}

\begin{proof}

By Lemma \ref{BBB6000}, there exists an orthonormal matrix $ O \in \mathbb{R}^{n \times n}$ with $\det O >0$ and $C^{1,\gamma}$-function $\varphi : B_{R}' \to \mathbb{R}$  such that
\begin{equation}\label{CCC2400}
\big( \varphi(y'),y' \big) \cdot O_{1} - \psi \big( \big( \varphi(y'),y' \big) \cdot O_{2}, \cdots , \big( \varphi(y'),y' \big) \cdot O_{n} \big) =0
\qquad (y' \in B_{R}'),
\end{equation}
\begin{equation}\label{CCC2500}
(zO)^{1} = \varphi \big( (zO)' \big) 
\qquad (z \in S \cap B_{R}),
\end{equation}
\begin{equation}\label{CCC2600}
O_{1} \cdot e_{1} > 0
\end{equation}
and
\begin{equation}\label{CCC2650}
\big\{ \big( \varphi(y'),y' \big) : y' \in B_{R}' \big\} \subset B_{8R},
\end{equation}
with the estimate
\begin{equation}\label{CCC2700}
|\varphi(0')| <  R,
\qquad
D_{y'}\varphi(0')=0',
\qquad
\| D_{y'}\varphi \|_{L^{\infty}(B_{R}')} 
\leq 2\sqrt{n},
\end{equation}
and
\begin{equation}\label{CCC2800}
[D_{y'}\varphi]_{C^{\gamma}(B_{R}')} 
\leq \frac{ 18n [D_{x'}\psi]_{C^{\gamma}(B_{8R}')} }{ \| (-1,D_{x'} \psi ) \|_{L^{\infty}(B_{8R}')} }.
\end{equation}
Also if $\psi(0')=0$ then $\varphi(0')=0$. So $\varphi$ satisfies \eqref{CCC1550}, \eqref{CCC1600} and \eqref{CCC1700} for
\begin{equation*}
V = O^{T}
\qquad \Longrightarrow \qquad
\det V > 0,
\end{equation*}
and it only remains to prove \eqref{CCC1500}. By a direct calculation,
\begin{equation}\label{CCC3200}
V_{1} \cdot e_{1} = O_{11} = O_{1} \cdot e_{1} >0.
\end{equation}
Let
\begin{equation*}
U^{+} = \left \{ (x^{1},x') \in B_{R} : x^{1} > \psi(x') \right \},
\quad 
V^{+} = \left \{ \sum\limits_{1 \leq k \leq n} y^{k} V_{k} \in B_{R} : y^{1} > \varphi(y') \right\},
\end{equation*}
\begin{equation*}
U^{0} = \left \{ (x^{1},x') \in B_{R} : x^{1} = \psi(x') \right \},
\quad 
V^{0} = \left \{ \sum\limits_{1 \leq k \leq n} y^{k} V_{k} \in B_{R} : y^{1} = \varphi(y') \right\},
\end{equation*}
\begin{equation*}
U^{-} = \left \{ (x^{1},x') \in B_{R} : x^{1} < \psi(x') \right \},
\quad 
V^{-} = \left \{ \sum\limits_{1 \leq k \leq n} y^{k} V_{k} \in B_{R} : y^{1} < \varphi(y') \right\}.
\end{equation*}
To apply Lemma \ref{PPP1000}, we need to check that
\begin{equation}\label{CCC3500}
U^{+} \cap V^{+} \not = \emptyset
\qquad \text{and} \qquad
U^{-} \cap V^{-} \not = \emptyset.
\end{equation}

We will only prove that $U^{+} \cap V^{+} \not = \emptyset$, because $U^{-} \cap V^{-} \not = \emptyset$ can be proved similarly. Let $\bar{z} = (\varphi(0'),0') V \in B_{R}$. From \eqref{CCC2400} and the fact that 
\begin{equation*}
\bar{z} = (\varphi(0'),0') V 
= (\varphi(0'),0') O^{T} 
= \big( (\varphi(0'),0') \cdot O_{1}, \cdots, (\varphi(0'),0') \cdot O_{n} \big),
\end{equation*}
we have that
\begin{equation}\label{CCC3800}
\bar{z}^{1} = \psi(\bar{z}').
\end{equation}
Since $\bar{z} \in B_{R}$, there exists $\epsilon_{0} \in (0,1]$ such that
\begin{equation*}
\epsilon \in (0,\epsilon_{0}] 
\qquad \Longrightarrow \qquad
\epsilon e_{1} + \bar{z} \in B_{R},
\end{equation*}
and we have from \eqref{CCC3800} that
\begin{equation}\label{CCC4000}
\epsilon e_{1} + \bar{z} \in U^{+}
\quad \text{for} \quad
\epsilon \in (0,\epsilon_{0}].
\end{equation}
So it only need to prove that there exists $\epsilon \in (0,\epsilon_{0}]$ with
\begin{equation*}
\epsilon e_{1} + \bar{z} \in V^{+}.
\end{equation*}
Since $\epsilon e_{1} + \bar{z} = \sum\limits_{1 \leq k \leq n} [(\epsilon e_{1} + \bar{z}) \cdot V_{k}] V_{k}$,
we claim that
\begin{equation}\label{CCC4200}
(\epsilon e_{1} + \bar{z})\cdot V_{1} 
> \varphi \big( (\epsilon e_{1} + \bar{z}) \cdot V_{2}, \cdots, (\epsilon e_{1} + \bar{z}) \cdot V_{n} \big).
\end{equation}
Since $\bar{z} \cdot V_{1} = (\varphi(0'),0') V V_{1}^{T} = \varphi(0')$ and $\bar{z} \cdot V_{k} = (\varphi(0'),0') V V_{k}^{T} = 0$ for $2 \leq k \leq 2$, it is suffice to prove that
\begin{equation}\label{CCC4300}
\epsilon e_{1} \cdot V_{1} 
> \varphi \big( (\epsilon e_{1} + \bar{z}) \cdot V_{2}, \cdots, (\epsilon e_{1} + \bar{z}) \cdot V_{n} \big) - \bar{z} \cdot V_{1}
= \varphi \big( (\epsilon e_{1} V^{T})' \big) - \varphi(0').
\end{equation}
By mean value theorem,
\begin{equation}\label{CCC4400}\begin{aligned}
\varphi \big( (\epsilon e_{1} V^{T})' \big) - \varphi(0') = D_{y'} \varphi \big( (\bar{\epsilon} e_{1} V^{T})' \big) \cdot (\epsilon e_{1} V^{T})'
\text{ for some } \bar{\epsilon} \in (0,\epsilon].
\end{aligned}\end{equation}
From \eqref{CCC3200}, $V_{1} \cdot e_{1} > 0$. Also by \eqref{CCC2700} and the fact that $\varphi \in C^{1,\gamma}(B_{R}')$,
\begin{equation*}
D_{y'} \varphi \big( (\bar{\epsilon} e_{1} V^{T})' \big) \cdot (e_{1} V^{T})'
\to 0
\quad \text{as} \quad
\epsilon \to 0,
\end{equation*}
and the claim \eqref{CCC4200} holds from \eqref{CCC4300} and \eqref{CCC4400} for some $\epsilon \in (0,\epsilon_{0}]$. Thus $\epsilon e_{1} + \bar{z} \in V^{+}$ and \eqref{CCC4000} implies that $U^{+} \cap V^{+} \not = \emptyset$. Similarly, one can prove that $U^{-} \cap V^{-} \not = \emptyset$, and obtain \eqref{CCC3500}. The lemma follows by applying Lemma \ref{PPP1000} to \eqref{CCC1300}, \eqref{CCC2400} and \eqref{CCC2650} for $V$, $O$ and $I_{n}$ instead of $O$, $W$ and $V$ respectively.
\end{proof}

In Lemma \ref{EEE1000}, we prove that for $C^{1,\gamma}$-domain $U$, if the normal on $\partial U \cap B_{8R}$ is almost opposite to $e_{1}$, then $U$ is also a graph in $B_{R}$ with respect to $x^{1}$-variable. To obtain Lemma \ref{EEE1000}, we derive Lemma \ref{DDD1000}.

\begin{lem}\label{DDD1000}
Let $V \in \mathbb{R}^{n \times n}$ be an orthonormal matrix with $\det V > 0$ satisfying
\begin{equation}\label{DDD1100}
| V_{1} + e_{1} | \leq \tau
\quad \text{ for some } \tau \in (0,1/(8n)].
\end{equation}
Assume that $\psi : B_{8R}' \to \mathbb{R}$ is  $C^{1}$-function with
\begin{equation}\label{DDD1200}
|\psi(0')| <R,
\qquad
D_{ y'  }\psi(0')=0
\qquad \text{and} \qquad
\| D_{ y'  }\psi \|_{L^{\infty}(B_{8R}')} \leq \tau.
\end{equation} 
Then there exists $C^{1}$-function $\varphi : B_{R}' \to \mathbb{R}$ satisfying
\begin{equation}\label{DDD1500}
\big( \varphi(x' ),x'  \big) \cdot V_{1} - \psi \big( \big( \varphi(x' ),x'  \big) \cdot V_{2}, \cdots , \big( \varphi(x' ),x'  \big) \cdot V_{n} \big) =0
\qquad (x'  \in B_{R}'),
\end{equation}
\begin{equation}\begin{aligned}\label{DDD1600}
\| D_{x' } \varphi \|_{L^{\infty}(B_{R}')} 
\leq 8 \tau  \sqrt{n} ,
\qquad
(zV)^{1} = \varphi \big( (zV)' \big)
\quad 
(z \in S \cap B_{R}),
\end{aligned}\end{equation}
and
\begin{equation}\label{DDD1700}
\big\{ \big( \varphi(x' ),x'  \big) : x'  \in B_{R}' \big\} \subset B_{8R}.
\end{equation}
Moreover, if $\psi \in C^{1,\gamma}(B_{8R}')$ then $\varphi \in C^{1,\gamma}(B_{R}')$ with the estimate
\begin{equation}\label{DDD1800}
[D_{x' }\varphi]_{C^{\gamma}(B_{R}')} \leq 16 [D_{ y'  }\psi]_{C^{\gamma}(B_{8R}')}.
\end{equation}
\end{lem}

\begin{proof}
Set $S = \{ (\psi( y'  ), y'  ) \in B_{8R}  :  y'   \in B_{8R}' \} $. Since $|\psi(0')| <R$, we have that  $ (\psi(0'),0') \in S \cap B_{R}$. Thus
\begin{equation}\label{DDD1900}
S \cap B_{R} \not = \emptyset.
\end{equation} 
We prove the lemma by using Lemma \ref{AAA1000}. To use the notation in Lemma \ref{AAA1000}, set
\begin{equation}\label{DDD2100}
O = V.
\end{equation}
One can check that
\begin{equation*}
x^{k} = \left \langle \sum_{j=1}^{n} y^{j} V_{j}, e_{k} \right \rangle
= \sum_{j=1}^{n} y^{j} V_{jk}
\qquad (k \in [1,n]),
\end{equation*}
which implies that $x = y O$. Since $O$ is an orthonormal matrix, we have that
\begin{equation}\label{DDD2300}
x = yO 
\qquad \text{and} \qquad
y = x O^{T} = (x \cdot O_{1}, \cdots, x \cdot O_{n}).
\end{equation}

With the fact that  $\big( (zO) \cdot O_{2}, \cdots, (zO) \cdot O_{n}) 
= \big( (zO) O^{T} \big)'
= z'$, we estimate
\begin{equation}\begin{aligned}\label{DDD2500}
& D_{x^{1}} \left[ x \cdot O_{1} - \psi \big( x \cdot O_{2}, \cdots , x \cdot O_{n} \big) \right] 
\big|_{x=z O} \\
& \quad = O_{11} - \sum_{2 \leq k \leq n} O_{k1} D_{y^{k}}\psi \big( (zO)  \cdot O_{2}, \cdots , (zO)  \cdot O_{n} \big) \\
& \quad = O_{11} - \sum_{2 \leq k \leq n} O_{k1} D_{y^{k}}\psi (z')
\end{aligned}\end{equation}
in $S$. In view of \eqref{DDD1100} and \eqref{DDD2100}, we have
\begin{equation}\label{DDD2600}
O_{11} = V_{11} \leq - 1 + \tau \leq - \frac{1}{2}
\quad \text{and} \quad
|(O_{12},\cdots,O_{1n})| = |(V_{12},\cdots,V_{1n})| \leq \tau.
\end{equation}
Since $\{ O_{1}, \cdots, O_{n} \}$ is orthonormal, we have that $O_{i} \cdot O_{1} =0$ ($i \in [2,n]$), and so
\begin{equation}\label{DDD2700}
|O_{i1}| \leq \frac{|(O_{i2}, \cdots, O_{in}) \cdot (O_{12}, \cdots, O_{1n})|}{|O_{11}|}
\leq 2\tau
\qquad (i \in [2,n]).
\end{equation}
Since $\tau \in (0,1/4]$, it follows from \eqref{DDD1200} that for any $z' \in B_{8R}'$,
\begin{equation}\label{DDD2900}
O_{11} - \sum_{2 \leq k \leq n} O_{k1} D_{y^{k}}\psi (z') 
\leq - \frac{1}{2} + |(O_{21}, \cdots, O_{n1})| \, \| D_{y' }\psi \|_{L^{\infty}(B_{8R}')}
\leq - \frac{1}{4}.
\end{equation}
So we have from \eqref{DDD2500} and \eqref{DDD2900} that
\begin{equation}\label{DDD3000}
D_{x^{1}} \left[ x \cdot O_{1} - \psi \big( x \cdot O_{2}, \cdots , x \cdot O_{n} \big) \right] \bigg|_{x = zO}
\leq -  \frac{1}{4}
\qquad (z \in S).
\end{equation}
For any $z \in S \in B_{8R}$, we have that $zO \in B_{8R}$ and $(zO)' \in B_{8R}'$. So by the implicit function theorem using  \eqref{DDD3000}, for any $z \in S$, there exist a ball $U_{z} \subset B_{8R}'$ and $C^{1}$-function $\varphi_{z}  : U_{z} \to \mathbb{R}$ such that 
\begin{equation}\label{DDD3100}
(zO)' \in U_{z},
\qquad (zO)^{1} = \varphi_{z}  \big( (zO)' \big),
\qquad \{ (\varphi_{z}(w'),w') : w' \in U_{z}\} \subset B_{8R},
\end{equation}
\begin{equation}\label{DDD3200}
0 = (\varphi_{z} (y' ),y' ) \cdot O_{1} - \psi \big( (\varphi_{z}  (y' ),y' ) \cdot O_{2}, \cdots ,  (\varphi_{z}  (y' ),y' ) \cdot O_{n}  \big) 
\quad \text{ in } U_{z},
\end{equation}
and for any $k \in [2,n]$,
\begin{equation}\begin{aligned}\label{DDD3300}
D_{k} \varphi_{z} (y' )
& = - \frac{ D_{x^{k}} \left[ x \cdot O_{1} - \psi \big( x \cdot O_{2}, \cdots , x \cdot O_{n} \big) \right] \big|_{ x = (\varphi_{z} (y' ),y' )}} { D_{x^{1}} \left[ x \cdot O_{1} - \psi \big( x \cdot O_{2}, \cdots , x \cdot O_{n} \big) \right] \big|_{x = (\varphi_{z} (y' ),y' )}}
~  \text{ in } U_{z}.
\end{aligned}\end{equation}
We remark that if $\{ (\varphi_{z}(y' ),y' ) : y'  \in U_{z}\} \not \subset B_{8R}$ then one can choose a smaller ball $U_{z}$ satisfying $\{ (\varphi_{z}(y' ),y' ) : y'  \in U_{z}\}  \subset B_{8R}$. In view of \eqref{DDD1900}, \eqref{DDD2900}, \eqref{DDD3100} and \eqref{DDD3200}, to apply Lemma \ref{AAA1000}, we need to estimate $ \|D_{y' }\varphi_{z} \|_{L^{\infty}(U_{z})}$.

\medskip

Fix $k \in [2,n]$. We next estimate $\| D_{k} \varphi_{z} \|_{L^{\infty}(U_{z})}$. By \eqref{DDD3300}, for any $w' \in U_{z}$,
\begin{equation*}\begin{aligned}
D_{k} \varphi_{z} (w')
& = - \frac{ D_{x^{k}} \left[ x \cdot O_{1} - \psi \big( x \cdot O_{2}, \cdots , x \cdot O_{n} \big) \right] \big|_{x = (\varphi_{z} (w'),w')}} { D_{x^{1}} \left[ x \cdot O_{1} - \psi \big( x \cdot O_{2}, \cdots , x \cdot O_{n} \big) \right] \big|_{x = (\varphi_{z} (w'),w')}} \\
& = - \frac{ O_{1k} - \sum\limits_{2 \leq i \leq n} O_{ik} D_{y^{i}}\psi \big(  (\varphi_{z} (w'),w') \cdot O_{2}, \cdots , (\varphi_{z} (w'),w') \cdot O_{n} \big) }{ O_{11} - \sum\limits_{2 \leq i \leq n} O_{i1} D_{y^{i}}\psi \big(  (\varphi_{z} (w'),w') \cdot O_{2}, \cdots , (\varphi_{z} (w'),w') \cdot O_{n} \big)} \\
& = - \frac{ O_{1k} - \sum\limits_{2 \leq i \leq n} O_{ik} D_{y^{i}}\psi \big( \big[ (\varphi_{z} (w'),w') O^{T} \big]' \big) }{ O_{11} - \sum\limits_{2 \leq i \leq n} O_{i1} D_{y^{i}}\psi \big( \big[ (\varphi_{z} (w'),w') O^{T} \big]' \big) }.
\end{aligned}\end{equation*}
For any $w' \in U_{z}$, we find from \eqref{DDD3100} that $(\varphi_{z}(w'),w') \in B_{8R}$, which implies that $\big[ (\varphi_{z} (w'),w') O^{T} \big]' \in B_{8R}'$. So by Cauchy-Schwarz's inequality, 
\begin{equation*}
\left| O_{1k} - \sum_{2 \leq i \leq n} O_{ik} D_{y^{i}}\psi \big( \big[ (\varphi_{z} (w'),w') O^{T} \big]' \big) \right| 
\leq |O_{1k}| + 
|(O_{2k}, \cdots, O_{nk})| \, \| D_{ y'  } \psi \|_{L^{\infty}(B_{8R}')}.
\end{equation*}
From \eqref{DDD1200} and \eqref{DDD2600}, we have that $|O_{1k}| \leq \tau$ and $\| D_{ y'  } \psi \|_{L^{\infty}(B_{8R}')} \leq \tau$. So by using the above two estimates and \eqref{DDD2900}, we find that 
\begin{equation}\begin{aligned}\label{DDD3700}
\| D_{k} \varphi_{z} \|_{L^{\infty}(U_{z})}
\leq 4 \big( |O_{1k}| + 
|(O_{2k}, \cdots, O_{nk})| \, \| D_{ y'  } \psi \|_{L^{\infty}(B_{8R}')} \big)
\leq 8 \tau.
\end{aligned}\end{equation}
Since $k \in [2,n]$ was arbitrary chosen, we apply Lemma \ref{AAA1000} by comparing \eqref{AAA1220} with \eqref{DDD1900} and \eqref{DDD2900}, \eqref{AAA1240} and \eqref{AAA1250} with \eqref{DDD3100} and \eqref{DDD3700}, and \eqref{AAA1260} and \eqref{DDD3200}. Then there exists $C^{1}$-function $\varphi : B_{R}' \to \mathbb{R}$ such that
\begin{equation*}
\big( \varphi(x' ),x'  \big) \cdot O_{1} - \psi \big( \big( \varphi(x' ),x'  \big) \cdot O_{2}, \cdots , \big( \varphi(x' ),x'  \big) \cdot O_{n} \big) =0
\quad (x'  \in B_{R}'),
\end{equation*}
\begin{equation}\label{DDD3800}
\| D_{x' } \varphi \|_{L^{\infty}(B_{R}')} 
\leq 8 \sqrt{n}\tau,
\qquad
(zO)^{1} = \varphi  \big( (zO)' \big),
\quad (z \in S \cap B_{R}),
\end{equation}
and
\begin{equation*}
\{ (\varphi(x' ),x' ) : x'  \in B_{R} \} \subset B_{8R}.
\end{equation*}
So we find that \eqref{DDD1500}, \eqref{DDD1600} and \eqref{DDD1700} holds from \eqref{DDD2100}.

\medskip

To prove the lemma, it only remains to estimate $[D_{x' }\varphi]_{C^{\gamma}(B_{R}')}$ in \eqref{DDD1800} under the assumption $\psi \in C^{1,\gamma}(B_{8R}')$. We repeat the proof for showing \eqref{BBB9800} in the proof of Lemma \ref{BBB6000}. One can check from  \eqref{DDD1500} that
\begin{equation*}
D_{x^{k}} \varphi (w')
= - \frac{  O_{1k}  - \sum\limits_{2 \leq i \leq n}  O_{ik}   D_{y^{i}}\psi \Big( \big[ (\varphi(w'),w')O^{T}  \big]' \Big)  }{ O_{11} - \sum\limits_{2 \leq i \leq n} O_{i1}  D_{y^{i}}\psi \Big( \big[ (\varphi(w'),w') O^{T} \big]' \Big)  }
\quad \text{ in } B_{R}'.
\end{equation*}
Since $\tau \in (0,1/(8n)]$, for any $w',z' \in B_{R}'$, we have from \eqref{DDD1600} and \eqref{DDD1700} that
\begin{equation}\label{DDD4200}
|(\varphi(w'),w') - (\varphi(z'),z')|
\leq |(8\tau\sqrt{n}|w'-z'|,w'-z')|
\leq 2 |w'-z'|.
\end{equation}
Then for any $w',z' \in B_{R}'$, we have
\begin{equation}\begin{aligned}\label{DDD4300}
|D_{x^{k}}\varphi(w') - D_{x^{k}}\varphi(z')| 
& \leq \bigg| \frac{I}{II}  - \frac{III}{IV} \bigg| \\
& \leq \bigg| \frac{I(IV-II) + II(I-III))}{II \cdot IV} \bigg| \\
& \leq \bigg| \frac{I}{II} \bigg| 
 \bigg| \frac{IV-II}{IV} \bigg| + \bigg| \frac{I-III}{IV} \bigg|,
\end{aligned}\end{equation}
where
\begin{equation*}
I = O_{1k}  - \sum_{2 \leq i \leq n}  O_{ik}  D_{y^{i}}\psi \Big( \big[ (\varphi(w'),w') O^{T} \big]' \Big) ,
\end{equation*}
\begin{equation*}
II = O_{11} - \sum_{2 \leq i \leq n} O_{i1}  D_{y^{i}}\psi \Big( \big[ (\varphi(w'),w') O^{T} \big]' \Big) ,
\end{equation*}
\begin{equation*}
III = O_{1k}  - \sum_{2 \leq i \leq n} O_{ik}   D_{y^{i}}\psi \Big( \big[ (\varphi(z'),z') O^{T} \big]' \Big),
\end{equation*}
and
\begin{equation*}
IV = O_{11}  - \sum_{2 \leq i \leq n} O_{i1}  D_{y^{i}}\psi \Big( \big[ (\varphi(z'),z') O^{T}  \big]' \Big).
\end{equation*}
By \eqref{DDD1700}, we find that $\big[ (\varphi(w'),w')O^{T} \big]' \in B_{8R}'$ for any $w' \in B_{R}'$. So one can check from \eqref{DDD2900} and  \eqref{DDD4200} that
\begin{equation*}\begin{aligned}
\bigg| \frac{IV-II}{IV} \bigg|
& \leq \frac{4 \left| D_{ y'  }\psi \Big( \big[ (\varphi(w'),w')O^{T} \big]' \Big) - D_{ y'  }\psi \Big( \big[ (\varphi(z'),z')O^{T} \big]' \Big) \right|}{ \big\| (-1, D_{ x'  }\psi) \big\|_{L^{\infty}(B_{8R}')} } \\
& \leq \frac{4 [D_{ y'  }\psi]_{C^{\gamma}(B_{8R}')} \left| \big[ (\varphi(w'),w')O^{T} \big]' - (\varphi(z'),z')O^{T} \big]' \right|^{\gamma}}{ \big\| (-1, D_{ x'  }\psi) \big\|_{L^{\infty}(B_{8R}')} } \\
& \leq 8 [D_{ y'  }\psi]_{C^{\gamma}(B_{8R}')} |w'-z'|^{\gamma} ,
\end{aligned}\end{equation*}
and one can easily check from  \eqref{DDD1200}, \eqref{DDD2600} and \eqref{DDD2900} that
\begin{equation*}\begin{aligned}
\bigg| \frac{I}{II} \bigg|
\leq 4 \big( |O_{1k}| + |(O_{2k}, \cdots, O_{nk})| \big \| D_{ x'  }\psi \big \|_{L^{\infty}(B_{8R}')} \big) 
\leq 8 \tau \leq 1.
\end{aligned}\end{equation*}
By combining the above two estimates,
\begin{equation}\begin{aligned}\label{DDD5600}
\bigg| \frac{I}{II} \bigg| \bigg| \frac{IV-II}{IV} \bigg|
\leq 8 [D_{ x'  }\psi]_{C^{\gamma}(B_{8R}')} |w'-z'|^{\gamma} .
\end{aligned}\end{equation}
Similarly,
\begin{equation}\begin{aligned}\label{DDD5700}
\bigg| \frac{I-III}{IV} \bigg|
\leq 8 [D_{ x'  }\psi]_{C^{\gamma}(B_{8R}')} |w'-z'|^{\gamma}.
\end{aligned}\end{equation}
By combining \eqref{DDD4300}, \eqref{DDD5600} and \eqref{DDD5700}, we obtain that
\begin{equation}\label{DDD5800}
|D_{x^{k}}\varphi(w') - D_{x^{k}}\varphi(z')| 
\leq 16 [D_{ x'  }\psi]_{C^{\gamma}(B_{8R}')} |w'-z'|^{\gamma}
\qquad (w',z' \in B_{R}').
\end{equation}
So we discover that the estimate \eqref{DDD1800} holds from \eqref{DDD5800}.
\end{proof}

With Lemma \ref{DDD1000}, we obtain Lemma \ref{EEE1000}.

\begin{lem}\label{EEE1000}
Let $V \in \mathbb{R}^{n \times n}$ be an orthonormal matrix with $\det V > 0$. Let $\{ e_{1}, \cdots, e_{n} \}$ and $\{ V_{1}, \cdots, V_{n}\}$ be the orthonormal bases of $x$-coordinate system and $y$-coordinate system respectively, satisfying
\begin{equation}\label{EEE1100}
|V_{1} + e_{1}| \leq \tau
\quad \text{ for some } \tau \in (0,1/(8n)].
\end{equation}
For $U \subset \mathbb{R}^{n}$, assume that there exists $C^{1}$-function $\psi : B_{8R}' \to \mathbb{R}$ such that
\begin{equation}\label{EEE1300}
U \cap B_{R} = \left \{ \sum\limits_{1 \leq k \leq n} y^{k} V_{k} \in B_{R} : y^{1} > \psi(y') \right \},
\end{equation}
\begin{equation}\label{EEE1400}
|\psi(0')| <R,
\qquad
D_{y'}\psi(0')=0'
\qquad \text{and} \qquad
\| D_{y'}\psi \|_{L^{\infty}(B_{8R}')} \leq \tau.
\end{equation}
Then there exists $C^{1}$-function $\varphi : B_{R}' \to \mathbb{R}$ such that 
\begin{equation}\label{}
U \cap B_{R} 
= \{ (x^{1},x') \in B_{R} : x^{1} < \varphi(x') \},
\end{equation}
with the estimate
\begin{equation}\label{EEE1700}
\| D_{x'} \varphi \|_{L^{\infty}(B_{R}')} 
\leq 8 \sqrt{n}\tau .
\end{equation}
Moreover, if $\psi \in C^{1,\gamma}(B_{8R}')$ then $\varphi \in C^{1,\gamma}(B_{R}')$ with the estimate
\begin{equation}\label{EEE1800}
[D_{x'}\varphi]_{C^{\gamma}(B_{R}')} 
\leq 16 [D_{y'}\psi]_{C^{\gamma}(B_{8R}')}.
\end{equation}
\end{lem}

\begin{rem}
In Lemma \ref{EEE1000}, the set $U \cap B_{R} = \left \{ (y^{1},y') \in B_{R} : y^{1} > \psi(y') \right \} $ in $y$-coordinate system is represented by $U \cap B_{R} = \{ (x^{1},x') \in B_{R} : x^{1} < \varphi(x')\}$ with respect to $x$-coordinate system.
\end{rem}

\begin{proof}
With \eqref{EEE1100} and \eqref{EEE1400}, by Lemma \ref{DDD1000}, there exists $C^{1}$-function $\varphi : B_{R}' \to \mathbb{R}$ satisfying
\begin{equation}\label{EEE2400}
\big( \varphi(x'),x' \big) \cdot V_{1} - \psi \big( \big( \varphi(x'),x' \big) \cdot V_{2}, \cdots , \big( \varphi(x'),x' \big) \cdot V_{n} \big) =0
\qquad (x' \in B_{R}'),
\end{equation}
\begin{equation}\begin{aligned}\label{EEE2700}
\| D_{x'} \varphi \|_{L^{\infty}(B_{R}')} 
\leq 8 \tau \sqrt{n},
\qquad
(zV)^{1} = \varphi \big( (zV)' \big)
\quad 
(z \in S \cap B_{8R}),
\end{aligned}\end{equation}
and
\begin{equation}\label{EEE2600}
\big\{ \big( \varphi(x'),x' \big) : x' \in B_{R}' \big\} \subset B_{8R}.
\end{equation}
Moreover, if $\psi \in C^{1,\gamma}(B_{8R}')$ then $\varphi \in C^{1,\gamma}(B_{R}')$ with the estimate
\begin{equation}\label{EEE2650}
[D_{x'}\varphi]_{C^{\gamma}(B_{R}')} 
\leq 16 [D_{y'}\psi]_{C^{\gamma}(B_{8R}')}.
\end{equation}
By \eqref{EEE1300}, \eqref{EEE2400}, \eqref{EEE2600} and \eqref{EEE2650}, to apply Lemma \ref{PPP7000}, set
\begin{equation*}
U^{+} = \left \{ \sum\limits_{1 \leq k \leq n} y^{k} V_{k} \in B_{R} : y^{1} > \psi(y') \right \},
\quad 
V^{+} = \{ (x^{1},x') \in B_{R} : x^{1} > \varphi(x')\},
\end{equation*}
\begin{equation*}
U^{-} = \left \{ \sum\limits_{1 \leq k \leq n} y^{k} V_{k} \in B_{R} : y^{1} < \psi(y') \right \},
\quad
V^{-} = \{ (x^{1},x')  \in B_{R} : x^{1} < \varphi(x')\},
\end{equation*}
and
\begin{equation*}
U^{0} = \left \{ \sum\limits_{1 \leq k \leq n} y^{k} V_{k} \in B_{R} : y^{1} = \psi(y') \right \},
\quad
V^{0} = \{ (x^{1},x') \in B_{R} : x^{1} = \varphi(x')\},
\end{equation*}
and need to check that
\begin{equation}\label{EEE3000}
U^{+} \cap V^{-} \not = \emptyset
\qquad \text{and} \qquad
U^{-} \cap V^{+} \not = \emptyset.
\end{equation}
Since the proof for that $U^{+} \cap V^{-} \not = \emptyset$ is similar to that of $U^{-} \cap V^{+} \not = \emptyset$, we only prove $U^{-} \cap V^{+} \not = \emptyset$. 

Since $V$ is orthonormal, let $O = V^{T} V = Id_{n}$. Since $(\psi(0'),0') \in S \cap B_{R}$, we have from \eqref{EEE2700} that
\begin{equation}\label{EEE3100}
\bar{z} : = (\psi(0'),0') V = \psi(0') V_{1} \in B_{R}
\quad \Longrightarrow \quad
\bar{z}^{1}  = \varphi( \bar{z}')
\quad \text{and} \quad
\bar{z} \in V^{0}. 
\end{equation}
So by the definition of $V^{+} \subset B_{R}$, one can choose $\epsilon_{0} \in (0,8R)$ so that 
\begin{equation}\label{EEE3200}
\epsilon \in (0,\epsilon_{0}]
\qquad \Longrightarrow \qquad
\epsilon e_{1} + \bar{z} \in V^{+}.
\end{equation}
Then we claim that
\begin{equation}\label{EEE3300}
\epsilon e_{1} + \bar{z} 
\in U^{-}
\quad \text{for some } \epsilon \in (0,\epsilon_{0}].
\end{equation}
Since $e_{1} = (e_{1} \cdot V_{1}) V_{1} + \cdots + (e_{1} \cdot V_{n}) V_{n}= (V_{11}) V_{1} + \cdots + (V_{n1}) V_{n}$, we have that
\begin{equation*}
\epsilon e_{1} + \bar{z} = \epsilon e_{1} + \psi(0') V_{1}
=  \sum_{1 \leq k \leq n} [\epsilon V_{k1} + \delta_{k1} \psi(0') ]V_{k}.
\end{equation*}
So to prove the claim \eqref{EEE3300}, by the definition of $U^{-}$, it is suffice to show that
\begin{equation}\label{EEE3400}
\epsilon V_{11} + \psi(0') < 
\psi \big( \epsilon (V_{21}, \cdots, V_{n1}) \big)
\Longleftrightarrow
\epsilon V_{11} < 
\psi \big( \epsilon (V_{21}, \cdots, V_{n1}) \big)
- \psi(0').
\end{equation}
Since $\epsilon_{0} \in (0,8R)$, we have that $ \epsilon(V_{21},\cdots,V_{n1}) \in B_{8R}'$. Thus
\begin{equation}\label{EEE3500}\begin{aligned}
\psi \big( \epsilon (V_{21},\cdots, V_{n1}) \big) - \psi (0') = D_{y'}\psi \big( \bar{\epsilon}(V_{21},\cdots,  V_{n1})  \big)  \cdot \epsilon (V_{21},\cdots,V_{n1}) 
\end{aligned}\end{equation}
for some $\bar{\epsilon} \in (0,\epsilon]$. By \eqref{EEE1100}, we have that $V_{11} <0$. Since $\psi$ is $C^{1}$-function and $D_{y'}\psi (0') = 0'$ in \eqref{EEE1400},
we find from \eqref{EEE3500} that
\begin{equation*}
\epsilon^{-1} \big[ \psi \big( \epsilon (V_{21},\cdots, V_{n1}) \big) - \psi (0') \big]
\to 0 \text{ as } \epsilon \to 0,
\end{equation*}
and there exists a small $\epsilon \in (0,\epsilon_{0}]$ such that \eqref{EEE3400} holds. So the claim \eqref{EEE3300} holds. From \eqref{EEE3200} and \eqref{EEE3300}, we obtain that
$V^{+} \cap U^{-} \not = \emptyset$. Similarly, one can also prove that $V^{-} \cap U^{+} \not = \emptyset$. Thus \eqref{EEE3000} holds, and the lemma follows by applying Lemma \ref{PPP7000} with \eqref{EEE2400}, \eqref{EEE2700}, \eqref{EEE2600}, \eqref{EEE2650} and \eqref{EEE3000} for $V$, $V$ and $Id_{n} =  V^{T}V$ instead of $V$, $W$ and $O=W^{T}V$.
\end{proof}

\section{Coordinate system in composite materials}

Our proof is based on the fact that for two disjoint Reifenberg flat domains $U_{1}$ and $U_{2}$, the (outward) normals on $\partial U_{1} \cap B_{R}$ and $\partial U_{2} \cap B_{R}$ are almost opposite if the radius $R>0$ is sufficiently small. This result obtained in \cite{JK_CV1}, and we start this section with the following definition of Reifenberg flat domains which appears in \cite{JK_CV1}. For $(\delta,R)$-Reifenberg flat domains, for any boundary point and for any scale less that $R$, there exists a coordinate system such that the boundaries trapped between two narrow hyperplanes distance less than $2 \delta R$.

\begin{defn}[Reifenberg flat domain]\label{CTS4000}
$U$ is a $(\delta,R)$-Reifenberg flat domain if for any $y \in \partial U$ and any $r \in (0,R]$, there exists a coordinate system such that
\begin{equation*}
\{ x \in Q_{r}(y) : x^{1} > y^{1} + \delta r\} 
\subset Q_{r}(y) \cap U 
\subset \{ x \in Q_{r}(y) : x^{1} > y^{1} - \delta r\}.
\end{equation*}
\end{defn}

\medskip

We use the following Lemma \ref{CTS6000} for handling the normal vectors on the boundaries of two disjoint Reifenberg flat domains. 

\begin{lem}\cite[Lemma 2.4]{JK_CV1}                                                                                                                                                                                                                                                                                                                                                                                                                                                                                                                                                                                                                                                                                                                                                                                                                                                                                                                                                                                                                                                                                                                                                                                                                                                                                                                                                                                                                                                                                                                                                                                                                                                                                                                                                                                                                                                                                                                                                                                                                                                                                                                                                                                                                                                                                                                                                                                                                                                                                                                                                                                                                                                                                                                                                                                                                                                                                    \label{CTS6000}
There exists $\delta_{1}(n) \in (0,1/16)$ such that the following holds. Suppose that $U_{k}$ and $U_{l}$ are disjoint $(\delta,5R)$-Reifenberg flat domains with $\delta \in (0,\delta_{1}]$. For any $r \in (0,R]$, if $P \in \partial U_{k}$ and $Q \in \partial U_{l}$ satisfy $|P-Q| < r$ then
\begin{equation*}
|\vec{n}_{P,5r} + \vec{n}_{Q,5r}| \leq \frac{\delta^{\frac{1}{4}}}{2},
\end{equation*}
where $\vec{n}_{P,5r}$ and $\vec{n}_{Q,5r}$ are the normal vectors at $P \in \partial U_{k}$ and $Q \in \partial U_{l}$ with the radius $5r$.
\end{lem}

With Lemma \ref{CTS7000}, our problem can be turned to a simpler problem. In view of Lemma \ref{CTS7000}, we only need to consider the case that at most two disjoint Reifenberg flat domains intersect a small ball.

\begin{lem}\cite[Lemma 2.5]{JK_CV1}\label{CTS7000}
For $\delta_{1}(n) \in (0,1/16)$ in Lemma \ref{CTS6000}, if $U_{1},U_{2},U_{3}$ are mutually disjoint nonempty $(\delta,10R)$-Reifenberg flat domains with $\delta \in (0,\delta_{1}]$ then
\begin{equation*}
U_{k} \cap B_{R} = \emptyset 
\text{ for some } k \in [1,3].
\end{equation*}
\end{lem}

\medskip

We use the following $C^{1,\gamma}$-class domains which was defined in Definition \ref{intro_domain}.

\introdomain*

\begin{rem}
If $U$ is $(C^{1,\gamma},R,\theta)$-domain then $\mathbb{R}^{n} \backslash \overline{U}$ is also $(C^{1,\gamma},R,\theta)$-domain.
\end{rem}

To use Lemma \ref{CTS6000}, we need to check that $(C^{1,\gamma},R,\theta)$-domains are also Reifenberg flat domains, which will be done in Lemma \ref{CTS8000} and Lemma \ref{GGG1000}.

\begin{lem}\label{CTS8000}
For any $\tau \in (0,1]$, there exists $R_{1} = R_{1}(n,\gamma,\theta,\tau) \in (0,1]$ such that the following holds for any $R \in (0,R_{1}]$. If $U$ is $(C^{1,\gamma},8R,\theta)$-domain with $\partial U \cap B_{R} \not = \emptyset$ then there exist an orthonormal matrix $V$ with $\det V>0$ and  $C^{1,\gamma}$-function $\varphi : B_{R}' \to \mathbb{R}$ such that
\begin{equation}\label{CTS8400}
U \cap B_{R} = \left \{ \sum\limits_{1 \leq k \leq n} y^{k} V_{k} \in B_{R} : y^{1} > \varphi(y') \right\}
\end{equation}
with the estimate
\begin{equation}\label{CTS8500}
|\varphi(0')| < R,
~ ~
D_{y'}\varphi(0')=0',
~ ~
\| D_{y'}\varphi \|_{L^{\infty}(B_{R}')} 
\leq \tau
~ ~ \text{and} ~ ~
[D_{y'}\varphi]_{C^{\gamma}(B_{R}')} 
\leq 18 n \theta.
\end{equation}
Also if $\mathbf{0} \in \partial U$ then $\varphi(0')=0$. 
\end{lem}

\begin{proof}
We take $R_{*} = R_{*}(n,\gamma,\theta) \in (0,R_{1}]$ so that 
\begin{equation}\label{CTS8600}
n \theta (16R_{*})^{\gamma} \leq 1/4.
\end{equation}
Since $U$ is $(C^{1,\gamma},8R,\theta)$-domain, there exists $C^{1,\gamma}$-function $B_{8R}' \to \mathbb{R}$ such that
\begin{equation}\label{CTS8700}
U \cap B_{R} = \{ (x^{1},x') \in B_{R} : x^{1} > \psi(x') \}
\qquad \text{and} \qquad
\| \psi \|_{C^{1,\gamma}(B_{R}')} \leq \theta.
\end{equation}
Recall from \eqref{CTS8600}, we find that $n [ D_{x'}\psi]_{C^{\gamma}(B_{8R}')} (16R)^{\gamma} \leq n\theta(16R_{*})^{\gamma} \leq 1/4$. So from the fact that $\partial U \cap B_{R} \not = \emptyset$, we obtain that
\begin{equation}\label{CTS8800}
\{ ( \psi(x'),x') \in B_{8R} : x \in B_{8R}' \} \cap B_{R} \not = \emptyset 
\text{ and }
n [ D_{x'}\psi]_{C^{\gamma}(B_{8R}')} (16R)^{\gamma} \leq 1/4.
\end{equation}
Apply Lemma \ref{CCC1000} by comparing  \eqref{CTS8700} and \eqref{CTS8800} with \eqref{CCC1200} and \eqref{CCC1300}. Then there exist an orthonormal matrix $ V \in \mathbb{R}^{n \times n}$ with $\det V >0$ and  $C^{1,\gamma}$-function $\varphi : B_{R}' \to \mathbb{R}$ such that
\begin{equation*}\label{}
U \cap B_{R} = \left \{ \sum\limits_{1 \leq k \leq n} y^{k} V_{k} \in B_{R} : y^{1} > \varphi(y') \right\}
\end{equation*}
with the estimate
\begin{equation*}
|\varphi(0')| < R,
\quad
D_{y'}\varphi(0')=0'
\quad \text{and} \quad
[D_{y'}\varphi]_{C^{\gamma}(B_{R}')} 
\leq 18n [D_{x'}\psi]_{C^{\gamma}(B_{8R}')}
\leq 18n\theta.
\end{equation*}
So there exists $R_{1} = R_{1}(n,\gamma,\theta,\tau) \in (0,R_{*}]$ such that  \eqref{CTS8400} and \eqref{CTS8500} holds. Also if $\mathbf{0} \in \partial U$ then $\psi(0')=0$ in \eqref{CTS8700}, and so Lemma \ref{CCC1000} gives that $\varphi(0')=0$. 
\end{proof}

In view of Lemma \ref{CTS8000}, for $(C^{1,\gamma},R,\theta)$-domains, there exists a coordinate system such that the boundary is almost flat, and which will be obtained Lemma \ref{GGG1000}.

\begin{lem}\label{GGG1000}
For any $\delta \in (0,1/8]$, there exists $ R_{2}(n,\gamma,\theta,\delta) \in (0,R_{1}(n,\gamma,\theta,2)]$ such that if $U$ is $(C^{1,\gamma},8R,\theta)$-domain with $R \in (0,R_{2}]$ then $U$ is $(\delta,R)$-Reifenberg flat domain. Set $R_{3}(n,\gamma,\theta) = R_{2}(n,\gamma,\theta,\delta_{1})/10$ for $\delta_{1}(n)$ in Lemma \ref{CTS6000}.
\end{lem}

\begin{proof}
Let $\mathbf{0} \in \partial U$.
Since $U$ is $(C^{1,\gamma},8R,\theta)$-domain, by Lemma \ref{CTS8000}, there exist an orthonormal matrix $V$ with $\det V>0$ and  $C^{1,\gamma}$-function $\varphi : B_{2R}' \to \mathbb{R}$ such that
\begin{equation}\label{GGG3400}
U \cap B_{R} = \left \{ \sum\limits_{1 \leq k \leq n} y^{k} V_{k} \in B_{R} : y^{1} > \varphi(y') \right\}
\end{equation}
with the estimate
\begin{equation}\label{GGG3600}
\varphi(0')=0,
\quad
D_{y'}\varphi(0')=0'
\quad \text{and} \quad
[D_{y'}\varphi]_{C^{\gamma}(B_{R}')} 
\leq 18 n \theta.
\end{equation}
Choose $R_{2} = R_{2}(n,\gamma,\theta,\delta) \in (0,R_{1}]$ so that $36 n \theta (4R_{2})^{\gamma} \leq \delta$. Then 
 \eqref{GGG3600} gives that
\begin{equation*}
\sup_{y' \in B_{R}'} |\varphi(y')| 
= \sup_{y' \in B_{R}'} |\varphi(y') - \varphi(0') - D_{y'}\varphi(0') \cdot y'| 
\leq [D_{y'}\varphi]_{C^{\gamma}(B_{R}')} (2R)^{\gamma} < \delta R.
\end{equation*} 
So the boundary $\partial U$ in $B_{R}$ is trapped between two narrow hyperplanes with distance less than $2\delta R$. Since the boundary point can be arbitrary chosen, $U$ is $(\delta,R)$-Reifenberg flat domain by Definition \ref{CTS4000}.
\end{proof}

If two disjoint $(C^{1,\gamma},R,\theta)$-domains $U_{1}$ and $U_{2}$ intersect a ball, then one can find a coordinate system such that $\partial U_{1}$ and $\partial U_{2}$ become graph, and $\partial U_{1}$ is almost flat in that ball.

\begin{lem}\label{HHH1000}
For any $\tau \in (0,1]$, there exists $R_{4} = R_{4}(n,\gamma,\theta,\tau) \in (0,\min \{R_{1}, R_{3}\} ]$ such that the following holds for any $R \in (0,R_{4}]$. Suppose that $U_{1}$ and $U_{2}$ are disjoint $(C^{1,\gamma},80R,\theta)$-domains with $\partial U_{1} \cap B_{R} \not = \emptyset$ and $\partial U_{2} \cap B_{R} \not = \emptyset$. Also assume that there exists  $C^{1,\gamma}$-function $\varphi_{1} : B_{R}' \to \mathbb{R}$ such that
\begin{equation}\label{HHH1300}
U_{1} \cap B_{R} = \left \{ (x^{1},x') \in B_{R} : x^{1} > \varphi_{1}(x') \right\}
\qquad \text{and} \qquad
|\varphi_{1}(0')| < R
\end{equation}
with the estimate
\begin{equation*}
D_{x'}\varphi_{1}(0')=0',
\qquad
\| D_{x'}\varphi_{1} \|_{L^{\infty}(B_{R}')} 
\leq 2\sqrt{n}
\qquad \text{and} \qquad
[D_{x'}\varphi_{1}]_{C^{\gamma}(B_{R}')} 
\leq 18 n \theta.
\end{equation*}
Then there exists $C^{1,\gamma}$-function $\varphi_{2} : B_{R}' \to \mathbb{R}$ such that
\begin{equation*}
U_{2} \cap B_{R} 
 = \{ (x^{1},x') \in B_{R} : x^{1} < \varphi_{2}(x')\}
\end{equation*}
with the estimates
\begin{equation*}
\|D_{x'} \varphi_{1} \|_{L^{\infty}(B_{R}')}, \|D_{x'} \varphi_{2} \|_{L^{\infty}(B_{R}')} \leq  \tau
~ \text{ and } ~
[D_{x'} \varphi_{1}]_{C^{\gamma}(B_{R}')},
[D_{x'} \varphi_{2}]_{C^{\gamma}(B_{R}')} \leq 288n\theta.
\end{equation*}
\end{lem}

\begin{proof}
Since $U_{1}$ and $U_{2}$ are $(C^{1,\gamma},80R,\theta)$-domain, by Lemma \ref{GGG1000} there exists a constant  $R_{*} = R_{*}(n,\gamma,\theta,\tau) \in (0,R_{3}]$ such that 
\begin{equation}\label{HHH2100}
R \in (0,R_{*}]
\quad \Longrightarrow \quad
U_{1}, U_{2} \text{ are } ( \min \{ \delta_{1}, [\tau/(8n)]^{4} \}, 10R)\text{-Reifenberg flat.}
\end{equation}
By \eqref{HHH2100} and Lemma \ref{CTS6000},
\begin{equation}\label{HHH2200}
P \in \partial U_{1}, 
\quad
Q \in \partial U_{2} 
\quad \text{ and } \quad
|P-Q| < 2R
\quad \Longrightarrow \quad
|\vec{n}_{P} + \vec{n}_{Q}| \leq \tau/(8n),
\end{equation}
where $\vec{n}_{P}$ and $\vec{n}_{Q}$ are the normal vectors at $P \in \partial U_{1}$ and $Q \in \partial U_{2}$. We remark that for $(C^{1,\gamma},R,\theta)$-class domains, the normals do not depend on the size and we use the notation $\vec{n}_{P}$ and $\vec{n}_{Q}$ instead of $\vec{n}_{P,5r}$ and $\vec{n}_{Q,5r}$.

\medskip

By the assumption of the lemma
\begin{equation*}
|\varphi_{1}(0')| < R,
~
D_{x'}\varphi_{1}(0')=0',
~
\| D_{x'}\varphi_{1} \|_{L^{\infty}(B_{R}')} 
\leq 2 \sqrt{n}
~ \text{and} ~
[D_{x'}\varphi_{1}]_{C^{\gamma}(B_{R}')} 
\leq 18 n \theta.
\end{equation*}
So there exists $R_{**} = R_{**}(n,\gamma,\theta,\tau) \in (0,R_{*}]$ such that
\begin{equation}\label{HHH3600}
\| D_{x'} \varphi_{1} \|_{L^{\infty}(B_{R}')} \leq \tau.
\end{equation}
From \eqref{HHH1300} and the fact that $D_{x'}\varphi_{1}(0')=0'$, we find that
\begin{equation}\label{HHH3700}
-e_{1} \text{ is the normal vector of } U_{1} \text{ at } 
(\varphi_{1}(0'),0') \in \partial U_{1} \cap B_{R}.
\end{equation}

\medskip

With the fact that $U_{2}$ is a $(C^{1,\gamma},80R,\theta)$-domain and $\partial U_{2} \cap B_{R} \not = \emptyset$, we apply Lemma \ref{CTS8000} to $U_{2}$. Then there exists $R_{***} = R_{***}(n,\gamma,\theta) \in (0,R_{3}]$ such that if $R \in (0,R_{***}]$, then there exist an orthonormal matrix $V \in \mathbb{R}^{n \times n}$ with $\det V>0$ and $C^{1,\gamma}$-function $\psi : B_{R}' \to \mathbb{R}$ such that 
\begin{equation}\label{HHH4200}
U_{2} \cap B_{8R} = \left \{ \sum\limits_{1 \leq k \leq n} y^{k} V_{k} \in B_{8R} : y^{1} > \psi(y') \right\}
\qquad \text{and} \qquad
|\psi(0')| < R,
\end{equation}
with the estimate
\begin{equation}\label{HHH4300}
D_{y'}\psi(0')=0'
\qquad \text{and} \qquad
[D_{y'}\psi]_{C^{\gamma}(B_{8R}')} \leq 18n\theta.
\end{equation}
So by \eqref{HHH4300}, there exists $R_{4} = R_{4}(n,\gamma,\theta,\tau) \in (0, \min \{ R_{*},R_{**},R_{***} \}]$ such that $R \in (0,R_{4}]$ implies that $ \| D_{y'}\psi \|_{L^{\infty}(B_{R}')} \leq \tau / (8n)$. Thus
\begin{equation}\label{HHH4500}
D_{y'}\psi(0')=0',
\quad
\| D_{y'}\psi \|_{L^{\infty}(B_{8R}')} \leq \tau / (8n)
\quad \text{and} \quad
[D_{y'}\psi]_{C^{\gamma}(B_{8R}')} \leq 18 n\theta.
\end{equation}
By \eqref{HHH4200} and \eqref{HHH4500}, 
\begin{equation}\label{HHH4600}
-V_{1} \text{ is the normal vector of } U_{2} \text{ at } \psi(0') V_{1} \in \partial U_{2} \cap B_{R}.
\end{equation}

\medskip

By recalling from \eqref{HHH3700} and \eqref{HHH4600} that $-e_{1}$ is the normal vector of $U_{1}$ at $
(\varphi_{1}(0'),0') \in \partial U_{1} \cap B_{R}$ and $-V_{1}$ is the normal vector of $U_{2}$ at $\psi(0') V_{1} \in \partial U_{2} \cap B_{R}$, and it follows from \eqref{HHH2200} that
\begin{equation}\label{HHH5100}
|e_{1} + V_{1}| \leq \tau/(8n)
\qquad \text{and} \qquad
\tau/(8n) \leq 1/(8n).
\end{equation}
By comparing \eqref{HHH5100}, \eqref{HHH4200} and \eqref{HHH4500} with \eqref{EEE1100}, \eqref{EEE1300}, \eqref{EEE1400} and \eqref{EEE1800}, we apply Lemma \ref{EEE1000} to $\psi$ for $\tau/(8n)$ instead of $\tau$. 
Then there exists $C^{1,\gamma}$-function $\varphi_{2} : B_{R}' \to \mathbb{R}$ such that 
\begin{equation}\label{HHH5400}
U_{2} \cap B_{R} 
= \{ (x^{1},x') \in B_{R} : x^{1} < \varphi_{2} (x') \},
\end{equation}
with the estimate
\begin{equation}\label{HHH5500}
\| D_{x'} \varphi_{2} \|_{L^{\infty}(B_{R}')} 
\leq \tau
\quad \text{ and } \quad
[D_{x'}\varphi_{2}]_{C^{\gamma}(B_{R}')} 
\leq 16 [D_{y'}\psi]_{C^{\gamma}(B_{8R}')}
\leq 288n \theta.
\end{equation}
So the lemma follows from  \eqref{HHH3600}, \eqref{HHH5400} and \eqref{HHH5500}.
\end{proof}

We considered only two disjoint 
$(C^{1,\gamma},R,\theta)$-domains in Lemma \ref{HHH1000}. To handle the general case which appears in composite materials, we use Definition \ref{intro_composite domain} which already mentioned in the introduction.

\introcompositedomain*

\medskip

The following Lemma \ref{HHH7000} is a simple application of Lemma \ref{CTS7000} to composite $(C^{1,\gamma},80R,\theta)$-domains.

\begin{lem}\label{HHH7000}
For any  $R \in (0,R_{3}]$ with $R_{3}(n,\gamma,\theta) \in (0,1]$ in Lemma \ref{GGG1000}, suppose that $U$ is composite $(C^{1,\gamma},80R,\theta)$-domain with subdomains $\{  U_{1}, \cdots, U_{K} \}$ and
\begin{equation*}
U_{k} \cap B_{R} \not = \emptyset,
\qquad
U_{l} \cap B_{R} \not = \emptyset
\qquad \text{and} \qquad
U_{m} \cap B_{R} \not = \emptyset
\qquad (k,l,m \in \{ 0,\cdots, K\} ).
\end{equation*}
If $U_{k} \cap U_{l} = \emptyset$ and $U_{k} \cap U_{m} = \emptyset$ then $U_{l} \cap U_{m} \not = \emptyset$.
\end{lem}

\begin{proof}
Since $10 R \in (0,10 R_{3}] \subset (0,R_{2}]$ for $R_{2}(n,\gamma,\theta,\delta_{1})$ in Lemma \ref{GGG1000}, $U$ is composite $(C^{1,\gamma},80R,\theta)$-domain with subdomains $\{ U_{1}, \cdots, U_{K} \}$. So by Lemma \ref{GGG1000}, $\{ U_{0} : =U, U_{1}, \cdots, U_{K} \}$ are $(\delta_{1},10R)$-Reifenberg domains. Assume that $U_{l} \cap U_{m} = \emptyset$. Since $U_{k} \cap U_{l} = \emptyset$ and $U_{k} \cap U_{m} = \emptyset$,  Lemma \ref{CTS7000} gives that
\begin{equation*}
U_{k} \cap B_{R} = \emptyset
\quad \text{or} \quad
U_{l} \cap B_{R} = \emptyset
\quad \text{or} \quad
U_{m} \cap B_{R} = \emptyset,
\end{equation*}
which contradicts the assumption of the lemma. So we find that $U_{l} \cap U_{m} \not = \emptyset$.
\end{proof}

$W_{j}$ in \eqref{intro_W_{j}} represents an individual component and $\bigcup\limits_{U_{i} \in S, U_{i} \subsetneq U_{j}} U_{i}$ represents the union of the components inside  $W_{j}$. In fact, we have the following lemma holds.

\begin{lem}\label{JJJ1000}
Assume that $U$ is composite $(C^{1,\gamma},R,\theta)$-domain with subdomains $\{ U_{1}, \cdots, U_{K} \}$. Set $S = \{ U_{0}:=U, U_{1}, \cdots, U_{K} \}$ and
\begin{equation}\label{JJJ1100}
W_{j} = U_{j} \setminus \left( \bigcup_{U_{i} \in S, U_{i} \subsetneq U_{j} } U_{i} \right)
\qquad (j \in \{ 0, \cdots, K \} ).
\end{equation}
Then $\{ W_{0}, \cdots, W_{K} \}$ are mutually disjoint and 
\begin{equation}\label{JJJ1200}
U = W_{0} \sqcup W_{1} \sqcup \cdots \sqcup W_{K}.
\end{equation}
\end{lem}

\begin{proof}

[Step 1] We prove that $\{ W_{0}, \cdots, W_{K} \}$ are mutually disjoint. Suppose not. Then there exists $W_{k}$ and $W_{j}$ such that 
\begin{equation}\label{JJJ2200}
W_{k} \cap W_{j} \not = \emptyset,
\end{equation}
which implies that $ U_{k} \cap U_{j} \not = \emptyset $. So by \eqref{HHH6500}, $ U_{k} \subsetneq U_{j} $ or $ U_{j} \subsetneq U_{k} $ holds. Without loss of generality, we assume that $ U_{j} \subsetneq U_{k} $. Then from the definition of $ W_{k} = U_{k} \setminus \left( \bigcup_{U_{i} \in S, ~ U_{i} \subsetneq U_{k} } U_{i} \right) $, we have that $W_{k} \cap U_{j} = \emptyset$, which contradicts \eqref{JJJ2200}. So we find that $\{ W_{0}, \cdots, W_{K} \}$ are mutually disjoint.

\medskip

[Step 2] We prove that $U = W_{0} \cup \cdots \cup W_{K}$ in \eqref{JJJ1200}. Fix $ y \in U = U_{0}$, and let
\begin{equation}\label{JJJ3100}
W = \bigcap_{ U_{i} \in S , ~ y \in U_{i}}  U_{i}.
\end{equation}
From \eqref{HHH6500}, if $U_{i} \owns y$, $U_{j} \owns y $ and $U_{i},U_{j} \in S$ then $U_{i} \subsetneq U_{j}$ or $U_{j} \subsetneq U_{i}$ holds. Thus
\begin{equation*}
U_{i} \owns y, ~ U_{j} \owns y 
\quad \text{and} \quad 
U_{i},U_{j} \in S
\quad \Longrightarrow \quad
U_{i} \cap U_{j} = U_{i}
\quad \text{or} \quad
U_{i} \cap U_{j} = U_{j}.
\end{equation*}
So by the fact that the number of elements in $S$ is finite, one can easily prove that
\begin{equation}\label{JJJ3200}
W = U_{k} \in S
\quad \text{for some } k \in \{ 0,\cdots, K \}.
\end{equation}
With the chosen $U_{k}$ in \eqref{JJJ3200}, we claim that 
\begin{equation}\label{JJJ3300}
y \in W_{k}.
\end{equation}
Suppose not. Then $y \not \in W_{k}$. Since $U_{k} = W \owns y$ and $W_{k} = U_{k} \setminus \left( \bigcup_{U_{i} \subsetneq U_{k}, ~ U_{i} \in S } U_{i} \right)$, there exists $j \in \{ 0,\cdots, K \}$ such that $U_{j} \subsetneq U_{k}$ and $y \in U_{j} \in S$. Then \eqref{JJJ3200} and that $U_{j} \subsetneq U_{k}$ give that $U_{j} \supsetneq U_{k} = W$. On the other-hand, from the definition of $W$ in \eqref{JJJ3100} and the fact that $y \in U_{j} \in S$, we have that $W \subset U_{j}$, and a contradiction occurs. So the claim \eqref{JJJ3300} holds. Since $y \in U$ was arbitrary chosen, for any $y \in U$ there exists $W_{k} \in \{ W_{0}, \cdots, W_{K} \}$ with $W_{k} \owns y$. This proves \eqref{JJJ1200}.
\end{proof}
Let $U$ be composite $(C^{1,\gamma},8R,\theta)$-domain with subdomains $\{ U_{1}, \cdots, U_{K} \}$ and $S = \{ U_{0}:= U, U_{1}, \cdots, U_{K}, U_{K+1} : = \emptyset \}$. In Lemma \ref{KKK3000} and Lemma \ref{KKK5000}, we decompose $U_{j} \in S$ with $W_{j}$ in \eqref{JJJ1100} and the elements in $S$. For this decomposition,  we use Lemma \ref{KKK3000} when $U_{j} \not \supset U \cap B_{R}$ and Lemma \ref{KKK5000} when $U_{j} \supset U \cap B_{R}$. We remark that the center of the ball in Lemma \ref{KKK3000} and Lemma \ref{KKK5000} are chosen as the origin but it can be chosen as any point in $U$ by using the translation.

\begin{lem}\label{KKK3000}
For any $R \in (0,R_{3}]$ with $R_{3}(n,\gamma,\theta)$ in Lemma \ref{GGG1000}, suppose that $U$ is composite $(C^{1,\gamma},80R,\theta)$-domain with subdomains $\{ U_{1}, \cdots, U_{K} \}$. Also for $S = \{ U_{0}:= U,  U_{1}, \cdots, U_{K}, U_{K+1} : = \emptyset \}$, suppose that 
\begin{equation*}
U_{j} \cap B_{R} \not = \emptyset
\quad \text{and} \quad
U_{j} \not \supset U \cap B_{R}
\quad \text{for some} \quad U_{j} \in S. 
\end{equation*}
Set
\begin{equation}\label{KKK3050}
U_{k} = \bigcup_{ U_{i} \in S, ~ U_{i} \subsetneq U_{j}, ~ U_{i} \cap B_{R} \not = \emptyset } U_{i}.
\end{equation} 
Then 
\begin{equation}\label{KKK3100}
U_{k} \in S
\qquad \text{and} \qquad
U_{k} \subsetneq U_{j}.
\end{equation}
In addition, for 
\begin{equation}\label{KKK3300}
W_{j} = U_{j} \setminus \left( \bigcup_{U_{i} \in S, ~ U_{i} \subsetneq U_{j} } U_{i} \right)
\qquad (j \in \{ 0,\cdots, K \} ),
\end{equation}
we have
\begin{equation}\label{KKK3400}
U_{j} \cap B_{R} = (W_{j} \sqcup U_{k}) \cap B_{R}.
\end{equation}
\end{lem}

\begin{proof}
If $\{ U_{i} \in S :  U_{i} \subsetneq U_{j} \text{ and }   U_{i} \cap B_{R} \not = \emptyset \} = \emptyset$, then the lemma holds by taking $U_{k} = \emptyset$. So we assume that
\begin{equation*}
\{ U_{i} \in S :  U_{i} \subsetneq U_{j} \text{ and }  U_{i} \cap B_{R} \not = \emptyset \} \not = \emptyset.
\end{equation*}

\medskip
We first prove \eqref{KKK3100}. Since $U_{j} \in S$ and $U_{j} \not \supset  U \cap B_{R}$, we have
\begin{equation}\label{KKK3600}
B_{R} \cap (\mathbb{R}^{n} \setminus \overline{U_{j}}) \not = \emptyset
\quad \text{and} \quad
\mathbb{R}^{n} \setminus \overline{U_{j}} 
\text{ is a } (C^{1,\gamma},80R,\theta)\text{-domain.}
\end{equation}
We claim that
\begin{equation}\label{KKK3500}\begin{aligned}
& U_{\alpha},U_{\beta} \in S,
\quad 
U_{\alpha}, U_{\beta} \subsetneq U_{j},
\quad
U_{\alpha} \cap B_{R} \not = \emptyset 
\quad \text{and} \quad
U_{\beta} \cap B_{R} \not  = \emptyset.\\
&\qquad \Longrightarrow \qquad U_{\alpha} \cup U_{\beta} \in S 
\quad \text{and} \quad
U_{\alpha} \cup U_{\beta} \subsetneq U_{j}.
\end{aligned}\end{equation}
To prove this claim, fix $U_{\alpha},U_{\beta} \in S$ with 
\begin{equation}\label{KKK3700}
U_{\alpha},U_{\beta} \in S,
\quad 
U_{\alpha}, U_{\beta} \subsetneq U_{j},
\quad
U_{\alpha} \cap B_{R} \not  = \emptyset
\quad \text{and} \quad
U_{\beta} \cap B_{R} \not  = \emptyset.
\end{equation}
Then we have that
\begin{equation}\label{KKK3800}
U_{\alpha} \cap (\mathbb{R}^{n} \setminus \overline{U_{j}}) = \emptyset
\quad \text{and} \quad
U_{\beta} \cap (\mathbb{R}^{n} \setminus \overline{U_{j}}) = \emptyset.
\end{equation}
With \eqref{KKK3600}, \eqref{KKK3700} and \eqref{KKK3800}, we apply Lemma \ref{HHH7000} to $U_{\alpha}$, $U_{\beta}$ and $\mathbb{R}^{n} \setminus \overline{U_{j}}$. Then we have that $U_{\alpha} \cap U_{\beta} \not = \emptyset$, and so \eqref{HHH6500} yields that
\begin{equation*}
U_{\alpha} \subset U_{\beta}
\quad \text{or} \quad
U_{\beta} \subset U_{\alpha}.
\end{equation*}
It follows from \eqref{KKK3700} that
\begin{equation}\label{KKK3900}
U_{\alpha} \cup U_{\beta} \in S 
\quad \text{and} \quad
U_{\alpha} \cup U_{\beta} \subsetneq U_{j}.
\end{equation}
So under the assumption \eqref{KKK3700}, we have \eqref{KKK3900}. This proves the claim \eqref{KKK3500}. 

Since the number of the elements of $S$ is finite, by an induction using \eqref{KKK3500}, one can prove that 
\begin{equation}\label{KKK4000}
U_{k} \in S
\qquad \text{and} \qquad
U_{k} \subsetneq U_{j},
\end{equation}
which proves \eqref{KKK3100}. So it only remains to prove \eqref{KKK3400}.

By the choice of $U_{k}$ in \eqref{KKK3050}, we have that
\begin{equation}\label{KKK4100}
U_{k} \cap B_{R}
= \bigcup_{ U_{i} \in S, ~  U_{i} \subsetneq U_{j}, ~ U_{i} \cap B_{R} \not = \emptyset } [U_{i} \cap B_{R}]
= \bigcup_{U_{i} \in S , U_{i} \subsetneq  U_{j}  } [U_{i} \cap B_{R}].
\end{equation}
So we find from \eqref{KKK3300} and \eqref{KKK4100} that
\begin{equation}\label{KKK4400}
W_{j} \cap B_{R} 
= [U_{j} \cap B_{R}] \setminus \left( \bigcup_{U_{i} \in S, ~ U_{i} \subsetneq U_{j} } [U_{i} \cap B_{R}] \right)
= [U_{j} \cap B_{R}] \setminus [U_{k} \cap B_{R}].
\end{equation}
In view of \eqref{KKK4000}, we obtain that $ U_{k} \cap B_{R} \subset U_{j} \cap B_{R} $. So it follows from \eqref{KKK4400} that
\begin{equation}\label{KKK4500}
U_{j} \cap B_{R} 
= [W_{j} \cap B_{R}] \cup [U_{k} \cap B_{R}].
\end{equation}
On the other-hand, by \eqref{KKK4000} and the definition of $W_{j}$ in \eqref{KKK3300}, we have that
\begin{equation}\label{KKK4700}
W_{j} \cap U_{k} = \emptyset.
\end{equation}
So \eqref{KKK3400} holds from \eqref{KKK4500} and \eqref{KKK4700}.
\end{proof}

\begin{lem}\label{KKK5000}
For any $R \in (0,R_{3}]$ with $R_{3}(n,\gamma,\theta)$ in Lemma \ref{GGG1000},
let $U$ be composite $(C^{1,\gamma},80R,\theta)$-domain with subdomains $\{ U_{1}, \cdots, U_{K} \}$ and $U \cap B_{R} \not = \emptyset$. Set $S = \{ U_{0} : = U , U_{1}, \cdots, U_{K}, U_{K+1} : = \emptyset \}$ and
\begin{equation}\label{KKK5200}
U_{j} = \bigcap_{U_{i} \in S, ~ U_{i} \cap B_{R} \supset U \cap B_{R} } U_{i}.
\end{equation}
Then
\begin{equation}\label{KKK5300}
U_{j} \in S
\qquad \text{and} \qquad
U_{j} \cap B_{R} = U \cap B_{R}.
\end{equation}
In addition, for 
\begin{equation}\label{KKK5400}
W_{j} = U_{j} \setminus \left( \bigcup_{U_{i} \in S, ~ U_{i} \subsetneq U_{j} } U_{i} \right)
\qquad (j \in \{ 0,\cdots, K \} ),
\end{equation}
there exist $U_{k}, U_{l} \in S$ such that
\begin{equation}\label{KKK5600}
U_{j} \cap B_{R} 
= (W_{j} \sqcup U_{k} \sqcup U_{l}) \cap B_{R},
\end{equation} 
\begin{equation}\label{KKK5500}
U_{k},U_{l} \subsetneq U_{j}
\qquad \text{and} \qquad
U_{k} \cap B_{R}, U_{l} \cap B_{R} 
\subsetneq 
U_{j} \cap B_{R}.
\end{equation}
\end{lem}

\begin{proof}

[Step 1 : Proof of \eqref{KKK5300}] For any $U_{\alpha}, U_{\beta} \in S$ with $U_{\alpha} \cap B_{R}, U_{\beta} \cap B_{R} \supset U \cap B_{R} \not = \emptyset$, \eqref{HHH6500} gives that $U_{\alpha} \subset U_{\beta}$ or $U_{\beta} \supset U_{\alpha}$, and so we find that $U_{\alpha} \cap U_{\beta} \in S$. Since the number of the elements in $S$ are finite, one can prove that
\begin{equation}\label{KKK6200}
U_{j} \in S
\end{equation}
by using an induction. Also with the definition of $U_{j}$ in \eqref{KKK5200} and the fact that the number of the elements in $S$ are finite, one can easily check that 
\begin{equation}\label{KKK6300}
U_{j} \cap B_{R} \supset U \cap B_{R}.
\end{equation}
Since $U_{j} \subset U$, \eqref{KKK5300} holds from \eqref{KKK6200} and \eqref{KKK6300}. 

\medskip

\noindent [Step 2 : Splitting into three cases] We only consider the following three cases :

\medskip

(1) for any $U_{i} \in S$ with $U_{i} \cap B_{R} \not = \emptyset$, $U_{i} \supset U \cap B_{R}$ holds.

\medskip

(2) there exists $U_{\alpha} \in S$ such that 
\begin{equation}\label{KKK6500}
U_{\alpha} \cap B_{R} \not = \emptyset
\quad \text{ and } \quad
U_{\alpha}\not \supset U \cap B_{R} 
\end{equation}
satisfying
\begin{equation}\label{KKK6700}
U_{i} \in S 
\quad \text{and} \quad
U_{i} \cap B_{R} \not = \emptyset
\quad \Longrightarrow \quad
U_{i} \cap U_{\alpha} \not  = \emptyset.
\end{equation}

\medskip

(3) there exists $U_{\alpha},U_{\beta} \in S$ such that 
\begin{equation}\label{KKK7000}
U_{\alpha} \cap U_{\beta} = \emptyset,
\quad
U_{\alpha},U_{\beta} \not \supset U \cap B_{R} ,
\quad
U_{\alpha} \cap B_{R}  \not = \emptyset
\quad \text{ and } \quad
U_{\beta} \cap B_{R} \not = \emptyset
\end{equation}
satisfying
\begin{equation}\label{KKK7100}
U_{i} \in S 
\quad \text{and} \quad
U_{i} \cap B_{R} \not = \emptyset
\quad \Longrightarrow \quad
U_{i} \cap U_{\alpha} \not  = \emptyset 
\qquad \text{or} \qquad
U_{i} \cap U_{\beta} \not  = \emptyset.
\end{equation}

\medskip

We explain about these three cases. Suppose that (1) does not holds. Then there exists $U_{\alpha} \in S$ with $U_{\alpha} \cap B_{R} \not = \emptyset$ and $U_{\alpha}\not \supset U \cap B_{R}$. If three mutually disjoint $U_{1},U_{2},U_{3} \in  S$ satisfy
\begin{equation}\label{KKK7200}
U_{i} \in S,
\quad
U_{i} \cap B_{R} \not = \emptyset
\quad \text{and} \quad
U_{i} \not \supset U \cap B_{R}
\qquad (i \in \{ 1,2,3 \} )
\end{equation}
then by Lemma \ref{HHH7000}, we have a contradiction because $U_{1},U_{2},U_{3}$ are mutually disjoint $(C^{1,\gamma},80R,\theta)$-domains. So there exist at most two disjoint $U_{\alpha},U_{\beta} \in S$ with
\begin{equation*}
U_{\alpha},U_{\beta} \in S,
\quad
U_{\alpha} \cap B_{R} \not = \emptyset,
\quad
U_{\beta} \cap B_{R} \not = \emptyset
\quad \text{and} \quad
U_{\alpha},U_{\beta} \not \supset U \cap B_{R}
\end{equation*}
satisfying
\begin{equation*}
U_{i} \in S,
\quad
U_{i} \cap B_{R} \not = \emptyset
\quad \text{and} \quad
U_{i} \not \supset U \cap B_{R}
\quad \Longrightarrow \quad
U_{i} \cap U_{\alpha} \not  = \emptyset 
\quad \text{or} \quad
U_{i} \cap U_{\beta} \not  = \emptyset.
\end{equation*}
Also if $U_{i} \supset U \cap B_{R}$ then by the fact that $U_{\alpha} \cap B_{R} \not = \emptyset$ and $U_{\beta} \cap B_{R} \not = \emptyset$, we have that $U_{i} \cap U_{\alpha} \not  = \emptyset$ and $U_{i} \cap U_{\beta} \not  = \emptyset$. So Case (2) or Case (3) holds depending on the number of disjoint $U_{i}$ satisfying \eqref{KKK7200}.

\medskip

\noindent [Step 3 : Case (1)] In Case (1), we take $U_{k} = U_{l} = \emptyset$. First, suppose that
\begin{equation}\label{KKK7230}
U_{i} \in S,
\qquad
U_{i} \subsetneq U_{j}
\qquad \text{and} \qquad 
U_{i} \cap B_{R} \not = \emptyset.
\end{equation}
Then by the assumption of Case (1), $U_{i} \cap B_{R} \supset U \cap B_{R} \not = \emptyset$. So from \eqref{KKK5200}, we find that $U_{j} \subset U_{i}$, which contradicts that \eqref{KKK7230}. Thus
\begin{equation}\label{KKK7250}
\{ U_{i} \in S : U_{i} \subsetneq U_{j}
\quad \text{and} \quad 
U_{i} \cap B_{R} \not = \emptyset \} = \emptyset.
\end{equation}
By \eqref{KKK7250} and the definition of $W_{j}$ in \eqref{KKK5400}, we have that 
\begin{equation*}\begin{aligned}
W_{j} \cap B_{R} 
& = [ U_{j} \cap B_{R} ] \setminus \left( \bigcup_{U_{i} \in S, ~ U_{i} \subsetneq U_{j} } [U_{i} \cap B_{R}] \right) \\
& = [ U_{j} \cap B_{R} ] \setminus \left( \bigcup_{U_{i} \in S, ~ U_{i} \subsetneq U_{j}, ~ U_{i} \cap B_{R} \not = \emptyset } [U_{i} \cap B_{R}] \right) \\
& = U_{j} \cap B_{R}.
\end{aligned}\end{equation*}
So lemma holds for the Case (1) with $U_{k} = \emptyset$ and $U_{l} = \emptyset$.

\medskip

\noindent [Step 4 : Preliminary for Case (2) and (3)] For Case (2) and (3), we will use that
\begin{equation}\label{KKK7300}\begin{aligned}
U_{\alpha} \in S, \, &  U_{\alpha} \not \supset U \cap B_{R}, \, U_{\alpha} \cap B_{R} \not = \emptyset, \,
U_{k} = \bigcup_{U_{i} \in S, U_{i} \subsetneq  U_{j}, U_{i} \cap U_{\alpha} \not = \emptyset,
U_{i} \cap B_{R} \not = \emptyset } U_{i} \\
& \quad \Longrightarrow \quad
U_{k} \in S,
\quad
U_{k} \subsetneq U_{j}
\quad \text{and} \quad
U_{k} \cap B_{R} \subsetneq U_{j} \cap B_{R}.
\end{aligned}\end{equation}
To prove \eqref{KKK7300}, we show the following holds for $U_{\alpha} \in S$ with $U_{\alpha} \not \supset U \cap B_{R}$ :
\begin{equation}\begin{aligned}\label{KKK7370}
&  U_{i} \in S,
\quad 
U_{i} \subsetneq  U_{j}, 
\quad
U_{i} \cap U_{\alpha} \not = \emptyset
\quad \text{ and }  \quad
U_{i} \cap B_{R} \not = \emptyset
\quad \text{ for } i \in \{ 1,2 \} \\
& \qquad \Longrightarrow \qquad U_{1} \cap U_{2} \not = \emptyset.
\end{aligned}\end{equation}
Suppose not. Then there exist $U_{1}, U_{2} \in S$ satisfying
\begin{equation}\label{KKK7380}
U_{i} \in S,
~
U_{i} \subsetneq  U_{j}, 
~
U_{i} \cap U_{\alpha} \not = \emptyset,
~
U_{i} \cap B_{R} \not = \emptyset
~ \text{ and } ~
U_{1} \cap U_{2} = \emptyset
\quad \text{ for } i \in \{ 1,2 \}.
\end{equation}
So we have that $U_{1} \cap U_{\alpha} \not = \emptyset$ and $U_{2} \cap U_{\alpha} \not = \emptyset$, and it follows from \eqref{HHH6500} that
\begin{equation*}
U_{1} \subset U_{\alpha} 
\qquad \text{or} \qquad
U_{\alpha} \subset U_{1},
\end{equation*}
and
\begin{equation*}
U_{2} \subset U_{\alpha} 
\qquad \text{or} \qquad
U_{\alpha} \subset U_{2}.
\end{equation*}
By comparing with $U_{1} \cap U_{2} = \emptyset$ in \eqref{KKK7380}, $U_{1} \cap U_{\alpha} \not = \emptyset$ and $U_{2} \cap U_{\alpha} \not = \emptyset$ in \eqref{KKK7370},
\begin{equation}\label{KKK7385}
U_{1} , U_{2} \subset U_{\alpha}.
\end{equation}
Since $U_{\alpha} \not \supset U \cap B_{R}$, we have that $B_{R} \setminus \overline{U_{\alpha}} \not = \emptyset$. So we obtain from \eqref{KKK7385} that 
\begin{equation}\label{KKK7390}
U_{1} \cap (\mathbb{R}^{n} \setminus \overline{U_{\alpha}}) = \emptyset,
\quad 
U_{2} \cap (\mathbb{R}^{n} \setminus \overline{U_{\alpha}}) = \emptyset
\quad \text{and} \quad
B_{R} \cap (\mathbb{R}^{n} \setminus \overline{U_{\alpha}}) \not= \emptyset.
\end{equation}
By \eqref{KKK7380}, $U_{1} \cap B_{R} \not = \emptyset$, $U_{2} \cap B_{R} \not = \emptyset$ and $U_{1} \cap U_{2} = \emptyset$. So with \eqref{KKK7390} and the fact that $\mathbb{R}^{n} \setminus \overline{U_{\alpha}}$ is $(C^{1,\gamma},80R,\theta)$-domain, Lemma \ref{HHH7000} gives a contradiction for $U_{1}$, $U_{2}$ and $\mathbb{R}^{n} \setminus \overline{U_{\alpha}}$. So \eqref{KKK7370} holds. In view of \eqref{KKK7370} and \eqref{HHH6500},
\begin{equation}\begin{aligned}\label{KKK7400}
& U_{i} \in S,
\quad 
U_{i} \subsetneq  U_{j}, 
\quad
U_{i} \cap U_{\alpha} \not = \emptyset
\quad \text{ and }  \quad
U_{i} \cap B_{R} \not = \emptyset
\quad \text{ for } i \in \{ 1,2 \} \\
& \qquad \Longrightarrow \qquad U_{1} \cap U_{2} \not = \emptyset \\
& \qquad \Longrightarrow \qquad 
U_{1} \subset U_{2} 
\quad \text{or} \quad
U_{2} \subset U_{1} \\
& \qquad \Longrightarrow \qquad 
U_{1} \cup U_{2} \in S
\qquad \text{and} \qquad
U_{1} \cup U_{2} \subsetneq U_{j}.
\end{aligned}\end{equation}
Since the number of the elements in $S$ are finite, by an induction using the definition of $U_{k}$ in \eqref{KKK7300} and \eqref{KKK7400}, one can show that 
\begin{equation}\label{KKK7450}
U_{k} \in S
\qquad \text{and} \qquad
U_{k} \subsetneq U_{j}.
\end{equation}
So it only remains to prove that $U_{k} \cap B_{R} \subsetneq U_{j} \cap B_{R}$. Suppose not. Then $U_{k} \cap B_{R} \supset U_{j} \cap B_{R}$. By \eqref{KKK5300},
\begin{equation*}
U_{k} \cap B_{R} \supset  U_{j} \cap B_{R} \supset U \cap B_{R}.
\end{equation*}
So by the definition of $U_{j}$ in \eqref{KKK5200}, we have that $U_{j} \subset U_{k}$, which contradicts \eqref{KKK7450}. So we find that $U_{k} \cap B_{R} \subsetneq U_{j} \cap B_{R}$, and \eqref{KKK7300} holds from \eqref{KKK7450}.

\medskip

\noindent [Step 5 : Case (2)] We handle Case (2). With the assumption \eqref{KKK6500} and \eqref{KKK6700}, set
\begin{equation}\label{KKK7500}
U_{k} = \bigcup_{U_{i} \in S, U_{i} \subsetneq  U_{j}, U_{i} \cap U_{\alpha} \not = \emptyset, U_{i} \cap B_{R} \not = \emptyset } U_{i} \text{ as in \eqref{KKK7300} \quad and \quad } U_{l} = \emptyset.
\end{equation}
Then by \eqref{KKK7300}, \eqref{KKK5500} holds. We next prove \eqref{KKK5600}. We claim that
\begin{equation}\label{KKK7600}
U_{i} \in S,
\quad
U_{i} \subsetneq U_{j} 
\quad \text{and} \quad
U_{i} \cap B_{R} \not = \emptyset
\quad \Longrightarrow \quad
U_{i} \subset U_{k}.
\end{equation}
By \eqref{KKK6700},
\begin{equation}\label{KKK7620}
U_{i} \in S,
\quad
U_{i} \subsetneq U_{j} 
\quad \text{and} \quad
U_{i} \cap B_{R} \not = \emptyset
\qquad \Longrightarrow \qquad
U_{i} \cap U_{\alpha} \not  = \emptyset.
\end{equation}
From \eqref{KKK7620} and \eqref{KKK7500}, if $U_{i} \in S$ satisfies
$U_{i} \subsetneq U_{j}$ and $U_{i} \cap B_{R} \not = \emptyset$ then
\begin{equation*}
U_{i} \subset \left( \bigcup_{U_{i} \in S, U_{i} \subsetneq  U_{j}, U_{i} \cap U_{\alpha} \not = \emptyset,
U_{i} \cap B_{R} \not = \emptyset } U_{i} \right)= U_{k},
\end{equation*}
and the claim \eqref{KKK7600} follows. By \eqref{KKK7600}, 
\begin{equation*}
\bigcup_{U_{i} \in S , U_{i} \subsetneq  U_{j}} [U_{i} \cap B_{R}]
= \bigcup_{U_{i} \in S , U_{i} \subsetneq  U_{j}, U_{i} \cap  B_{R}  \not = \emptyset} [U_{i} \cap B_{R}]
 \subset U_{k} \cap B_{R}.
\end{equation*}
On the other-hand, by the definition of $U_{k}$ in \eqref{KKK7500}, 
\begin{equation*}
U_{k} \cap B_{R}
= \bigcup_{U_{i} \in S , U_{i} \subsetneq  U_{j}, U_{i} \cap U_{\alpha} \not = \emptyset, U_{i} \cap B_{R} \not = \emptyset } [U_{i} \cap B_{R}]
\subset \bigcup_{U_{i} \in S , U_{i} \subsetneq  U_{j}  } [U_{i} \cap B_{R}].
\end{equation*}
By combining the above two inclusions,
\begin{equation}\label{KKK7700}
\bigcup_{U_{i} \in S , U_{i} \subsetneq U_{j} } [U_{i} \cap B_{R}]
= U_{k} \cap B_{R}.
\end{equation} 
From the definition of $W_{j}$ in \eqref{KKK5400}, \eqref{KKK7700} yields that
\begin{equation}\label{KKK7720}\begin{aligned}
W_{j} \cap B_{R} 
& = [U_{j} \cap B_{R}] \setminus \left( \bigcup_{U_{i} \in S, ~ U_{i} \subsetneq U_{j} } [U_{i} \cap B_{R}] \right) \\
& = [ U_{j} \cap B_{R}] \setminus [ U_{k} \cap B_{R}].
\end{aligned}\end{equation}
Also by \eqref{KKK5500},
\begin{equation*}
U_{k} \cap B_{R} \subset U_{j} \cap B_{R},
\end{equation*}
and we find from \eqref{KKK7720} that
\begin{equation}\label{KKK7740}
U_{j} \cap B_{R}
= ( W_{j} \cap B_{R} ) \cup (U_{k} \cap B_{R}).
\end{equation}
In view of \eqref{KKK5500}, we have that $U_{k} \subsetneq U_{j}$. So by the definition of $W_{j}$ in \eqref{KKK5400},
\begin{equation}\label{KKK7760}
W_{j} \cap U_{k} = \emptyset.
\end{equation}
Since $U_{l} = \emptyset$, it follows from \eqref{KKK7740} and \eqref{KKK7760} that
\begin{equation}
U_{j} \cap B_{R} 
= (W_{j} \sqcup U_{k}) \cap B_{R}
= (W_{j} \sqcup U_{k} \sqcup U_{l}) \cap B_{R}
\end{equation} 
which proves \eqref{KKK5600} for Case (2).

\medskip

\noindent [Step 6 : Case (3)] We handle Case (3). With the assumption \eqref{KKK7000} and \eqref{KKK7100}, set
\begin{equation}\label{KKK7800}
U_{k} = \bigcup_{U_{i} \in S, U_{i} \subsetneq  U_{j}, U_{i} \cap U_{\alpha} \not = \emptyset,
U_{i} \cap B_{R} \not = \emptyset } U_{i}
\end{equation}
and
\begin{equation}\label{KKK7900}
U_{l} = \bigcup_{U_{i} \in S, U_{i} \subsetneq  U_{j},  U_{i} \cap U_{\beta} \not = \emptyset, U_{i} \cap B_{R} \not = \emptyset } U_{i}.
\end{equation}
Then by \eqref{KKK7300},
\begin{equation}\label{KKK7950}
U_{k},U_{l} \in S,
\quad 
U_{k},U_{l} \subsetneq U_{j},
\quad 
U_{k} \cap B_{R} \subsetneq U_{j} \cap B_{R}
\quad \text{and} \quad
U_{l} \cap B_{R} \subsetneq U_{j} \cap B_{R},
\end{equation}
and \eqref{KKK5500} holds. So it only remains to prove \eqref{KKK5600}. To prove \eqref{KKK5600}, we claim that
\begin{equation}\label{KKK8000}
U_{i} \in S,
\quad
U_{i} \subsetneq U_{j} 
\quad \text{and} \quad
U_{i} \cap B_{R} \not = \emptyset
\quad \Longrightarrow \quad
U_{i} \subset U_{k} \cup U_{l}.
\end{equation}
By \eqref{KKK7100},
\begin{equation}\label{KKK8100}
U_{i} \in S,
~ ~
U_{i} \subsetneq U_{j} 
~ ~ \text{and} ~ ~ 
U_{i} \cap B_{R} \not = \emptyset
\quad \Longrightarrow \quad
U_{i} \cap U_{\alpha} \not  = \emptyset 
~ ~ \text{or} ~ ~
U_{i} \cap U_{\beta} \not  = \emptyset,
\end{equation}
It follows from \eqref{KKK8100}, the definition of $U_{k}$ in \eqref{KKK7800} and $U_{l}$ in \eqref{KKK7900} that if $U_{i} \in S$, $U_{i} \subsetneq U_{j}$ and $U_{i} \cap B_{R} \not = \emptyset$ then 
\begin{equation*}\begin{aligned}
U_{i} 
& \subset \left( \bigcup_{U_{m} \in S, U_{m} \subsetneq  U_{j}, U_{m} \cap U_{\alpha} \not = \emptyset,
U_{m} \cap B_{R} \not = \emptyset } U_{m} \right)
\cup \left( \bigcup_{U_{m} \in S, U_{m} \subsetneq  U_{j}, U_{m} \cap U_{\alpha} \not = \emptyset,
U_{m} \cap B_{R} \not = \emptyset } U_{m} \right) \\
& = U_{k} \cup U_{l},
\end{aligned}\end{equation*}
and the claim \eqref{KKK8000} follows. In view of \eqref{KKK8000}, 
\begin{equation*}
\bigcup_{U_{i} \in S , U_{i} \subsetneq  U_{j}} [U_{i} \cap B_{R}]
= \bigcup_{U_{i} \in S , U_{i} \subsetneq  U_{j}, U_{i} \cap  B_{R}  \not = \emptyset} [U_{i} \cap B_{R}]
 \subset (U_{k} \cup U_{l}) \cap B_{R}.
\end{equation*}
On the other-hand, by the definition of $U_{k}$, 
\begin{equation*}
U_{k} \cap B_{R}
= \bigcup_{U_{i} \in S , U_{i} \subsetneq  U_{j}, U_{i} \cap U_{\alpha} \not = \emptyset, U_{i} \cap B_{R} \not = \emptyset } [U_{i} \cap B_{R}]
\subset \bigcup_{U_{i} \in S , U_{i} \subsetneq  U_{j}  } [U_{i} \cap B_{R}].
\end{equation*}
Similarly, by the definition of $U_{l}$, 
\begin{equation*}
U_{l} \cap B_{R}
= \bigcup_{U_{i} \in S , U_{i} \subsetneq  U_{j}, U_{i} \cap U_{\beta} \not = \emptyset, U_{i} \cap B_{R} \not = \emptyset } [U_{i} \cap B_{R}]
\subset \bigcup_{U_{i} \in S, U_{i} \subsetneq  U_{j}  } [U_{i} \cap B_{R}].
\end{equation*}
By combining the above three inclusions,
\begin{equation}\label{KKK9100}
\bigcup_{U_{i} \in S , U_{i} \subsetneq U_{j} } [U_{i} \cap B_{R}]
= (U_{k} \cup U_{l}) \cap B_{R}.
\end{equation} 
From the definition of $W_{j}$ in \eqref{KKK5400}, \eqref{KKK9100} yields that
\begin{equation}\label{KKK9400}\begin{aligned}
W_{j} \cap B_{R} 
& = [U_{j} \cap B_{R}] \setminus \left( \bigcup_{U_{i} \in S, ~ U_{i} \subsetneq U_{j} } [U_{i} \cap B_{R}] \right) \\
& = [ U_{j} \cap B_{R}] \setminus [(U_{k} \cup U_{l}) \cap B_{R}].
\end{aligned}\end{equation}
In view of \eqref{KKK7950},
\begin{equation}\label{KKK9500}
(U_{k} \cup U_{l}) \cap B_{R} \subset U_{j} \cap B_{R},
\end{equation}
We find from \eqref{KKK9400} and \eqref{KKK9500} that
\begin{equation}\label{KKK9600}
U_{j} \cap B_{R}
= [ W_{j} \cap B_{R} ] \cup [ (U_{k} \cup U_{l}) \cap B_{R}].
\end{equation}
From \eqref{KKK7950}, we have that $U_{k},U_{l} \subsetneq U_{j}$. So by the definition of $W_{j}$ in \eqref{KKK5400},
\begin{equation}\label{KKK9200}
W_{j} \cap U_{k} = \emptyset
\qquad \text{and} \qquad
W_{j} \cap U_{l} = \emptyset.
\end{equation}

We claim that 
\begin{equation}\label{KKK9300}
U_{k} \cap U_{l} \not = \emptyset
\qquad \Longrightarrow \qquad
\text{the assumption in Case (2) holds.}
\end{equation}
Suppose that $U_{k} \cap U_{l} \not = \emptyset$. From \eqref{HHH6500}, we have that $U_{k} \subset U_{l}$ or $U_{l} \subset U_{k}$ holds. Without loss of generality, we assume that $U_{k} \subset U_{l}$. Then there exists
\begin{equation*}
U_{m} \in \{ U_{i} \in S :  U_{i} \subsetneq  U_{j}, ~ U_{i} \cap U_{\alpha} \not = \emptyset, ~
U_{i} \cap B_{R} \not = \emptyset \} \not = \emptyset,
\end{equation*}
which holds from \eqref{KKK7800} and the fact that $U_{k} \not = \emptyset$. So again by \eqref{KKK7800},
\begin{equation*}
U_{\alpha} \cap U_{k} 
= \bigcup_{U_{i} \in S, U_{i} \subsetneq  U_{j}, U_{i} \cap U_{\alpha} \not = \emptyset,
U_{i} \cap B_{R} \not = \emptyset } (U_{i} \cap U_{\alpha})
\supset (U_{m} \cap U_{\alpha})
\not = \emptyset.
\end{equation*}
Similarly, one can prove that $U_{\beta} \cap U_{l} \not = \emptyset$. So by the fact that $U_{k} \subset U_{l}$,
\begin{equation}\label{KKK9350}
U_{\alpha} \cap U_{l} \not = \emptyset
\qquad \text{and} \qquad
U_{\beta} \cap U_{l} \not = \emptyset.
\end{equation}
From \eqref{HHH6500} and \eqref{KKK9350},
\begin{equation*}
U_{\alpha} \subset U_{l} 
\qquad \text{or} \qquad
U_{l} \subset U_{\alpha},
\end{equation*}
and 
\begin{equation*}
U_{\beta} \subset U_{l} 
\qquad \text{or} \qquad
U_{l} \subset U_{\beta}.
\end{equation*}
We have from \eqref{KKK7000} that $U_{\alpha} \cap U_{\beta} = \emptyset$. So by comparing \eqref{KKK9350} with the two above inclusions,
\begin{equation*}
U_{\alpha}, U_{\beta}, U_{k} \subset U_{l},
\end{equation*}
and it follows from \eqref{KKK7100} that
\begin{equation*}
U_{i} \in S 
~ \text{ and } ~
U_{i} \cap B_{R} \not = \emptyset
\quad \Longrightarrow \quad
U_{i} \cap U_{k} \not  = \emptyset
~ \text{ or } ~
U_{i} \cap U_{l} \not  = \emptyset.
\quad \Longrightarrow \quad
U_{i} \cap U_{l} \not  = \emptyset.
\end{equation*}
From \eqref{KKK5500}, we have that $U_{l} \not \supset U \cap  B_{R}$. Also by that $U_{l} \not = \emptyset$, we have that $\{ U_{i} \in S :  U_{i} \subsetneq  U_{j}, ~ U_{i} \cap U_{\beta} \not = \emptyset, ~ U_{i} \cap B_{R} \not = \emptyset \} \not = \emptyset$, which gives that $U_{l} \cap B_{R} \not = \emptyset$ by \eqref{KKK7900}. So $U_{l}$ satisfies the assumption of Case (2) instead of $U_{\alpha}$. Under the assumption $U_{k} \cap U_{l} \not = \emptyset$, Case (3) turns to Case (2) which we proved earlier. So we only consider the case that $U_{k} \cap U_{l} = \emptyset$. 

If $U_{k} \cap U_{l} = \emptyset$ then we obtain from \eqref{KKK9600} and \eqref{KKK9200} that
\begin{equation*}
U_{j} \cap B_{R} 
= (W_{j} \sqcup U_{k} \sqcup U_{l}) \cap B_{R}
\end{equation*} 
which proves \eqref{KKK5600}. This completes the proof.
\end{proof}

Let $U$ be composite $(C^{1,\gamma},80R,\theta)$-domain with subdomains $\{ U_{1}, \cdots, U_{K} \}$. By using Lemma \ref{KKK3000} and Lemma \ref{KKK5000}, we prove in Lemma \ref{MMM1000} that $U \cap B_{R}$ can be decomposed into $ W_{i_{l}} \sqcup \cdots \sqcup W_{i_{-m}}$ for $W_{k}$ in \eqref{MMM1050} so that $U_{i_{0}} \supsetneq \cdots \supsetneq U_{i_{l}}$, $U_{i_{0}} \supsetneq \cdots \supsetneq U_{i_{-m}}$ and $U_{i_{-m}}, \cdots, U_{i_{0}}, \cdots U_{i_{l}} \in S$ $(l,m \geq 0)$.

\begin{lem}\label{MMM1000}
For any $R \in (0,R_{4}]$ with $R_{4}(n,\gamma,\theta,\tau) $ in Lemma \ref{HHH1000}, let $U$ be composite $(C^{1,\gamma},80R,\theta)$-domain with subdomains $\{ U_{1}, \cdots, U_{K} \}$ and $U \cap B_{R} \not = \emptyset$. Set $S = \{ U_{0} : =U, U_{1}, \cdots, U_{K}, U_{K+1} : = \emptyset \}$ and
\begin{equation}\label{MMM1050}
W_{j} = U_{j} \setminus \left( \bigcup_{U_{i} \in S, ~ U_{i} \subsetneq U_{j} } U_{i} \right) 
\qquad (j \in \{ 0,\cdots, K \} ).
\end{equation}
Then there exist $U_{i_{l+1}}, \cdots, U_{i_{-m-1}} \in S$ with $l,m \geq 0$ such that
\begin{equation}\label{MMM1100}
U \cap B_{R} = ( W_{i_{l}} \sqcup \cdots \sqcup W_{i_{-m}} ) \cap B_{R},
\end{equation}
\begin{equation}\label{MMM1200}
U_{i_{k}} \supsetneq U_{i_{k+1}} \quad \text{for} \quad 0 \leq k \leq l,
\qquad
U_{i_{k}} \supsetneq U_{i_{k-1}} \quad \text{for} -m \leq k \leq 0
\end{equation}
and 
\begin{equation}\label{MMM1250}
U_{i_{-m-1}} = U_{i_{l+1}} = \emptyset
\end{equation}
satisfying 
\begin{equation}\label{MMM1300}
U \cap B_{R} 
= U_{i_{0}} \cap B_{R} 
= ( W_{i_{0}} \sqcup U_{i_{1}} \sqcup U_{i_{-1}} ) \cap B_{R}
\end{equation}
\begin{equation}\label{MMM1500}
U_{i_{k}} \cap B_{R}  = (W_{i_{k}} \sqcup U_{i_{k+1}}) \cap B_{R} 
\qquad
(1 \leq k \leq l)
\end{equation}
\begin{equation}\label{MMM1600}
U_{i_{k}} \cap B_{R}  =  (W_{i_{k}} \sqcup U_{i_{k-1}}) \cap B_{R} 
\qquad
(-m \leq k \leq -1),
\end{equation}
and
\begin{equation}\label{MMM1700}
U_{i_{1}} \cap B_{R}, U_{i_{-1}} \cap B_{R} 
\subsetneq 
U_{i_{0}} \cap B_{R}.
\end{equation}
Moreover, if $ \partial U \cap B_{R} \not = \emptyset$ then $m = 0$.
\end{lem}

\begin{rem}\label{MMM2000}
In view of Lemma \ref{MMM1000}, if $a_{ij} = \sum_{k} a_{ij}^{k} \chi_{W_{k}}$ then 
\begin{equation*}
a_{ij} = a_{ij}^{i_{0}} \chi_{ U_{i_{0}} \setminus (U_{i_{1}} \cup U_{i_{-1}}) } 
+ \sum_{1 \leq k \leq l} a_{ij}^{i_{k}} \chi_{ U_{i_{k}} \setminus U_{i_{k+1}} } + \sum_{-m \leq k \leq -1} a_{ij}^{i_{k}} \chi_{ U_{i_{k}} \setminus U_{i_{k-1}} } 
\text{ in } U \cap B_{R}.
\end{equation*}
\end{rem}

\begin{proof}

[Step 1 : $\partial U \cap B_{R} = \emptyset$] Assume that $\partial U \cap B_{R} = \emptyset$. By using Lemma \ref{KKK5000}, one can choose $U_{i_{0}}$, $U_{i_{1}}$ and $U_{i_{-1}}$ satisfying \eqref{MMM1300}, \eqref{MMM1700},
\begin{equation}\label{MMM3100}
U_{i_{0}} \supsetneq U_{i_{1}},
\quad
U_{i_{0}} \supsetneq U_{i_{-1}}
\quad \text{and} \quad
U_{i_{1}} \cap B_{R}, U_{i_{-1}} \cap B_{R}
\subsetneq U_{i_{0}} \cap B_{R}
= U \cap B_{R}.
\end{equation}
With Lemma \ref{KKK3000}, we decompose $U_{i_{1}}$ and  $U_{i_{-1}}$. We first handle $U_{i_{1}}$. If $U_{i_{1}} \cap B_{R} = \emptyset$ then the chain is finished by taking $l = 0$ and $U_{i_{1}} = \emptyset$. Otherwise, $U_{i_{1}} \cap B_{R} \not = \emptyset$ and $U_{i_{1}} \not \supset U \cap B_{R}$ by \eqref{MMM3100}. So from Lemma \ref{KKK3000}, there exists $U_{i_{2}}$ satisfying \eqref{MMM1500} for $k=1$ and $U_{i_{2}} \subsetneq U_{i_{1}}$ and $U_{i_{2}} \cap B_{R} \subsetneq U_{i_{1}} \cap B_{R}$. Since the number of the element in $S$ is finite, repeat this process until $U_{i_{l+1}} \cap B_{R} = \emptyset$. Then by letting $U_{i_{l+1}} = \emptyset$, we have \eqref{MMM1500} and 
\begin{equation*}
U_{i_{k}} \supsetneq U_{i_{k+1}} \quad  \text{for} \quad   0 \leq k \leq l
\qquad \text{and} \qquad 
U_{i_{l+1}} = \emptyset.
\end{equation*}
Similarly, one can repeat this process for handling $U_{i_{-1}}$ to obtain \eqref{MMM1600} and
\begin{equation*}
U_{i_{k}} \supsetneq U_{i_{k+1}} \quad  \text{for} \quad   -m \leq k \leq 0
\qquad  \text{and} \qquad 
U_{i_{-m-1}} = \emptyset.
\end{equation*}
In view of \eqref{MMM1250}, \eqref{MMM1300}, \eqref{MMM1500} and \eqref{MMM1600}, we obtain that \eqref{MMM1100}.

\medskip

[Step 2 : $\partial U \cap B_{R} \not = \emptyset$] Assume that $\partial U \cap B_{R} \not = \emptyset$. Then
\begin{equation}\label{MMM3500}
(\mathbb{R}^{n} \setminus \overline{U}) \cap B_{R} \not = \emptyset 
\qquad \text{and} \qquad
\mathbb{R}^{n} \setminus \overline{U}
\text{ is } (C^{1,\gamma},80R,\theta)\text{-domain}.
\end{equation}
By Lemma \ref{KKK5000}, one can find $U_{i_{0}}$, $U_{i_{1}}, U_{i_{-1}} \in S$ satisfying \eqref{MMM1300}, \eqref{MMM1700} and
\begin{equation*}
U_{i_{0}} \supsetneq U_{i_{1}},
\quad
U_{i_{0}} \supsetneq U_{i_{-1}}
\quad \text{and} \quad
U_{i_{1}} \cap B_{R}, U_{i_{-1}} \cap B_{R}
\subsetneq U_{i_{0}} \cap B_{R}
= U \cap B_{R},
\end{equation*}
which implies that  $U_{i_{1}} \cap U_{i_{-1}} = \emptyset$, $U_{i_{1}} \subset U$ and $U_{i_{-1}} \subset U$. So if $U_{i_{1}} \cap B_{R} \not = \emptyset$ and $U_{i_{-1}} \cap B_{R} \not = \emptyset$ then with \eqref{MMM3500}, one can apply Lemma \ref{HHH7000} to $U_{i_{1}}$, $U_{i_{-1}}$ and $ \mathbb{R}^{n} \setminus \overline{U}$ obtaining a contradiction. So $U_{i_{1}} = \emptyset$ or $U_{i_{-1}}  = \emptyset$ holds, and without loss of generality we let $U_{i_{-1}}  = \emptyset$ to obtain $m=0$. Then one can apply Lemma \ref{KKK3000} inductively as in the case that $\partial U \cap B_{R} = \emptyset$ to obtain \eqref{MMM1200} - \eqref{MMM1600}. Since $m=0$, \eqref{MMM1100} follows from \eqref{MMM1300} - \eqref{MMM1600}.
\end{proof}

To prove Theorem \ref{main theorem1} and Theorem \ref{main theorem2}, we apply Lemma \ref{CTS8000} and Lemma \ref{HHH1000} to $U_{i_{0}} \supsetneq \cdots \supsetneq U_{i_{l}}$ and $U_{i_{0}} \supsetneq \cdots \supsetneq U_{i_{-m}}$ in Lemma \ref{MMM1000}. Then we find that the boundaries of a composite $(C^{1,\gamma},80R,\theta)$-domain become graphs with respect to some coordinate system. We first handle the case that $B_{R} \subset U$.

\maintheoreminterior*

\begin{proof}
For $R_{4}(n,\gamma,\theta,\tau) \in (0,1]$ in Lemma \ref{HHH1000}, we let $R_{0} = R_{4} \in (0,1]$. 

\medskip

\noindent [Step 1: Settings] By Lemma \ref{MMM1000}, there exist $U_{i_{l+1}}, \cdots, U_{i_{-m-1}} \in S$ $(l,m \geq 0)$ such that
\begin{equation}\label{MMM6100}
U \cap B_{R} = ( W_{i_{l}} \sqcup \cdots \sqcup W_{i_{-m}} ) \cap B_{R},
\end{equation}
\begin{equation}\label{MMM6200}
U_{i_{k}} \supsetneq U_{i_{k+1}} \quad \text{for} \quad 0 \leq k \leq l,
\qquad
U_{i_{k}} \supsetneq U_{i_{k-1}} \quad \text{for} -m \leq k \leq 0
\end{equation}
and 
\begin{equation}\label{MMM6250}
U_{i_{-m-1}} = U_{i_{l+1}} = \emptyset
\end{equation}
satisfying 
\begin{equation}\label{MMM6300}
U \cap B_{R} 
= U_{i_{0}} \cap B_{R} 
= ( W_{i_{0}} \sqcup U_{i_{1}} \sqcup U_{i_{-1}} ) \cap B_{R}
\end{equation}
\begin{equation}\label{MMM6400}
U_{i_{k}} \cap B_{R}  = (W_{i_{k}} \sqcup U_{i_{k+1}}) \cap B_{R} 
\qquad
(1 \leq k \leq l)
\end{equation}
\begin{equation}\label{MMM6500}
U_{i_{k}} \cap B_{R}  =  (W_{i_{k}} \sqcup U_{i_{k-1}}) \cap B_{R} 
\qquad
(-m \leq k \leq -1)
\end{equation}
and
\begin{equation}\label{MMM6550}
U_{i_{1}} \cap B_{R}, U_{i_{-1}} \cap B_{R} 
\subsetneq 
U_{i_{0}} \cap B_{R}.
\end{equation}
So \eqref{MMM5300} holds from \eqref{MMM6100}.

We claim that
\begin{equation}\label{MMM6700}
U_{i_{k}} \not \supset B_{R} 
\qquad ( k \in \{ -m,\cdots, l \} , ~ k \not = 0 ).
\end{equation}
Suppose not. Then we have that $B_{R} \subset U_{i_{k}}$ for some $k \in \{ -m,\cdots, l \} , ~ k \not = 0$. On the other-hand, we have from \eqref{MMM6300}, \eqref{MMM6400} and \eqref{MMM6500} that 
\begin{equation*}
U_{i_{k}} \cap B_{R} \subsetneq U_{i_{0}} \cap B_{R} \subset B_{R},
\end{equation*}
and a contradiction occurs from the fact that $B_{R} \subset U_{i_{k}}$. So the claim \eqref{MMM6570} holds.

One can easily prove that
\begin{equation}\label{MMM6560}
\partial V \cap B_{R} = \emptyset
\qquad \Longrightarrow \qquad
B_{R} \subset V 
\qquad \text{or} \qquad
B_{R} \cap V = \emptyset,
\end{equation}
and
\begin{equation}\label{MMM6570}
\partial V \cap B_{R} \not = \emptyset
\qquad \Longrightarrow \qquad
V \not \supset B_{R}
\qquad \text{and} \qquad
V^{c} \not \supset B_{R}.
\end{equation}

We claim that 
\begin{equation}\label{MMM6600}
\partial U_{i_{0}} \cap B_{R} = \emptyset.
\end{equation}
Suppose not. Then we discover that $\partial U_{i_{0}} \cap B_{R} \not = \emptyset$, and so \eqref{MMM6570} implies that $U_{i_{0}} \not \supset B_{R}$. From \eqref{MMM6300}, we have that $U \cap B_{R} = U_{i_{0}} \cap B_{R} \not \supset B_{R}$, and this contradicts the assumption of the lemma that $U \supset B_{R}$. So the claim \eqref{MMM6600} holds.

From \eqref{MMM6100} and the fact that $B_{R} \subset U$, there exists $k \in \{ -m,\cdots, l \} $ satisfying
\begin{equation}\label{MMM6900}
\mathbf{0} \in W_{i_{k}} \subset U_{i_{k}}.
\end{equation} 
Without loss of generality, we may assume that $k \geq 0$. We first prove the lemma when $k \geq 1$.

\bigskip

\noindent [Step 2] In this step, we prove the lemma for $k \geq 1$.

\medskip

\noindent [Step 2-1 : $U_{i_{k}}, U_{i_{k+1}},\cdots, U_{i_{l}}$] In this step, we show that \eqref{MMM5400}, \eqref{MMM5500} for $d \in \{ k, \cdots, l \} $ and \eqref{MMM5600} hold when $k \geq 1$. We claim that
\begin{equation}\label{MMM6950}
\partial U_{i_{k}} \cap B_{R} \not = \emptyset.
\end{equation}
Suppose not. Then we have that $\partial U_{i_{k}} \cap B_{R} = \emptyset$, and it follows from \eqref{MMM6560} and \eqref{MMM6900} that $B_{R} \subset U_{i_{k}}$. But this contradicts \eqref{MMM6700}. So the claim \eqref{MMM6950} holds.

Since $U_{i_{k}}$ is $(C^{1,\gamma},80R,\theta)$-domain, by applying Lemma \ref{CTS8000}, there exist an orthonormal matrix $V$ with $\det V>0$ and  $C^{1,\gamma}$-function $\varphi : B_{R}' \to \mathbb{R}$ such that
\begin{equation}\label{MMM7000}
U_{i_{k}} \cap B_{R} = \left \{ \sum\limits_{1 \leq k \leq n} y^{k} V_{k} \in B_{R} : y^{1} > \varphi_{k}(y') \right\}
\qquad \text{and} \qquad
|\varphi_{k}(0')| < R,
\end{equation}
with the estimates
\begin{equation}\label{MMM7200}
D_{y'}\varphi_{k}(0')=0',
\quad
\| D_{y'}\varphi_{k} \|_{L^{\infty}(B_{R}')} 
\leq \tau
\quad \text{and} \quad
[D_{y'}\varphi_{k}]_{C^{\gamma}(B_{R}')} 
\leq 18 n\theta.
\end{equation}
So for $y$-coordinate system with the orthonormal basis $\{ V_{1}, \cdots, V_{n} \}$, we have that
\begin{equation}\label{MMM7100}
U_{i_{k}} \cap B_{R}  = \left \{ (y^{1},y') \in B_{R}  : y^{1} > \varphi_{k}(y') \right\}
\qquad \text{and} \qquad
|\varphi_{k}(0')| < R.
\end{equation}
Thus \eqref{MMM5500} for $d=0$ and \eqref{MMM5600} hold. We next handle $U_{i_{k+1}}$. 

Suppose that $\partial U_{i_{k+1}} \cap B_{R} = \emptyset$. Then by \eqref{MMM6560}, we have that $B_{R} \subset U_{i_{k+1}}$ or $B_{R} \cap U_{i_{k+1}} = \emptyset$. On the other-hand, from \eqref{MMM6200} and \eqref{MMM6550}, we have that $U_{i_{k+1}} \cap B_{R} \subset U_{i_{1}} \cap B_{R} \subsetneq U_{i_{0}} \cap B_{R} \subset B_{R}$. So we find that $B_{R} \cap U_{i_{k+1}} = \emptyset$. So with \eqref{MMM7100}, choose $l=k$ and $\varphi_{l+1} \equiv R$. This finishes the proof for [Step 2-1] when $\partial U_{i_{k+1}} \cap B_{R} = \emptyset$. 

Next, assume that for some $l \geq k+1$, we have that
\begin{equation}\label{MMM7250}
U_{i_{k+1}} \cap B_{R}  \not = \emptyset , \cdots, U_{i_{l}} \cap B_{R} \not = \emptyset 
\qquad \text{ and } \qquad
U_{i_{l+1}} \cap B_{R} = \emptyset.
\end{equation}
Let $z$-coordinate system be the coordinate system with $z = -y$. Then by letting $\bar{\varphi}_{k}(z') = - \varphi_{k}(-z')$, we have that $(C^{1,\gamma},80R,\theta)$-domain $\mathbb{R}^{n} \setminus \overline{U_{i_{k}}}$ satisfy
\begin{equation}\label{MMM7300}
(\mathbb{R}^{n} \setminus \overline{U_{i_{k}}}) \cap B_{R} 
 = \{ (z^{1},z') \in B_{R} : z^{1} > \bar{\varphi}_{k}(z')\}
\quad \text{and} \quad
|\bar{\varphi}_{k}(0')| < R,
\end{equation}
with the estimates
\begin{equation}\label{MMM7400}
D_{z'}\bar{\varphi}_{k}(0')=0,
\quad
\| D_{z'}\bar{\varphi}_{k} \|_{L^{\infty}(B_{R}')} 
\leq 1
\quad \text{and} \quad
[D_{z'}\bar{\varphi}_{k}]_{C^{\gamma}(B_{R}')} 
\leq 18 n\theta.
\end{equation}
Since $k \geq 1$, we have from \eqref{MMM6200} that $(C^{1,\gamma},80R,\theta)$-domains $\mathbb{R}^{n} \setminus \overline{U_{i_{k}}}$ and $U_{i_{k+1}}$ are disjoint. So we apply Lemma \ref{HHH1000} to $\mathbb{R}^{n} \setminus \overline{U_{i_{k}}}$, $U_{i_{k+1}}$, $\bar{\varphi}_{0}$ and $\bar{\varphi}_{1}$ instead of $U_{1}$, $U_{2}$, $\varphi_{1}$ and $\varphi_{2}$  with respect to  $z$-coordinate  by using \eqref{MMM7300} and \eqref{MMM7400}. Then there exists $C^{1,\gamma}$-function $\bar{\varphi}_{k+1} : B_{R}' \to \mathbb{R}$ such that 
\begin{equation*}
U_{	i_{k+1}} \cap B_{R} 
 = \{ (z^{1},z') \in B_{R} : z^{1} < \bar{\varphi}_{k+1}(z')\}
\end{equation*}
with the estimates
\begin{equation*}
\|D_{z'} \bar{\varphi}_{k} \|_{L^{\infty}(B_{R}')}, \|D_{z'} \bar{\varphi}_{k+1} \|_{L^{\infty}(B_{R}')} \leq  \tau
\end{equation*}
and
\begin{equation*}
[D_{z'} \bar{\varphi}_{k}]_{C^{\gamma}(B_{R}')},
[D_{z'} \bar{\varphi}_{k+1}]_{C^{\gamma}(B_{R}')} \leq  288n\theta.
\end{equation*}
Thus by letting $\varphi_{k+1}(y') = - \bar{\varphi}_{k+1}(-y')$, we obtain that
\begin{equation*}
U_{i_{k+1}} \cap B_{R} 
 = \{ (y^{1},y') \in B_{R} : y^{1} > \varphi_{k+1}(y')\}
\end{equation*}
with the estimates
\begin{equation*}
\|D_{y'} \varphi_{k} \|_{L^{\infty}(B_{R}')}, \|D_{y'} \varphi_{k+1} \|_{L^{\infty}(B_{R}')} \leq  \tau
\end{equation*}
and
\begin{equation*}
[D_{y'} \varphi_{k}]_{C^{\gamma}(B_{R}')},
[D_{y'} \varphi_{k+1}]_{C^{\gamma}(B_{R}')} \leq 288n\theta.
\end{equation*}
Repeat this process for $U_{i_{k+2}}, U_{i_{k+3}}, \cdots, U_{i_{l}}$ instead of $U_{i_{k+1}}$. Then one can find $C^{1,\gamma}$-functions $\varphi_{k}, \cdots, \varphi_{l} : B_{R}' \to \mathbb{R}$ such that for $d \in \{ k, \cdots, l \} $,
\begin{equation}\label{MMM7700}
U_{i_{d}} \cap B_{R} 
 = \{ (y^{1},y') \in B_{R} : y^{1} > \varphi_{d}(y')\}
\end{equation}
with the estimates
\begin{equation}\label{MMM7800}
\|D_{y'} \varphi_{d} \|_{L^{\infty}(B_{R}')} \leq  \tau
\qquad \text{and} \qquad
[D_{y'} \varphi_{d}]_{C^{\gamma}(B_{R}')} \leq 288n\theta.
\end{equation}
Recall from \eqref{MMM6200} and \eqref{MMM6400} that
\begin{equation*}
W_{i_{d}} \cap B_{R} = [U_{i_{d}} \setminus U_{i_{d+1}}] \cap B_{R}
\qquad \text{and} \qquad 
U_{i_{d}} \supsetneq U_{i_{d+1}}
\qquad \text{for} \qquad  
d \in \{ k, \cdots, l \} .
\end{equation*}
So by \eqref{MMM7700} and \eqref{MMM7800}, we find that \eqref{MMM5400} and \eqref{MMM5500} holds for $d \in \{ k, \cdots, l \}$ by taking $\varphi_{l+1} \equiv R$. Moreover, \eqref{MMM5600} follows from \eqref{MMM6900} and \eqref{MMM7200}. To finish [Step 2] for $k \geq 1$, it only remains to prove \eqref{MMM5400}  and \eqref{MMM5500} for $d \in \{ -m, \cdots, k-1 \}$.

\medskip

\noindent [Step 2-2 : $U_{i_{1}}, U_{i_{2}},\cdots, U_{i_{k-1}}$] In this step, we show that \eqref{MMM5400}, \eqref{MMM5500} for $d \in \{ 1, \cdots, k-1 \} $ and \eqref{MMM5600} hold when $k \geq 1$. 

If $k \leq 1$ then $\{ 1, \cdots, k-1 \}  = \emptyset$ and there is nothing to prove. So we assume that $k \geq 2$. We claim that
\begin{equation}\label{MMM8300}
\partial U_{i_{k-1}} \cap B_{R} \not = \emptyset.
\end{equation}
Suppose not. Then we have that $\partial U_{i_{k-1}} \cap B_{R} = \emptyset$. So with the fact that $U_{i_{k-1}} \supset U_{i_{k}} \owns \mathbf{0}$, we find from \eqref{MMM6560} that  $U_{i_{k-1}} \supset B_{R}$. This contradict \eqref{MMM6700} and the fact that $k \geq 2$. So the claim \eqref{MMM8300} holds.

Since $(C^{1,\gamma},80R,\theta)$-domains $\mathbb{R}^{n} \setminus \overline{U_{i_{k-1}}}$ and $U_{i_{k}}$ are disjoint, and so we apply Lemma \ref{HHH1000} to $U_{i_{k}}$, $\mathbb{R}^{n} \setminus \overline{U_{i_{k-1}}}$, $\varphi_{k}$ and $\varphi_{k-1}$ instead of $U_{1}$, $U_{2}$, $\varphi_{1}$ and $\varphi_{2}$ by using  \eqref{MMM7200} and \eqref{MMM7100}. Then there exists $C^{1,\gamma}$-function $\varphi_{k-1} : B_{R}' \to \mathbb{R}$ such that 
\begin{equation*}
(\mathbb{R}^{n} \setminus \overline{U_{i_{k-1}}}) \cap B_{R} 
 = \{ (y^{1},y') \in B_{R} : y^{1} < \varphi_{k-1}(y')\}
\end{equation*}
with the estimates
\begin{equation*}
\|D_{y'} \varphi_{k-1} \|_{L^{\infty}(B_{R}')} \leq  \tau
\qquad \text{and} \qquad
[D_{y'} \varphi_{k-1}]_{C^{\gamma}(B_{R}')} \leq  288n\theta,
\end{equation*}
which implies that
\begin{equation*}
U_{i_{k-1}} \cap B_{R} 
 = \{ (y^{1},y') \in B_{R} : y^{1} > \varphi_{k-1}(y') \}
\end{equation*}
Repeat the process for $U_{i_{k-2}}, U_{i_{k-3}}, \cdots, U_{i_{1}}$ instead of $U_{i_{k-1}}$. Then one can find $C^{1,\gamma}$-functions $\varphi_{k-1}, \cdots, \varphi_{1} : B_{R}' \to \mathbb{R}$ such that for $d \in \{ 1, \cdots, k-1 \} $,
\begin{equation}\label{MMM8700}
U_{i_{d}} \cap B_{R} 
 = \{ (y^{1},y') \in B_{R} : y^{1} > \varphi_{d}(y')\}
\end{equation}
with the estimates
\begin{equation}\label{MMM8800}
\|D_{y'} \varphi_{d} \|_{L^{\infty}(B_{R}')} \leq  \tau
\quad \text{and} \quad
[D_{y'} \varphi_{d}]_{C^{\gamma}(B_{R}')} \leq 288n\theta.
\end{equation}
Recall from \eqref{MMM6200} and \eqref{MMM6400} that
\begin{equation*}
W_{i_{d}} \cap B_{R} = [U_{i_{d}} \setminus U_{i_{d+1}}] \cap B_{R}
\qquad \text{and} \qquad 
U_{i_{d}} \supsetneq U_{i_{d+1}}
\qquad \text{for} \qquad  
d \in \{ 1, \cdots, k-1 \} .
\end{equation*}
By \eqref{MMM7100}, \eqref{MMM8700} and \eqref{MMM8800}, \eqref{MMM5400} and \eqref{MMM5500} hold for $d \in \{ 1, \cdots, k-1 \} $. 

\medskip 

\noindent [Step 2-3 : $U_{-1},\cdots, U_{-m}$] In this step, we show that \eqref{MMM5400} and \eqref{MMM5500} for $d \in \{ -m, \cdots, -1 \} $ when $k \geq 1$.

Since $k \geq 1$, we have that $(C^{1,\gamma},80R,\theta)$-domains $U_{i_{-1}}$ and $U_{i_{k}}$ are disjoint. If $\partial U_{i_{-1}} \cap B_{R} = \emptyset$, then we have from \eqref{MMM6560} and \eqref{MMM6700} that  $B_{R} \cap U_{i_{-1}} = \emptyset$, and so the proof for [Step 1-4] finished by taking $m=0$ and $\varphi_{-m} \equiv -R$. 

Next, assume that for some $m \geq 1$,
\begin{equation}\label{MMM9400}
\partial U_{i_{-1}} \cap B_{R}  \not = \emptyset , \cdots, \partial U_{i_{-m}} \cap B_{R} \not = \emptyset
\quad \text{and} \quad
\partial U_{i_{-m-1}} \cap B_{R} = \emptyset.
\end{equation}
Then by using \eqref{MMM7200} and \eqref{MMM7100}, we apply Lemma \ref{HHH1000} to $U_{i_{k}}$, $U_{i_{-1}}$, $\varphi_{k}$ and $\varphi_{0}$ instead of $U_{1}$, $U_{2}$, $\varphi_{1}$ and $\varphi_{2}$ to find that there exists $C^{1,\gamma}$-function $\varphi_{-k} : B_{R}' \to \mathbb{R}$ such that 
\begin{equation*}
U_{i_{-1}} \cap B_{R} 
 = \{ (y^{1},y') \in B_{R} : y^{1} < \varphi_{0}(y')\}
\end{equation*}
with the estimates
\begin{equation*}\|D_{y'} \varphi_{0} \|_{L^{\infty}(B_{R}')} \leq  \tau
\qquad \text{and} \qquad
[D_{y'} \varphi_{0}]_{C^{\gamma}(B_{R}')} \leq 288n\theta.
\end{equation*}
Repeat the process for $U_{i_{-2}}, U_{i_{-3}}, \cdots, U_{i_{-m}}$ instead of $U_{i_{-1}}$. Then one can find $C^{1,\gamma}$-functions $\varphi_{0}, \cdots, \varphi_{-m+1} : B_{R}' \to \mathbb{R}$ such that for $d \in \{ -m, \cdots, -1 \} $,
\begin{equation}\label{MMM9700}
U_{i_{d}} \cap B_{R} 
 = \{ (y^{1},y') \in B_{R} : y^{1} < \varphi_{d+1}(y')\}
\end{equation}
with the estimates
\begin{equation}\label{MMM9800}
\|D_{y'} \varphi_{d+1} \|_{L^{\infty}(B_{R}')} \leq  \tau
\quad \text{and} \quad
[D_{y'} \varphi_{d+1}]_{C^{\gamma}(B_{R}')} \leq 288n\theta.
\end{equation}
Recall from \eqref{MMM6200} and \eqref{MMM6500} that
\begin{equation*}
W_{i_{d}} \cap B_{R} = [U_{i_{d}} \setminus U_{i_{d-1}}] \cap B_{R}
\quad \text{ and } \quad
U_{i_{d}} \supsetneq U_{i_{d-1}}
\quad \text{ for } \quad
d \in \{ -m, \cdots, -1 \} .
\end{equation*}
So by  \eqref{MMM8700}, \eqref{MMM9700} and \eqref{MMM9800}, we find that \eqref{MMM5400} and \eqref{MMM5500} holds for $d \in \{ -m, \cdots, -1 \} $ by taking $\varphi_{-m} \equiv -R$.

\medskip

\noindent [Step 2-4 : $W_{i_{0}}$] In this step, we show that \eqref{MMM5400}, \eqref{MMM5500} for $d = 0$ and \eqref{MMM5600} hold when $k \geq 1$.

Recall from \eqref{MMM6200} and \eqref{MMM6300} that 
\begin{equation*}
W_{i_{0}} \cap B_{R} = [U_{i_{0}} \setminus (U_{i_{1}} \sqcup U_{i_{-1}} )] \cap B_{R}
\quad \text{and} \quad
U_{i_{1}}, U_{i_{-1}}
\subsetneq 
U_{i_{0}} .
\end{equation*}
So it follows from \eqref{MMM8700} and \eqref{MMM9700} that
\begin{equation*}\begin{aligned}
& \{ (y^{1},y') \in B_{R} : \varphi_{0}(y') < y^{1} < \varphi_{1}(y')  \} \\
& \qquad \subset
W_{i_{0}} \cap B_{R} \\
& \qquad \subset \{ (y^{1},y') \in B_{R} : \varphi_{0}(y') \leq y^{1} \leq \varphi_{1}(y')  \}.
\end{aligned}\end{equation*}
Thus \eqref{MMM5400} and \eqref{MMM5500}  for $d=0$ follows by taking $d= -1$ in  \eqref{MMM9800}. 

\medskip

By [Step 2-1] to [Step 2-3], we find that the lemma holds when $k \geq 1$.

\medskip

\noindent [Step 3 : $k=0$] Suppose that $k=0$. If $\partial U_{i_{1}} \cap B_{R} = \emptyset$ and $\partial U_{i_{-1}} \cap B_{R} = \emptyset$ then we have from \eqref{MMM6560} and \eqref{MMM6700} that $B_{R} \cap U_{i_{1}} = \emptyset$ and $B_{R} \cap U_{i_{-1}} = \emptyset$. So if follows from \eqref{MMM6300} that $W_{i_{0}} \cap B_{R} = U_{i_{0}} \cap B_{R} = U \cap B_{R} = B_{R}$, and the proof is finished by taking $l=m=0$ because of that $B_{R} \subset W_{0}$.

Next, suppose that one of $\partial U_{i_{1}} \cap B_{R} = \emptyset$ or $\partial U_{i_{-1}} \cap B_{R} = \emptyset$ holds. Without loss of generality, we assume that
\begin{equation}\label{NNN1900}
\partial U_{i_{-1}} \cap B_{R} \not = \emptyset.
\end{equation}
Then there exists $l \geq 0$ such that
\begin{equation}\label{NNN2000}
\partial U_{i_{-1}} \cap B_{R} \not = \emptyset, \partial U_{i_{1}} \cap B_{R} \not = \emptyset, \cdots, \partial U_{i_{l}} \cap B_{R} \not = \emptyset
\quad \text{and} \quad
\partial U_{i_{l+1}} \cap B_{R} = \emptyset.
\end{equation}
By comparing \eqref{MMM7250} and \eqref{NNN2000}, one can repeat the proof of [Step 2-1] for $\mathbb{R}^{n} \setminus \overline{U_{i_{-1}}},U_{i_{1}},\cdots, U_{i_{l+1}}$ instead of $U_{i_{k}},U_{i_{k+1}},\cdots, U_{i_{l+1}}$ to find that one can find $C^{1,\gamma}$-functions $\varphi_{0}, \cdots, \varphi_{l+1} : B_{R}' \to \mathbb{R}$ instead of $\varphi_{k}, \cdots, \varphi_{l+1} : B_{R}' \to \mathbb{R}$ satisfying \eqref{MMM5400}, \eqref{MMM5500} and \eqref{MMM5600} holds for $d \in \{ 0, \cdots, l \} $. 

In view of \eqref{NNN1900}, there exists $m \geq 1$ satisfying that
\begin{equation}\label{NNN2300}
\partial U_{i_{-1}} \cap B_{R}  \not = \emptyset , \cdots, \partial U_{i_{-m}} \cap B_{R} \not = \emptyset
\quad \text{and} \quad
\partial U_{i_{-m-1}} \cap B_{R} = \emptyset.
\end{equation}
By comparing \eqref{MMM9400} and \eqref{NNN2300}, one can repeat [Step 2-3] to find that one can find $C^{1,\gamma}$-functions $\varphi_{0}, \varphi_{-1}, \cdots, \varphi_{-m+1} : B_{R}' \to \mathbb{R}$ instead of $\varphi_{k}, \varphi_{-1},\cdots, \varphi_{-m+1} : B_{R}' \to \mathbb{R} $ satisfying \eqref{MMM5400}, \eqref{MMM5500} and \eqref{MMM5600} for $d \in \{ -m, \cdots, -1 \} $.

By [Step 3], the lemma holds when $k=0$. This completes the proof.
\end{proof}

Next, we handle the case that $\mathbf{0} \in \partial U$. For the convenience of the reader, we restate Theorem \ref{main theorem2}.

\maintheoremboundary*

\begin{proof}
For $R_{4}(n,\gamma,\theta,\tau) \in (0,1]$ in Lemma \ref{HHH1000}, we let $R_{0} = R_{4} \in (0,1]$. 

In view of Lemma \ref{MMM1000}, there exist $U_{i_{l+1}}, \cdots, U_{i_{0}} \in S$ $(l \geq 0)$ with
\begin{equation}\label{NNN6100}
U \cap B_{R} = ( W_{i_{l}} \sqcup \cdots \sqcup W_{i_{0}} ) \cap B_{R},
~
U_{i_{k}} \supsetneq U_{i_{k+1}} ~ \text{ for } ~  0 \leq k \leq l
~ \text{ and } ~ 
U_{i_{l+1}} = \emptyset
\end{equation}
satisfying 
\begin{equation}\label{NNN6200}
U \cap B_{R} 
= U_{i_{0}} \cap B_{R} 
\quad \text{and} \quad
U_{i_{k}} \cap B_{R}  = (W_{i_{k}} \sqcup U_{i_{k+1}}) \cap B_{R} 
\quad
(0 \leq k \leq l).
\end{equation}
In view of \eqref{NNN6100}, \eqref{NNN5300} holds.

One can easily check that
\begin{equation}\label{NNN6560}
\partial V \cap B_{R} = \emptyset
\qquad \Longrightarrow \qquad
B_{R} \subset V 
\qquad \text{or} \qquad
B_{R} \cap V = \emptyset,
\end{equation}

We claim that
\begin{equation}\label{NNN6600}
\mathbf{0} \in \partial U_{i_{0}}.
\end{equation}
One can check from \eqref{NNN6200} that for any $\rho \in (0,R]$,
\begin{equation*}
B_{\rho} \cap U_{i_{0}}^{c}
= (B_{\rho} \cap B_{R}^{c}) \cup (B_{\rho} \cap U_{i_{0}}^{c})
= B_{\rho} \cap (B_{R}^{c} \cup U_{i_{0}}^{c})
= B_{\rho} \cap (B_{R} \cap U_{i_{0}})^{c}
\end{equation*}
and
\begin{equation*}
B_{\rho} \cap (B_{R} \cap U_{i_{0}})^{c}
= B_{\rho} \cap (B_{R} \cap U)^{c}
= B_{\rho} \cap (B_{R}^{c} \cup U^{c})
= B_{\rho} \cap U^{c}.
\end{equation*}
So from the fact that $\mathbf{0} \in \partial U$, we have that for any $\rho \in (0,R]$,
\begin{equation*}
B_{\rho} \cap U_{i_{0}}
= B_{\rho} \cap (B_{R} \cap U_{i_{0}})
= B_{\rho} \cap (B_{R} \cap U) 
= B_{\rho} \cap U \not = \emptyset
\end{equation*}
and
\begin{equation*}
B_{\rho} \cap U_{i_{0}}^{c} = B_{\rho} \cap U^{c} 
\not = \emptyset.
\end{equation*}
Thus we find that the claim \eqref{NNN6600} holds.

Since $U_{i_{0}}$ is $(C^{1,\gamma},80R,\theta)$-domain, by applying Lemma \ref{CTS8000} with \eqref{NNN6600}, there exist an orthonormal matrix $V$ with $\det V>0$ and  $C^{1,\gamma}$-function $\varphi : B_{R}' \to \mathbb{R}$ such that
\begin{equation}\label{NNN7000}
U_{i_{0}} \cap B_{R} = \left \{ \sum\limits_{1 \leq k \leq n} y^{k} V_{k} \in B_{R} : y^{1} > \varphi_{0}(y') \right\}
\qquad \text{and} \qquad
|\varphi_{0}(0')| = 0,
\end{equation}
with the estimates
\begin{equation}\label{NNN7200}
D_{y'}\varphi_{0}(0')=0',
\quad
\| D_{y'}\varphi_{0} \|_{L^{\infty}(B_{R}')} 
\leq \tau
\quad \text{and} \quad
[D_{y'}\varphi_{0}]_{C^{\gamma}(B_{R}')} 
\leq 18 n\theta.
\end{equation}
So for $y$-coordinate system with the orthonormal basis $\{ V_{1}, \cdots, V_{n} \}$, we have that
\begin{equation}\label{NNN7100}
U_{i_{0}} \cap B_{R}  = \left \{ (y^{1},y') \in B_{R}  : y^{1} > \varphi_{0}(y') \right\}
\qquad \text{and} \qquad
|\varphi_{0}(0')| < R.
\end{equation}
Thus \eqref{NNN5600} holds. We next prove \eqref{NNN5400} and \eqref{NNN5500}.

Suppose that $\partial U_{i_{1}} \cap B_{R} = \emptyset$. Then by \eqref{NNN6560}, we have that $B_{R} \subset U_{i_{1}}$ or $B_{R} \cap U_{i_{1}} = \emptyset$. On the other-hand, from \eqref{NNN6100} and \eqref{NNN6200}, we have that $U_{i_{1}} \cap B_{R} \subsetneq U_{i_{0}} \cap B_{R} \subset B_{R}$. So we find that $B_{R} \cap U_{i_{1}} = \emptyset$. So with \eqref{NNN7100}, choose $l=1$ and $\varphi_{l} \equiv R$. This finishes the proof for when $\partial U_{i_{1}} \cap B_{R} = \emptyset$. 

Next assume that for some $l \geq 1$,
\begin{equation}\label{NNN7250}
U_{i_{1}} \cap B_{R}  \not = \emptyset , \cdots, U_{i_{l}} \cap B_{R} \not = \emptyset 
\qquad \text{ and } \qquad
U_{i_{l+1}} \cap B_{R} = \emptyset.
\end{equation}
Let $z$-coordinate system be the coordinate system with $z = -y$. Then by letting $\bar{\varphi}_{0}(z') = - \varphi_{0}(-z')$, we have that $(C^{1,\gamma},80R,\theta)$-domain $\mathbb{R}^{n} \setminus \overline{U_{i_{0}}}$ satisfy
\begin{equation}\label{NNN7300}
(\mathbb{R}^{n} \setminus \overline{U_{i_{0}}}) \cap B_{R} 
 = \{ (z^{1},z') \in B_{R} : z^{1} > \bar{\varphi}_{0}(z')\}
\quad \text{and} \quad
|\bar{\varphi}_{0}(0')| < R,
\end{equation}
with the estimates
\begin{equation}\label{NNN7400}
D_{z'}\bar{\varphi}_{0}(0')=0,
\quad
\| D_{z'}\bar{\varphi}_{0} \|_{L^{\infty}(B_{R}')} 
\leq 1
\quad \text{and} \quad
[D_{z'}\bar{\varphi}_{0}]_{C^{\gamma}(B_{R}')} 
\leq 18 n\theta.
\end{equation}
We have from \eqref{NNN6200} that $(C^{1,\gamma},80R,\theta)$-domains $\mathbb{R}^{n} \setminus \overline{U_{i_{0}}}$ and $U_{i_{1}}$ are disjoint. So we apply Lemma \ref{HHH1000} to $\mathbb{R}^{n} \setminus \overline{U_{i_{0}}}$, $U_{i_{1}}$, $\bar{\varphi}_{0}$ and $\bar{\varphi}_{1}$ instead of $U_{1}$, $U_{2}$, $\varphi_{1}$ and $\varphi_{2}$  with respect to  $z$-coordinate  by using \eqref{NNN7300} and \eqref{NNN7400}. Then there exists $C^{1,\gamma}$-function $\bar{\varphi}_{1} : B_{R}' \to \mathbb{R}$ such that 
\begin{equation*}
U_{	i_{1}} \cap B_{R} 
 = \{ (z^{1},z') \in B_{R} : z^{1} < \bar{\varphi}_{1}(z')\}
\end{equation*}
with the estimates
\begin{equation*}
\|D_{z'} \bar{\varphi}_{0} \|_{L^{\infty}(B_{R}')}, \|D_{z'} \bar{\varphi}_{1} \|_{L^{\infty}(B_{R}')} \leq  \tau
\end{equation*}
and
\begin{equation*}
[D_{z'} \bar{\varphi}_{0}]_{C^{\gamma}(B_{R}')},
[D_{z'} \bar{\varphi}_{1}]_{C^{\gamma}(B_{R}')} \leq  288n\theta.
\end{equation*}
Thus by letting $\varphi_{1}(y') = - \bar{\varphi}_{1}(-y')$, we obtain that
\begin{equation*}
U_{i_{1}} \cap B_{R} 
 = \{ (y^{1},y') \in B_{R} : y^{1} > \varphi_{1}(y')\}
\end{equation*}
with the estimates
\begin{equation*}
\|D_{y'} \varphi_{0} \|_{L^{\infty}(B_{R}')}, \|D_{y'} \varphi_{1} \|_{L^{\infty}(B_{R}')} \leq  \tau
\end{equation*}
and
\begin{equation*}
[D_{y'} \varphi_{0}]_{C^{\gamma}(B_{R}')},
[D_{y'} \varphi_{1}]_{C^{\gamma}(B_{R}')} \leq 288n\theta.
\end{equation*}
Repeat this process for $U_{i_{2}}, U_{i_{k+3}}, \cdots, U_{i_{l}}$ instead of $U_{i_{1}}$. Then one can find $C^{1,\gamma}$-functions $\varphi_{0}, \cdots, \varphi_{l} : B_{R}' \to \mathbb{R}$ such that for $d \in \{ 0, \cdots, l \} $,
\begin{equation}\label{NNN7700}
U_{i_{d}} \cap B_{R} 
 = \{ (y^{1},y') \in B_{R} : y^{1} > \varphi_{d}(y')\}
\end{equation}
with the estimates
\begin{equation}\label{NNN7800}
\|D_{y'} \varphi_{d} \|_{L^{\infty}(B_{R}')} \leq  \tau
\qquad \text{and} \qquad
[D_{y'} \varphi_{d}]_{C^{\gamma}(B_{R}')} \leq 288n\theta.
\end{equation}
Recall from \eqref{NNN6200} that
\begin{equation}\label{NNN7900}
W_{i_{d}} \cap B_{R} = [U_{i_{d}} \setminus U_{i_{d+1}}] \cap B_{R}
\quad \text{and} \quad 
U_{i_{d}} \supsetneq U_{i_{d+1}}
\quad \text{for} \quad  
d \in \{ 0, \cdots, l \} .
\end{equation}
So by \eqref{NNN7700}, \eqref{NNN7800} and \eqref{NNN7900}, we find that \eqref{NNN5400} and \eqref{NNN5500} holds for $d \in \{ 0, \cdots, l \} $ by taking $\varphi_{l+1} \equiv R$.
\end{proof}

\section*{Acknowledgements}
Y. Kim was supported by the National Research Foundation of Korea (NRF) grant funded by the Korea Government NRF-2020R1C1C1A01013363.
P. Shin was supported by Basic Science Research Program through the National Research Foundation of Korea(NRF) funded by the Ministry of Education(No. NRF-2020R1I1A1A01066850).

\bibliographystyle{amsplain}

\end{document}